\documentclass[12pt]{article}

\usepackage{amsmath,amsfonts,amssymb}
\usepackage{bbm}

\usepackage{lipsum}

\newtheorem{defn}{Definition}[section]
\newtheorem{theo}[defn]{Theorem}
\newtheorem{lem}[defn]{Lemma}
\newtheorem{prop}[defn]{Proposition}

\newtheorem{rem}[defn]{Remark}
\newtheorem{exam}[defn]{Example}

\newenvironment{proof}{{\bf Proof }}{{\vskip 0.1cm \hfill$\Box$}}

\def\N {{\mathbb N}}

\def\R {{\mathbb R}}
\def\E{{\mathbb E}}
\def\P{{\mathbb P}}
\def\M{{\mathbb M}}
\newcommand{\F}{\mathcal{F}}

\begin{document}

\noindent
{{\Large\bf Well-posedness for a class of degenerate It\^o-SDEs with fully discontinuous coefficients}
{\footnote{This research was supported by Basic Science Research Program through the National Research Foundation of Korea(NRF) funded by the Ministry of Education(NRF-2017R1D1A1B03035632).}}\\ \\
\bigskip
\noindent
{\bf Haesung Lee},
{\bf Gerald Trutnau}
}\\

\noindent
{\bf Abstract.} We show uniqueness in law for a  general class of stochastic differential equations in $\mathbb{R}^d$, $d\ge 2$, with possibly degenerate and/or fully discontinuous locally bounded coefficients among all weak solutions that spend zero time at the points of degeneracy of the dispersion matrix. The points of degeneracy have $d$-dimensional Lebesgue-Borel measure zero. Weak existence is obtained for more general, not necessarily locally bounded drift coefficient.
\iffalse
For fairly general second order elliptic partial differential operators $L$ in $\mathbb{R}^d, d\ge 2$, with possibly degenerate diffusion matrix and locally integrable first order terms, we construct a pre-invariant measure $\mu$.
This then allows at the same time for a functional analytic frame which describes $L$ as infinitesimal generator of a $C_0$-semigroup in $L^p(\mathbb{R}^d,\mu)$, $p\ge 1$, and to derive through elliptic and parabolic regularity theory strong Feller properties for the semigroup and resolvent. Subsequently, we show that the continuous version of the semigroup that we obtained through elliptic and parabolic regularity, is the transition function of a Hunt process that has continuous sample paths on the one point compactification of $\mathbb{R}^d$ with some point at infinity $\Delta$.  This Hunt process then weakly solves a stochastic differential equation (SDE) with possibly degenerate discontinuous dispersion and locally integrable drift coefficient up to some explosion time. Finally, we present conditions for the explosion time to be infinite, i.e. conditions for non-explosion,  and conditions for uniqueness in law of the solution to the SDE. Our new approach allows to  consider well-posedness for not necessarily locally uniformly strictly elliptic SDEs with fully discontinuous dispersion and drift coefficients.
\fi

\section{Introduction}
The question whether a solution to a stochastic differential equation (hereafter SDE) on $\R^d$ exists that is pathwise unique and strong occurs widely in the mathematical literature, see for instance introduction of \cite{LT18} for some detailed, but possibly incomplete recent development. Sometimes, strong solutions, which are roughly described as weak solutions for a given Brownian motion are required, for instance in signal processing where a noisy signal is implicitly given.
Sometimes, it may be impossible to obtain a strong solution or only weak solutions are important to consider, or for some reason only the strong Markov property of the solution is needed. Then uniqueness in law, i.e. the question whether given an initial distribution,  the distribution of any weak solution no matter on which probability space it is considered is the same, plays an important role. Also, it might be that pathwise uniqueness and strong solution results are just too restrictive, so that one is naturally led to consider weak solutions and their uniqueness. (Here we consider weak uniqueness of an SDE with respect to all initial conditions $x\in\R^d$ as defined for instance in \cite[Chapter 5]{KaSh}, see also Definition \ref{wellposedness} below). \\
In order to explain our motivation for this work, fix a symmetric matrix $C=(c_{ij})_{1 \leq i,j \leq d}$ of bounded measurable functions $c_{ij}$, such that for some $\lambda \ge 1$
\begin{equation*}
\lambda^{-1} \|\xi\|^2 \leq \langle C(x) \xi, \xi \rangle \leq \lambda \|\xi\|^2, \quad \text{ for all } x,\xi \in \mathbb{R}^d,
\end{equation*}
and a vector $\mathbf{H}=(h_1,..., h_d)$ of locally bounded measurable functions. Let 
\begin{equation}\label{weaksolutiongenerator}
{\mathcal L}f=\sum_{i,j=1}^d\frac{c_{ij}}{2}\partial_{ij}f+\sum_{i=1}^d h_i\partial_i f
\end{equation}
be the corresponding linear operator and 
\begin{equation}\label{weaksolutionintrofirst}
X_t = x+  \int_0^t \sqrt{C}(X_s) \, dW_s+   \int^{t}_{0}   \mathbf{H}(X_s) \, ds, \quad t\ge 0, \ x\in \R^d,
\end{equation}
be the corresponding It\^o-SDE. If the $c_{ij}$ are continuous and the $h_i$ bounded then \eqref{weaksolutionintrofirst} is well-posed, i.e. there exists a solution and it is unique in law (see \cite{StrVar}). If the $h_i$ are bounded, then \eqref{weaksolutionintrofirst} is well-posed for $d=2$ (see \cite[Exercise 7.3.4]{StrVar}), but if $d\ge 3$, there exists an example of a measurable discontinuous $C$ for which uniqueness in law does not hold (\cite{Nadi}). Hence even in the nondegenerate case, well-posedness for discontinuous coefficients is non-trivial and one is naturally led to search for general subclasses in which well-posedness holds. Some of these are given when $C$ is not far from being continuous, i.e. continuous up to a small set (e.g. a discrete set or a set of $\alpha$-Hausdorff measure zero with sufficiently small $\alpha$, or else, see for instance introductions of \cite{Kry04} and \cite{Nadi}  for references). Another special subclass is given when $C$ is piecewise constant on a decomposition of $\R^d$ into a finite union of polyhedrons (\cite{BaPa}). Actually, the  work \cite{BaPa} is one of our sources of motivation for this article. Though we do not perfectly cover the conditions in \cite{BaPa}, we can complement them in many ways. In particular, we can consider arbitrary decompositions of $\R^d$ into bounded disjoint measurable sets (choose for instance $\sqrt{\frac{1}{\psi}}=\sum_{i=1}^\infty \alpha_i 1_{A_i}$, with $\R^d=\dot{\cup}_{i=1}^{\infty}A_i$, $(\alpha_i)_{i\in \N}\subset (0,\infty)$ in \eqref{weaksolutionintro} below).  A further example for a discontinuous $C$, where well-posedness holds can be found in \cite{Gao}. There the discontinuity is along the common boundary of the upper and lower half spaces.
In \cite{Kry04}, among others, the problem of uniqueness in law for \eqref{weaksolutionintrofirst} is related to the Dirichlet problem for ${\mathcal L}$ as in \eqref{weaksolutiongenerator}, locally on smooth domains. This method, was also used in \cite{Nadi}, using previous work of Krylov. In particular, a shorter proof of the well-posedness results of Bass and Pardoux \cite{BaPa} and Gao \cite{Gao} is presented in \cite[Theorems 2.16 and 3.11]{Kry04}. But the most remarkable is the derivation of well-posedness for a special subclass of processes with degenerate discontinuous $C$. Though the discontinuity is only along a hyperplane of codimension one and the coefficients are quite regular outside the hyperplane, it seems to be one of the first examples of a discontinuous degenerate $C$ where well-posedness still holds (\cite[Example 1.1]{Kry04}).  
This intriguing example was the second source of our motivation. As it was the case for the results in \cite{BaPa}, we could not perfectly cover \cite[Example 1.1]{Kry04}, but again we could complement it in many ways.  As a main observation besides the above considerations, it seems that so far no general subclass was presented, where $C$ is degenerate (or also nondegenerate if $d\ge 3$) and fully discontinuous, but nonetheless well-posedness holds. This will be another main goal of this paper and our method strongly differs from the techniques used in  \cite{BaPa} and \cite{Kry04} and in previous literature. Our techniques involve semigroup theory, elliptic and parabolic regularity theory, the theory of generalized Dirichlet forms (i.e. the construction of a Hunt process from a sub-Markovian $C_0$-semigroup of contractions on some $L^1$-space with a weight) and  an adaptation of an idea of Stroock and Varadhan to show uniqueness  for the martingale problem using a Krylov type estimate. Krylov type estimates have been widely used to obtain the weak solution and its uniqueness, in particular pathwise uniqueness, simultaneously. The advantage of our method is that 
weak existence of a solution and uniqueness in law are shown separately of each other using different techniques. We use local Krylov type estimates (Theorem \ref{krylovtype}) to show uniqueness in law and once uniqueness in law holds we can improve the original Krylov estimate at least for the time-homogeneous case (see Remark \ref{remuniquenessinlawconsequences}). In particular our method typically implies weak existence results that are more general than the uniqueness results (see Theorem \ref{weakexistence4} here and \cite{LT18,LT19}). \\
%{\bf Other examples of discontinuous degenerate $C$ ??}
%But our method of proof completely differs from the one in \cite{BaPa}.\\
%in \cite{Pri}.
%Finally, we mention three recent other papers other papers uniqueness paper
%A further paper, which deals with general results concerning the mentioned subclasses
%\cite{Tre}
%\cite{ChJa}
Now, let us describe our results. Let $d\ge 2$, and $A=(a_{ij})_{1 \leq i,j \leq d}$ be a symmetric matrix of functions $a_{ij} \in H_{loc}^{1,2d+2}(\mathbb{R}^d) \cap C(\mathbb{R}^d)$ such that for every open ball $B\subset \R^d$, there exist constants $\lambda_B, \Lambda_B>0$ with
\begin{equation*}
\lambda_{B} \|\xi\|^2 \leq \langle A(x) \xi, \xi \rangle \leq \Lambda_B \|\xi\|^2, \quad \text{ for all } \xi \in \mathbb{R}^d, \; x \in B.
\end{equation*}
Let $\psi \in L_{loc}^q(\mathbb{R}^d)$, with $q>2d+2$, $\psi>0$ a.e., such that $\frac{1}{\psi} \in L_{loc}^{\infty}(\mathbb{R}^d)$. 
Here we assume that the expression $\frac{1}{\psi}$ stands for an arbitrary but fixed Borel measurable function satisfying $\psi\cdot \frac{1}{\psi}=1$ a.e. and  $\frac{1}{\psi}(x)\in [0,\infty)$ for any $x\in \R^d$. Let
$\mathbf{G}=(g_1,...,g_d)\in  L_{loc}^{\infty}(\mathbb{R}^d, \mathbb{R}^d)$ be a vector of Borel measurable functions. 
Let $(\sigma_{ij})_{1 \le i \le d,1\le j \le m}$, $m\in \mathbb{N}$ arbitrary but fixed, be any matrix consisting of continuous functions,  such that $A=\sigma \sigma^T$.
Suppose there exists a constant  $M> 0$, such that 
\begin{eqnarray}\label{consuniquelaw}
-\frac{\langle A(x)x, x \rangle}{ \psi(x)(\left \| x \right \|^2 +1)}+ \frac{\mathrm{trace}( A(x))}{2\psi(x)}+ \big \langle \mathbf{G}(x), x \big \rangle  \leq M\left ( \left \| x \right \|^2+1\right )\left (  {\rm ln}(\left \| x \right \|^2+1)+1\right )  
\end{eqnarray}
for a.e. $x\in \mathbb{R}^d$.
The main result of our paper (Theorem \ref{poijoijweuniqe}) is that then weak existence and uniqueness in law, i.e. well-posedness, holds for the stochastic differential equation
\begin{equation}\label{weaksolutionintro}
X_t = x+  \int_0^t \big (\sqrt{\frac{1}{\psi}}\cdot \sigma\big )(X_s) \, dW_s+   \int^{t}_{0}   \mathbf{G}(X_s) \, ds, \quad t\ge 0, \ x\in \R^d.
\end{equation}
among all weak solutions  $(\Omega, \mathcal{F}, (\mathcal{F}_t)_{t\ge 0}, X_t=(X_t^1,...,X_t^d), W = (W^1,\dots,W^m),\mathbb{P}_x )$, $x\in \mathbb{R}^d$, such that 
\begin{equation}\label{psiconditionintro}
%\mathbb{P}_x(dt (\{ s\ge 0\,| \frac{1}{\psi}(X_s)\psi (X_s)\not= 1\})    =0))=1,
\int_0^{\infty} 1_{\big \{ \sqrt{ \frac{1}{\psi} }=0\big \}}(X_s)ds=0\qquad \P_x\text{-a.s. } \ \forall x\in \R^d.
\end{equation}
Here it is important to mention that the solution and the integrals involving the solution in \eqref{weaksolutionintro} may a priori depend on the Borel versions chosen for $\sqrt{\frac{1}{\psi}}$ and $\mathbf{G}$ but that condition \eqref{psiconditionintro} is exactly the condition that makes these objects independent of the chosen Borel versions (cf. Lemma \ref{condforpsitohold}). $\sqrt{ \frac{1}{\psi} }$ may of course be fully discontinuous but if it takes all its values in $(0,\infty)$, then \eqref{psiconditionintro} is automatically satisfied.
However, since $\psi \in L_{loc}^q(\mathbb{R}^d)$, it must be a.e. finite, so that the zeros $Z$ of $\sqrt{\frac{1}{\psi}}$ have Lebesgue-Borel measure zero. Nonetheless, our main result comprehends the existence of a whole class of degenerate (on $Z$) diffusions with fully discontinuous coefficients for which well-posedness holds. This seems to be new in the literature. For another condition that implies 
\eqref{psiconditionintro}, we refer to Lemma \ref{condforpsitohold}, and for an explicit example for well-posedness, which reminds the Engelbert/Schmidt condition for uniqueness in law in dimension one (see \cite{EngSch}), we refer to Example \ref{wellposednessexplicitexample}. \\
We derive weak existence of a solution to \eqref{weaksolutionintro} up to its explosion time under quite more general conditions on the coefficients, see Theorem \ref{weakexistence4}.  In this case, for non-explosion one only needs that \eqref{consuniquelaw} holds outside an arbitrarily large open ball (see Remark \ref{explanationweaksolution}(ii)). Moreover, \eqref{psiconditionintro} is always satisfied for the weak solution that we construct (see Remark \ref{explanationweaksolution}) and our weak solution is a Hunt process, not only a strong Markov process.\\
The techniques that we use for weak existence are as follows. First, any solution to \eqref{weaksolutionintro} determines the same (up to a.e. uniqueness of the coefficients) second order partial differential operator $L$ on $C_0^{\infty}(\R^d)$. In Theorem \ref{helholmop}, we find a measure $\mu :=\rho\psi\,dx$, with some nice regularity of  $\rho$, that is pre-invariant for $L$, i.e.
$$
\int_{\R^d} Lf \,d\mu =    \int_{\R^d}  \big (\sum_{i,j=1}^d \frac{\frac{1}{\psi} a_{ij}}{2}\partial_{ij}f+\sum_{i=1}^d g_i\partial_i f\big ) \,d\mu =   0,\qquad \forall f\in C_0^{\infty}(\R^d).
$$
Then, using the existence of the pre-invariant density, we adapt the method from Stannat (cf. \cite{St99}) to our case and construct a sub-Markovian $C_0$-semigroup of contractions $(T_t)_{t\ge 0}$ on each $L^s(\R^d,\mu)$, $s\ge 1$, whose generator extends $(L,C_0^{\infty}(\R^d))$, i.e. we have found a suitable functional analytic frame (see Theorem \ref{mainijcie}, which further induces a generalized Dirichlet form, see \eqref{fwpeokce}) to describe a potential infinitesimal generator of a weak solution to \eqref{weaksolutionintro}. This is done in Section \ref{4.1}, where we also derive with the help of the results about general regularity properties from Section \ref{section 3}, the regularity properties of $(T_t)_{t\ge 0}$ and its resolvent (see Section \ref{fewfore}). Then, using crucially the  existence of a Hunt process for a.e. starting point related to $(T_t)_{t\ge 0}$ in Proposition \ref{Huntex} (which follows similarly to \cite[Theorem 6]{Tr5}) this leads to a transition function of a Hunt process that not only weakly solves \eqref{weaksolutionintro}, but moreover has a transition function with such a nice regularity that many presumably optimal classical conditions for properties of a solution to \eqref{weaksolutionintro} carry over to our situation. We mention, for instance, the non-explosion condition \eqref{consuniquelaw} and moment inequalities (see Remark \ref{mostresults4}). But also irreducibility and classical ergodic properties, as in \cite{LT18}, could be studied in this framework by further investigating the influence of $\frac{1}{\psi}$ on properties of the transition function. Similarly to the results of \cite{LT18}, the only point where Krylov type estimates are used in our method, is when it comes up to uniqueness. Here, because of the possible degeneracy of  $\sqrt{\frac{1}{\psi}}$, we need the condition \eqref{psiconditionintro} to derive a Krylov type estimate that holds for any weak solution to \eqref{weaksolutionintro} (see Theorem \ref{krylovtype} which follows straightforwardly from the original Krylov estimate \cite[2. Theorem (2), p. 52]{Kry})). Again, our constructed transition function has such a nice regularity that a time dependent drift-eliminating It\^o-formula holds for the function $g(x,t):=P_{T-t}f(x)$, $f\in C_0^{\infty}(\R^d)$. In fact, it holds for any weak solution to \eqref{weaksolutionintro}, so that for all these the one dimensional, hence all finite dimensional, marginals coincide (cf. Theorem \ref{weifjowej}). This latter technique goes back to an idea of Stroock/Varadhan (\cite{StrVar}) and we use the treatise of this technique as presented in \cite[Chapter 5]{KaSh}.

\section{Notations}
Throughout, we will use the same notations as in \cite{LT18,LT19}. \\
Additionally, for an open set $U$ in $\R^d$ and a measure $\mu$ on $\R^d$, let $L^q(U,\R^d, \mu):=\{ \mathbf{F}=(f_1,...,f_d):U \to\R^d\, \mid \, f_i\in L^q(U, \mu), 1\le i\le d\}$, equipped with the norm, $\| \mathbf{F} \|_{L^q(U, \R^d,\mu)}  := \| \|\mathbf{F}\| \|_{L^q(U, \mu)}, \,  \mathbf{F} \in L^q(U, \R^d, \mu)$. If $\mu=dx$, we  write $L^q(U)$, $L^q(U, \R^d)$ for $L^q(U, dx)$, $L^q(U, \R^d, dx)$ respectively and  even $\| \mathbf{F} \|_{L^q(U)}$ for $\| \mathbf {F} \|_{L^q(U, \R^d)}$. 
If $I$ is an open interval $\R$ and $p, q \in [1, \infty]$, we denote by $L^{p,q}(U \times I)$ the space of all Borel measurable functions $f$ on $U \times I$ for which 
$$
\|f\|_{L^{p,q}(U \times I)}:=\| \|f(\cdot,  \cdot )\|_{L^{p}(U)} \|_{L^{q}(I)}< \infty,
$$
and let $\text{supp}(f):= \text{supp}(|f|dxdt)$. For a locally integrable function $g$ on $U \times I$ and  $i \in \{1, \dots d  \}$, we denote by $\partial_i g$ the $i$-th weak spatial derivative on $U \times I$ and by $\partial_t g$ the weak time derivative on  $U \times I$, provided it exists. For $p,q  \in [1, \infty]$, let $W^{2,1}_{p,q}( U \times I)$ be the set of all locally integrable functions
$g: U \times I \rightarrow \R$ such that $\partial_t g,\, \partial_i g,\, \partial_i \partial_j g \in L^{p,q}( U \times I)$ for all $1 \leq i,j \leq d$. Let $W^{2,1}_{p}( U \times I ):= W^{2,1}_{p,p}( U \times I)$.
\section{Regularity results}\label{section 3}
\subsection{Regularity of linear parabolic equations with weight in the time derivative term} \label{7.1sections}
The following lemma is a slight modification of \cite[Lemma 6]{ArSe} and involves a weight function $\psi$.
\begin{lem} \label{feeer}
Let $U$ be a bounded open subset of $\R^d$ and $T>0$. Let $w \in L^{2}(U \times (0,T))$ be such that $\text{\rm supp}(w) \subset U \times (0,T]$ and assume $\partial_t w \in L^2(U \times (0,T))$, $\psi \in L^2(U)$. Then for a.e. $\tau \in (0,T)$, it holds
$$
 \int_0^{\tau} \int_{U} \partial_t w \cdot \psi \,dx dt  = \int_U w|_{t=\tau}\, \psi dx.
$$
\end{lem}
\begin{proof}
Let $\psi_n \in C_0^{\infty}(U)$, $n \geq 1$, satisfy $\lim_{n \rightarrow \infty }\psi_n = \psi$ in $L^2(U)$. Then $w \psi \in L^{1,2}(U \times (0,T))$ and for any $\varphi \in C_0^{\infty}(U \times (0,T))$, we have 
\begin{eqnarray*}
\iint_{U \times (0,T)} \partial_t \varphi  \cdot  w \psi\, dx dt &=& \lim_{n \rightarrow \infty} \iint_{U \times (0,T)} \partial_t \varphi  \cdot  w \psi_n \, dx dt  \\
&=& \lim_{n \rightarrow \infty} \iint_{U \times (0,T)} \partial_t (\varphi \psi_n) \cdot  w \, dx dt \\
&=& -\lim_{n \rightarrow \infty} \iint_{U \times (0,T)} \varphi \psi_n \cdot  \partial_t w \, dx dt \\
&=&  -\iint_{U \times (0,T)} \varphi \cdot (\partial_t w \cdot \psi )dx dt.
\end{eqnarray*}
Thus \,$\partial_t(w \psi) = \partial_t w \cdot \psi \in L^{1,2}(U \times (0,T))$. Now let $f(t):= \int_{U} w(x,t)\psi(x) dx$. Then $f(t)$ is defined for a.e. $t \in (0,T)$ and is in $L^1((0,T))$. Let $g \in C_0^{\infty}\big((0,T)\big)$ be given. Take $\tau_0 \in (0,T)$ satisfying $\text{supp}(g) \subset (0, \tau_0)$. Let $V$ be a bounded open subset of $\R^d$ such that $\overline{V} \subset U$ and $\text{supp}(w) \cap \big(U \times (0, \tau_0)\big) \subset V \times (0, \tau_0)$. Let $\chi \in C_0^{\infty}(U)$ with $\chi \equiv 1$ on $V$. Then
\begin{eqnarray*}
\int_0^T \partial_t g \cdot f \, dt  &=&  \iint_{U \times (0,\tau_0 )} \partial_t g \cdot w \, \psi dxdt \\
&=& \iint_{V \times (0, \tau_0)} \partial_t(g \chi) \cdot (w \psi) dx dt \\
&=& -\iint_{U \times (0, T)} g \chi \,\partial_t w \cdot \psi dx dt \\
&=& -\int_0^{T} g \cdot \left( \int_U \partial_t w \cdot \psi dx  \right) dt.
\end{eqnarray*}
Thus $\partial_t f = \int_U \partial_t w \cdot \psi dx \in L^1\big((0,T)\big)$. Then by \cite[Theorem 4.20]{EG15}, $f$ has an absolutely continuous $dx$-version on $(0,T)$ and by the Fundamental Theorem of Calculus, for a.e $\tau_1, \tau \in (0,T)$ it holds
$$
\int_{\tau_1}^{\tau} \int_U \partial_t w\cdot \psi dx dt = \int_{\tau_1}^{\tau} \partial_t f dt  = \int_{\tau_1}^{\tau} f' dt = f(\tau)-f(\tau_1)= \int_U (w|_{t=\tau}- w|_{t=\tau_1})\, \psi dx.
$$
Choosing $\tau_1$ near $0$, our assertion follows. 
\end{proof}

Throughout this section, we assume the following condition.
\begin{itemize}
\item[\bf (I)]
$U \times (0,T)$ is a bounded open set in $\R^d\times \R$, $T>0$.  $A=(a_{ij})_{1 \leq i,j \leq d}$ is a matrix of functions on $U$ that is uniformly strictly elliptic and bounded, i.e. there  exists constants $\lambda>0$, $M>0$, such that for all $\xi = (\xi_1, \dots, \xi_d) \in \R^d$, $x \in U$, it holds
$$
\quad \sum_{i,j=1}^{d} a_{ij}(x) \xi_i \xi_j \geq \lambda\|\xi\|^2, \qquad \max_{1 \leq i, j \leq d} | a_{ij}(x)   | \leq M, 
$$
and let $\mathbf{B} \in L^p(U,\R^d)$ with $p>d$, \,$\psi \in L^q(U)$, $q \in [2 \vee \frac{p}{2}, p )$. There exists $c_0>0$ such that $c_0 \leq \psi$ on $U$, and finally
$$
u \in H^{1,2}(U \times (0,T) )\cap L^\infty(U \times (0,T)).
$$ 
\end{itemize}
\centerline{}
Assuming {\bf(I)}, we consider a divergence form linear parabolic equation with a singular weight in the time derivative term  as follows
\begin{eqnarray} 
\;\; \iint_{U \times (0,T)} (u \partial_t\varphi) \psi dx dt = \iint_{U \times (0,T)}   \big \langle A\nabla u, \nabla \varphi  \big \rangle +\langle \mathbf{B}, \nabla u \rangle \varphi \,dx dt, \quad \nonumber \\
\text{ for all }\varphi \in C_0^{\infty}(U \times (0,T)). \quad \label{baseq}
\end{eqnarray}
Using integration by parts in the left hand term, \eqref{baseq} is equivalent to
\begin{eqnarray}  
\;\; -\iint_{U \times (0,T)} ( \partial_t u)\, \varphi \psi dx dt = \iint_{U \times (0,T)}   \big \langle A\nabla u, \nabla \varphi  \big \rangle +\langle \mathbf{B}, \nabla u \rangle \varphi dx dt, \nonumber \\
\text{ for all }\varphi \in C_0^{\infty}(U \times (0,T)). \label{baseq23}
\end{eqnarray}
Define $\mathcal{A}:= \{ v \in L^{\infty}(U \times (0,T))  \mid  \nabla v \in L^{2}(U \times (0,T)) \, \text{ and }\, \text{supp}(v) \subset U \times (0,T)  \}$. Using the standard mollification on $\R^d \times \R$ to approximate functions in $\mathcal{A}$, \eqref{baseq23} extends to
\begin{eqnarray}  
\;\; -\iint_{U \times (0,T)} ( \partial_t u)\, \varphi \psi dx dt = \iint_{U \times (0,T)}   \big \langle A\nabla u, \nabla \varphi  \big \rangle +\langle \mathbf{B}, \nabla u \rangle \varphi dx dt, \nonumber \\
\text{ for all }\varphi \in \mathcal{A}. \label{baseq24}
\end{eqnarray}
Fix $\beta \geq 1$. For $t \in \R$, define functions $G(t):=(t^+)^{\beta}$, $H(t):= \frac{1}{\beta+1} (t^{+})^{\beta+1}$, where $t^+:= \max(0,t)$.  Then by \cite[Theorem 4.4]{EG15},  $G'(t)=\beta (t^+)^{\beta-1} 1_{[0, \infty)}(t)$ and $H'(t)=G(t)$.
Let $\eta \in C_0^{\infty}(U \times (0, T])$ with $\eta \geq 0$. \;Given $\tau \in (0,T)$, define $\widetilde{\varphi}:= \eta^2 G(u) 1_{(0,\tau)}$. Then by \cite[Theorem 4.4]{EG15} (or \cite[Lemma 4]{ArSe}),
$$
\nabla \widetilde{\varphi}=
\left\{\begin{matrix}
\eta^2 G'(u) \nabla u +2 \eta \nabla \eta\, G(u),& \quad 0<t<\tau, \\[5pt] 
0,& \quad \tau \leq t < T.
\end{matrix}\right.
$$
Thus $\widetilde{\varphi} \in  \mathcal{A}$ and by \eqref{baseq24}, we have
\begin{equation} \label{baseq2}
\;\; -\iint_{U \times (0,T)} \left(\partial_t u \right)  \widetilde{\varphi} \psi dx dt = \iint_{U \times (0,T)}   \big \langle A\nabla u, \nabla \widetilde{\varphi}  \big \rangle +\langle \mathbf{B}, \nabla u \rangle \widetilde{\varphi} dx dt.
\end{equation}
Observe that by \cite[Theorem 4.4]{EG15} (or \cite[Lemma 4]{ArSe}),
$$
\partial_t (\eta^2 H(u)  ) = 2\eta  \partial_t \eta \, H(u)+\eta^2 G(u) \partial_t u.
$$
Thus by Lemma \ref{feeer}
\begin{eqnarray}
&&\iint_{U \times (0,T)} \widetilde{\varphi} \,(\partial_t u) \,\psi dx dt   \nonumber \\
&=& \iint_{U \times (0,\tau)} \eta^2 G(u) \partial_t u \cdot \psi dx dt \nonumber \\
&=& \int_0^{\tau} \int_U  \partial_t (\eta^2 H(u)  ) \psi dx dt-   2 \int_0^{\tau} \int_U \eta  \partial_t \eta \, H(u) \psi dx dt \nonumber \\
 &=& \int_U \eta^2 H(u)\mid_{t=\tau} \psi dx - \int_0^{\tau} \int_{U} 2 \eta \partial_t \eta \,H(u)\psi dx dt, \;\text{ for a.e. } \tau \in (0,T).   \qquad \label{loijwdqpp3}
\end{eqnarray}
By \eqref{baseq2} and \eqref{loijwdqpp3}, we get
\begin{eqnarray}
&& \int_U \eta^2 H(u) \mid_{t= \tau} \psi dx dt+ \int_{0}^{\tau} \int_{U}   \big \langle A\nabla u, \nabla \widetilde{\varphi}  \big \rangle +\langle \mathbf{B}, \nabla u \rangle \widetilde{\varphi} dx dt \nonumber \\
&&\qquad \;\;= \int_0^{\tau} \int_{U} 2 \eta\, \partial_t \eta \,H(u) \, \psi dx dt, \qquad \text{ for a.e. } \tau \in (0,T). \label{maindivest}
\end{eqnarray}
On $\{ \widetilde{\varphi}>0  \}$, it holds $u>0$, so that $\nabla u = \nabla u^+$. Thus on $\{ \widetilde{\varphi}>0  \}$, we have
\begin{eqnarray*}
&&\, \quad \big \langle A\nabla u, \nabla \widetilde{\varphi}  \big \rangle +\langle \mathbf{B}, \nabla u \rangle \widetilde{\varphi} \\ 
&&= \big \langle A\nabla u^+,\eta^2 G'(u) \nabla u^+ \rangle  +\big \langle A\nabla u^+, 2 \eta \nabla \eta\, G(u) \big \rangle  +\langle \mathbf{B}, \nabla u^+ \rangle \eta^2 G(u) \\
&&\geq \eta^2 G'(u) \lambda \,\|\nabla u^+\|^2-2 \eta G(u) dM  \| \nabla \eta \| \|\nabla u^+  \| -\eta^2 G(u) \|\mathbf{B}\|\| \nabla u^+ \|. \qquad
\end{eqnarray*}
Moreover on $\{ \widetilde{\varphi}>0 \}$
\begin{eqnarray*}
(u^+)^{-\beta-1}\,G(u)^2 \leq \,G'(u).
\end{eqnarray*}
Hence using Young's inequality, we obtain
\begin{eqnarray*}
&&\quad 2 \eta G(u)d M  \| \nabla \eta \| \|\nabla u^+  \|\\
&& \leq 2\cdot \frac14 \frac{\Big(\sqrt{\lambda}\, (u^+)^{-\frac{\beta+1}{2}}  G(u) \,\eta \,\|\nabla u ^+\|\Big)^2}{2}+ 2 \cdot 4 \frac{\Big( dM \sqrt{\lambda^{-1}}\, (u^+)^{\frac{\beta+1}{2}} \|\nabla \eta \| \Big)^2}{2} \\
&&=\frac{\lambda}{4} \eta^2 G'(u) \|\nabla u \|^2  + \frac{4d^2M^2}{\lambda} \|\nabla \eta \|^2\, (u^{+})^{\beta+1},
\end{eqnarray*}
and
\begin{eqnarray*}
\eta^2 G(u) \|\mathbf{B}\|\|\nabla u^+ \| &\leq& \frac12 \cdot \frac{\Big( \sqrt{\lambda}\, (u^+)^{-\frac{\beta+1}{2}}\, G(u)  \eta \|\nabla u^+\|\Big)^2  }{2} +2 \cdot \frac{\Big(\sqrt{\lambda^{-1}}\, (u^+)^{\frac{\beta+1}{2}} \| \mathbf{B}\| \eta \Big)^2}{2} \\
&\leq& \frac{\lambda}{4} \eta^2 G'(u) \|\nabla u^+\|^2 +\frac{1}{\lambda} \|\mathbf{B}\|^2 (u^+)^{\beta+1} \eta^2 .
\end{eqnarray*}
Therefore, on $\{\widetilde{\varphi}>0\}$, it holds
\begin{eqnarray}
&&\frac{\lambda}{2} \eta^2 G'(u) \|\nabla u^+\|^2 \nonumber \\
&&\leq  \big \langle A\nabla u, \nabla \widetilde{\varphi}  \big \rangle+\langle \mathbf{B}, \nabla u \rangle \widetilde{\varphi} +\Big( \frac{\|\mathbf{B}\|^2}{\lambda} \eta^2+ \frac{4 d^2M^2}{\lambda}\| \nabla \eta \|^2\Big) (u^+)^{\beta+1}. \label{baseq3}
\end{eqnarray}
Since $\{ \widetilde{\varphi}=0 \} \cap \big(U \times (0, \tau)\big) = \{ \eta =0 \} \cup \{ u \leq 0 \}$ and $\nabla u^+=0$ on $\{ u \leq 0 \}$, \eqref{baseq3} holds on $U \times (0, \tau)$. Combining \eqref{baseq3} and \eqref{maindivest}, we obtain for a.e. $\tau \in (0, T)$
\begin{eqnarray}
&& \hspace{-2em}\frac{1}{\beta+1}\int_U \eta^2 (u^+)^{\beta+1} \mid_{t= \tau} \psi dx+\frac{\lambda \beta}{2} \int_0^{\tau} \int_U \eta^2 (u^+)^{\beta-1}  \| \nabla u^+\|^2 dx dt \nonumber  \\
&& \; \hspace{-2em}\leq \int_0^{\tau} \int_{U} \Big( \frac{\|\mathbf{B}\|^2}{\lambda} \eta^2+ \frac{4 d^2 M^2}{\lambda}\| \nabla \eta \|^2 \Big) (u^+)^{\beta+1} dx dt + \frac{2}{\beta+1} \int_0^{\tau} \int_U \eta |\partial_t \eta | (u^+)^{\beta+1} \, \psi dx dt. \nonumber \\
 \label{fest1}
\end{eqnarray}\\
Let $(\bar{x}, \bar{t})$ be an arbitrary but fixed point in $U \times (0,T)$ and $R_{\bar{x}}(r)$ be the open cube in $\R^d$ of edge length $r>0$ centered at $\bar{x}$. Define $Q(r)  := R_{\bar{x}}(r) \times (\bar{t}- r^2, \bar{t})$.
\begin{theo} \label{weoifo9iwjeof}
Suppose that  $Q(3r) \subset U \times (0,T)$. Under the assumption {\bf(I)}, we have
\begin{equation} \label{supest}
\|u\|_{L^{\infty}(Q(r))} \leq C \|u\|_{L^{\frac{2p}{p-2},2}(Q(2r))},
\end{equation}
where $C>0$ is a constant depending only on $r$, $\lambda$, $M$ and $\|\mathbf{B}\|_{L^p(R_{\bar{x}}(3r))}$.
\end{theo}
\begin{proof}
Let $\eta \in C_0^{\infty}(R_{\bar{x}}(r) \times (\bar{t}-9r^2, \bar{t}] )$. Then \eqref{fest1} holds with $U\times (0,T)$ replaced by $Q(3r)$.  Using appropriate scaling arguments (cf. \cite[proof of Theorem 2]{ArSe}), we may assume $r = \frac{1}{3}$.\, Set $v:= (u^+)^{\gamma}$ with $\gamma:=\frac{\beta+1}{2}$. Then $\|\nabla v \|^2 = {\gamma}^2 (u^+)^{\beta-1}\|\nabla u^+\|^2 $. By \eqref{fest1}, it holds for a.e. $\tau \in (\bar{t}-1, \bar{t})$
\begin{eqnarray*}
&&\frac{c_0}{2\gamma}\int_{R_{\bar{x}}(1)} \eta^2 v^2 \mid_{t= \tau} dx+\frac{\lambda}{2\gamma^2}  \int_{\bar{t}-1}^{\tau} \int_{R_{\bar{x}}(1)} \eta^2  \|\nabla v\|^2 dx dt \nonumber  \\
&& \qquad  \leq \iint_{Q(1)} \Big(\frac{\|\mathbf{B}\|^2}{\lambda} \eta^2+ \frac{4 d^2 M^2}{\lambda}\| \nabla \eta \|^2  \Big) v^2 dx dt + \iint_{Q(1)} \eta |\partial_t \eta |v^2 \, \psi dx. \nonumber 
\end{eqnarray*}
Let $l$ and $l'$ be positive numbers satisfying $\frac13 < l' <l \leq \frac23$. Assume that $\eta \equiv 1 $ in $Q(l')$, $\eta \equiv 0$ outside $Q(l)$, $0 \leq \eta \leq 1$, and $|\partial_t \eta |, \|\nabla \eta\|\leq 2d(l-l')^{-1}$. Then
\begin{eqnarray*}
&&\iint_{Q(1)} \Big( \frac{\|\mathbf{B}\|^2}{\lambda} \eta^2+ \frac{4 d^2M^2}{\lambda}\| \nabla \eta \|^2   \Big) v^2 dx dt \\
&& \;\;\leq \frac{4d^2}{\lambda}  (l-l')^{-2} \iint_{Q(l)} \left( \|\mathbf{B}\|^2 + 4 d^2 M^2 \right) v^2 dx dt \\
&& \;\; \leq\frac{4d^2}{\lambda}  (l-l')^{-2} (\|\mathbf{B}\|^2_{L^{p}(R_{\bar{x}}(1))} + 4d^2M^2)\|v\|^2_{L^{\frac{2p}{p-2},2}(Q(l))},
\end{eqnarray*}
and since $\frac{2q}{q-1} \leq \frac{2p}{p-2}$, it follows
\begin{eqnarray*}
\int_{\bar{t}-1}^{\bar{t}}  \int_{R(1)} \eta |\partial_t \eta |v^2 \, \psi dx &\leq& 2d(l-l')^{-1} \|\psi\|_{L^q(R_{\bar{x}}(1))}\|v\|^2_{L^{\frac{2q}{q-1},2}(Q(l))} \\
&\leq& 2d(l-l')^{-2} \|\psi\|_{L^q(R_{\bar{x}}(1))}  \|v\|^2_{L^{\frac{2p}{p-2},2 }(Q(l))}.
\end{eqnarray*}
Thus we obtain
\begin{eqnarray*}
\lambda \| \eta \nabla v \|^2_{L^2(Q(1))} \leq 2C_1(l-l')^{-2} \gamma^2 \|v\|^2_{L^{\frac{2p}{p-2},2}(Q(l))}
\end{eqnarray*}
and
$$
\| \eta v\|^2_{L^{2, \infty}(Q(1))} \leq 2c_0^{-1}C_1(l-l')^{-2} \gamma^2 \|v\|^2_{L^{\frac{2p}{p-2},2}(Q(l))},
$$
where $C_1=\frac{4d^2}{\lambda}(\|\mathbf{B}\|^2_{L^{p}(R_{\bar{x}}(1))} + 4d^2M^2)+ 2d\|\psi\|_{L^q(R_{\bar{x}}(1))} $.\\
\centerline{}
Now set $\theta:=1-\frac{d}{p}$ and $\sigma:=1+\frac{\theta}{2}$ if $d=2$,\; $\sigma:=1+\frac{2\theta}{d}$ if $d \geq 3$. Set $p_{\sigma}:=\left( \frac{\sigma p}{p-2}  \right)'=\frac{\sigma p}{\sigma p-p+2}$,\; $q_{\sigma}:= \sigma'=\frac{\sigma}{\sigma-1}$. Then 
$$
 \frac{d}{2p_{\sigma}}+\frac{1}{q_{\sigma}} <1 \;\text{ if } d=2, \qquad \; \frac{d}{2p_{\sigma}}+\frac{1}{q_{\sigma}} = 1 \;\text{ if } d\geq 3.
$$
By \cite[Lemma 3]{ArSe},
\begin{eqnarray}
\|v^{\sigma}\|^{2/\sigma}_{L^{\frac{2p}{p-2},2 }(Q(l'))} &\leq& \|(\eta v)^{\sigma}\|^{2/\sigma}_{L^{\frac{2p}{p-2},2}(Q(1))}\nonumber   \\
 &=& \|\eta v\|^{2}_{L^{\frac{2\sigma p}{p-2},2\sigma}(Q(1))} \nonumber  \\
&=&\|\eta v\|^{2}_{L^{2(p_{\sigma})',2(q_{\sigma})'}(Q(1))} \nonumber   \\
&\leq& K \Big( \|\eta v\|^2_{L^{2, \infty}(Q(1))}  + \| \nabla(\eta v) \|_{L^2(Q(1))}^2  \Big) \nonumber  \\
&\leq& K \Big( \|\eta v\|^2_{L^{2, \infty}(Q(1))}  + 2\| \eta \nabla  v \|_{L^2(Q(1))}^2  +8d^2(l-l')^{-2}\|v\|^2_{L^2(Q(l))}  \Big) \nonumber \\
&\leq& C_2(l-l')^{-2}\gamma^2  \|v\|^2_{L^{\frac{2p}{p-2},2 }(Q(l))}, \label{itstart}
\end{eqnarray}
where $C_2=K(4C_1\lambda^{-1} + 2C_1c_0^{-1}+8d^2)$. Now for $m \in \N \cup \{0\}$, set $l=l_{m}:=3^{-1}(1+2^{-m})$,\;\; $l'=l'_m := 3^{-1}(1+2^{-m-1})$,\; $\varphi_m:= \|(u^+)^{\sigma^m}\|^{2/\sigma^m}_{L^{\frac{2p}{p-2},2 }(Q(l_m))}$. Taking $\gamma=\sigma^m$ and $1/3<l'=l'_m <l= l_m \leq 2/3$ for $m \in \N \cup \{0\}$, we obtain using \eqref{itstart}
\begin{eqnarray} \label{iteration-3}
\varphi_{m+1} \leq (36C_2)^{\frac{1}{\sigma^m}} (2 \sigma)^{\frac{2m}{\sigma^m}} \varphi_m.
\end{eqnarray}
Iterating \eqref{iteration-3}, we get
\begin{eqnarray*}
\varphi_{m+1} &\leq& (36C_2)^{\sum_{i=0}^{m}\frac{1}{\sigma^i}}  (2 \sigma)^{\sum_{i=0}^{m}\frac{2i}{\sigma^i}} \varphi_0 \\
&\leq& \underbrace{(36C_2)^{ \frac{\sigma}{\sigma-1} }  (2 \sigma)^{\frac{2\sigma}{(\sigma-1)^2}}}_{=:C_3} \|u\|^2_{L^{\frac{2p}{p-2},2 }(Q(2/3))}.
\end{eqnarray*}
Letting $m \rightarrow \infty$, we get 
$$
\| u^+\|_{L^{\infty}(Q(1/3))} \leq \sqrt{C_3} \|u\|_{L^{\frac{2p}{p-2},2 }(Q(2/3))}.
$$
Exactly in the same way, but with $u$ replaced by $-u$, we obtain \eqref{supest} with $C=2\sqrt{C_3}$.
\end{proof}

\subsection{Elliptic H\"{o}lder regularity and estimate}
\begin{lem} \label{div}
Let $U$ be a bounded open ball in $\R^d$. Let $f \in L^{\widetilde{q}}(U)$ with $\frac{d}{2}<\widetilde{q}<d$. Then there exists  \;$\bold{F}=\left(f_1, \dots, f_d   \right) \in H^{1,\widetilde{q}}(U, \R^d)$ such that \;${\rm div}\bold{F} = f$   in $U$ and
\begin{align*}
\sum_{i=1}^{d }\|f_i\|_{H^{1,\widetilde{q}}(U)} \leq C \|f\|_{L^{\widetilde{q}}(U)},
\end{align*}
where $C>0$ only depends on $\widetilde{q}$, $U$.
In particular, applying the Sobolev inequality, we get
\[
\sum_{i=1}^{d }\|f_i\|_{L^{\frac{d\widetilde{q}}{d-\widetilde{q}}}(U)} \leq C' \|f\|_{L^{\widetilde{q}}(U)},
\]
where $C'>0$ only depends on $\widetilde{q}$, $U$.
\end{lem}
\begin{proof}
By \cite[Theorem 9.15 and Lemma 9.17]{Gilbarg}, there exists $u\in H^{2, \widetilde{q}}(U) \cap H_0^{1,\widetilde{q}}(U) $ such that $\Delta u = f$ in $U$ and
\[
\|u\|_{H^{2,\widetilde{q}}(U)} \leq C_1 \|f \|_{L^{\widetilde{q}}(U)},
\] 
where $C_1>0$ is a constant only depending on $\widetilde{q}$, $U$. Let $\bold{F}:= \nabla u$. Then $\bold{F} \in H^{1,\widetilde{q}}(U, \R^d)$ with $\text{div}\bold{F}=f$ in $U$ 
and it holds 
\begin{align*}
\sum_{i=1}^{d }\|f_i\|_{H^{1,{\widetilde{q}}}(U)}&=\sum_{i=1}^{d }\|\partial_i u\|_{H^{1,{\widetilde{q}}}(U)}=\sum_{i=1}^{d } \Big( \|\partial_i u\|^{\widetilde{q}}_{L^{\widetilde{q}}(U)} + \sum_{j=1}^{d} \left\|\partial_j \partial_i u \right\|^{\widetilde{q}}_{L^{\widetilde{q}}(U)} \Big)^{\frac{1}{\widetilde{q}}}\\
%&\leq \sum_{i=1}^{d } \Big( \|\partial_i u\|_{L^{\widetilde{q}}(U)} + \sum_{j=1}^{d} \left\|\partial_j \partial_i u \right\|_{L^{\widetilde{q}}(U)} \Big) \\
&=  \sum_{i=1}^{d } \|\partial_i u\|_{L^{\widetilde{q}}(U)}+\sum_{i=1}^{d } \sum_{j=1}^{d} \left\|\partial_j \partial_i u \right\|_{L^{\widetilde{q}}(U)}\\
&\leq \left( d+d^2 \right)^{\frac{\widetilde{q}-1}{\widetilde{q}}}\Big( \sum_{i=1}^{d } \|\partial_i u\|^{\widetilde{q}}_{L^{\widetilde{q}}(U)}+\sum_{i=1}^{d } \sum_{j=1}^{d} \left\|\partial_j \partial_i u \right\|^{\widetilde{q}}_{L^{\widetilde{q}}(U)} \Big)^{\frac{1}{\widetilde{q}}}\\
&\leq \left( d+d^2 \right)^{\frac{{\widetilde{q}}-1}{\widetilde{q}}} \|u\|_{H^{2,\widetilde{q}}(U)} \\
&\leq C_1 \left( d+d^2 \right)^{\frac{\widetilde{q}-1}{\widetilde{q}}}\,  \|f\|_{L^{\widetilde{q}}(U)}.
\end{align*} 
\vspace{-2.3em}
\end{proof}\\
\centerline{}
The following theorem is an adaptation of \cite[Théorème 7.2]{St65} using \cite[Theorem 1.7.4]{BKRS}. It might already exist in the literature but we couldn't find any reference for it and therefore provide a proof.
\begin{theo} \label{holreg}
Let $U$ be a bounded open ball in $\R^d$. Let $A=\left( a_{ij} \right)_{1 \leq i,j \leq d}$ be a matrix of bounded functions on $U$ that is uniformly strictly elliptic. Assume $\mathbf{B} \in L^p(U,\R^d)$, $c \in L^{q}(U)$,  $f \in L^{\widetilde{q}}(U)$ for some $p>d$, \,$q,\, \widetilde{q}>\frac{d}{2}$. If $u \in H^{1,2}(U)$ satisfies
\begin{equation} \label{pdedec}
\int_U \left \langle A \nabla u, \nabla \varphi \right \rangle + \left( \left \langle \mathbf{B}, \nabla u  \right \rangle  + cu  \right) \varphi \;dx=\int _U f \varphi\;dx,\;  \;  \text{ for all } \varphi \in C_0^{\infty}(U),
\end{equation}
then for any open ball $U_1$ in $\R^d$ with $\overline{U}_1 \subset U$, we have $u \in C^{0,\gamma}(\overline{U}_1)$ and
\[
\|u\|_{C^{0,\gamma}(\overline{U}_1)} \leq C \left( \|u\|_{L^1(U)}+\|f\|_{L^{\widetilde{q}}(U)} \right),
\]
where $\gamma \in (0,1)$ and $C>0$ are constants which are independent of $u$ and $f$.
\end{theo}
\begin{proof}
Without loss of generality, we may assume $\frac{d}{2}<\widetilde{q}<d$. Let $U_2$ be an open ball in $\R^d$ satisfying $\overline{U}_1 \subset U_2 \subset \overline{U}_2 \subset U$. By Lemma \ref{div}, we can find $\bold{F}=(f_1, \cdots, f_d) \in H^{1,\widetilde{q}}(U_2, \R^d) \subset L^{\frac{d\widetilde{q}}{d-\widetilde{q}}}(U_2, \R^d)$ such that 
\begin{align*}
\text{div} \bold{F} = f \; \text{ in } U_2,  \qquad \; \sum_{i=1}^{d} \|f_i\|_{L^{\frac{d\widetilde{q}}{d-\widetilde{q}}}(U_2)} \leq C_1 \|f\|_{L^{\widetilde{q}}(U_2)},
\end{align*}
where $C_1>0$ is a constant only depending on $\widetilde{q}$ and $U_2$. Then \eqref{pdedec} implies
\begin{equation*} 
\int_{U_2} \left \langle A \nabla u, \nabla \varphi \right \rangle + \left( \left \langle \mathbf{B}, \nabla u  \right \rangle  + cu  \right) \varphi \;dx=\int _{U_2}  \left \langle -\bold{F}, \nabla \varphi \right \rangle dx  \;  \text{ for all } \varphi \in C_0^{\infty}(U_2).
\end{equation*}
Given $x \in U_1$, $r>0$ with $r<\text{dist}(x, U_2)$, set $\omega_{x}(r):=\sup_{B_{x}(r)} u-\inf_{B_{x}(r)} u$.  By \cite[Théorème 7.2]{St65} and Lemma \ref{div},
\begin{align*}
\omega_{x}(r) & \leq K \left( \|u\|_{L^2(U_2)}+ \sum_{i=1}^{d} \|f_i\|_{L^{\frac{d\widetilde{q}}{d-\widetilde{q}}}(U_2)} \right) r^{\gamma} \\
& \leq K (1+C')\left( \|u\|_{L^2(U_2)}+ \|f\|_{L^{\widetilde{q}}(U_2)} \right) r^{\gamma},
\end{align*}
where $\gamma \in (0,1)$ and $K, C'>0$ are constants which are independent of $x$, $r$, $u$, $\bold{F}$, $f$. Thus we have
\[
\int_{B_{r}(x)} |u(y)-u_{x,r}|^2 dy \leq   (K')^2 \left( \|u\|_{L^2(U_2)}+ \|f\|_{L^{\widetilde{q}}(U_2)} \right)^2  r^{d+2\gamma},
\] 
where $u_{x,r}:=\frac{1}{|B_r(x)|}\int_{B_r(x)} u(y) \,dy$ and $(K')^2:= K^2 \cdot\frac{\pi^{d/2}}{\Gamma(\frac{d}{2}+1)}(1+C')^2$. Finally by \cite[Theorem 3.1]{HanLin}, \cite[Theorem 1.7.4]{BKRS}, 
we obtain 
\begin{align*}
\|u\|_{C^{0,\gamma}(\overline{U}_1)} &\leq c \big(  K' \left( \|u\|_{L^2(U_2)}+ \|f\|_{L^{\widetilde{q}}(U_2)} \right)+ \|u\|_{L^2(U_2)} \big) \\
%&\leq \left(cK \vee c\right) \left(\|u\|_{L^2(U_2)}+  \|f\|_{L^{\widetilde{q}}(U_2)}  \right) \\
&\leq \left(cK' \vee c\right) \left(\|u\|_{H^{1,2}(U_2)}+  \|f\|_{L^{\widetilde{q}}(U_2)}  \right) \\
& \leq \left(cK' \vee c\right) \big( C_1\|u\|_{L^{1}(U)}+ C_1\|f\|_{L^{\widetilde{q}}(U)}+\|f\|_{L^{\widetilde{q}}(U_2)} \big) \\
&\leq (C_1+1)\left(cK' \vee c\right) \big( \|u\|_{L^{1}(U)}+\|f\|_{L^{\widetilde{q}}(U)} \big),
\end{align*}
where $c>0$, $C_1>0$ are constants which are independent of $u$ and $f$. 
\end{proof}

\section{The $L^1$-generator and its strong Feller semigroup} \label{4.1}
\subsection{Framework} \label{4.1ok}
Let $\rho \in H^{1,2}_{loc}(\R^d) \cap L_{loc}^{\infty}(\R^d)$, $\psi \in L^1_{loc}(\R^d)$ be a.e. strictly positive functions satisfying $\frac{1}{\rho}$, $\frac{1}{\psi} \in L_{loc}^{\infty}(\R^d)$. Here we assume that  the expressions $\frac{1}{\rho}$, $\frac{1}{\psi}$, denote any Borel measurable functions satisfying $\rho\cdot \frac{1}{\rho}=1$ and $\psi\cdot \frac{1}{\psi}=1$ a.e. respectively (later, especially in Section \ref{effwegewe} it will be important which Borel measurable version version $\frac{1}{\psi}$ we choose, but for the moment it doesn't matter). Set $\widehat{\rho}:=\rho \psi$,\, $\mu:=\widehat{\rho} \,dx$. If $U$ is any open subset of $\R^d$, then the bilinear form $\int_{U} \langle \nabla u, \nabla v \rangle dx$, \, $u,v \in C_0^{\infty}(U)$ is closable in $L^2(U, \mu)$ by \cite[Subsection II.2a)]{MR}. Define $\widehat{H}_0^{1,2}(U, \mu)$ as the closure of $C_0^{\infty}(U)$ in $L^2(U, \mu)$ with respect to the norm $\left(\int_U \|\nabla u\|^2 dx + \int_{U} u^2 d\mu \right)^{1/2}$. Thus $u \in \widehat{H}_0^{1,2}(U, \mu)$, if and only if there exists $(u_n)_{n \geq 1} \subset C_0^{\infty}(U)$  such that 
\begin{equation} \label{closablet}
\lim_{n \rightarrow \infty} u_n =u \; \text{ in } L^2(U, \mu), \;\quad \lim_{n,m \rightarrow \infty} \int_U \| \nabla (u_n-u_m) \|^2 dx =0,
\end{equation}
and moreover $\widehat{H}^{1,2}_0(U, \mu)$ is a Hilbert space with the inner product 
$$
\langle u, v \rangle_{\widehat{H}^{1,2}_0(U, \mu)}= \lim_{n \rightarrow \infty} \int_{U} \langle \nabla u_n, \nabla v_n \rangle dx + \int_{U} uv \,d\mu,
$$

\text{} \hspace{-1.81em}where $(u_n)_{n \geq 1}, (v_n)_{n \geq 1} \subset C_0^{\infty}(U)$ are arbitrary sequences that satisfy \eqref{closablet}.\\
\centerline{}
If $u \in \widehat{H}^{1,2}_0(V, \mu)$ for some bounded open subset $V$ of $\R^d$, then $u \in H_0^{1,2}(V) \cap L^2(V, \mu)$ and there exists $(u_n)_{n \geq 1} \subset C_0^{\infty}(V)$, such that
$$
\lim_{n \rightarrow \infty} u_n = u \;\; \text{ in } H_0^{1,2}(V)  \; \text{ and } \text{ in } L^2(V, \mu).
$$
Consider a symmetric matrix of functions $A=(a_{ij})_{1 \leq i,j \leq d}$ satisfying 
\begin{equation*}\label{acondi}
a_{ij} =a_{ji} \in H_{loc}^{1.2}(\R^d), \quad  1\leq i,j \leq d,
\end{equation*}
and assume $A$ is locally uniformly strictly elliptic, i.e. for every open ball $B$, there exist constants $\lambda_B, \Lambda_B>0$ such that 
\begin{equation}\label{ellipticity}
\lambda_{B} \|\xi\|^2 \leq \langle A(x) \xi, \xi \rangle \leq \Lambda_B \|\xi\|^2, \quad \text{ for all } \xi \in \R^d, \; x \in B.
\end{equation}
Define $\widehat{A}:= \frac{1}{\psi} A$. By \cite[Subsection II.2b)]{MR},
the symmetric bilinear form 
\begin{equation*} \label{weofijfh}
\mathcal{E}^0(f,g) := \frac12 \int_{\R^d} \langle \widehat{A} \nabla f, \nabla g \rangle d\mu, \quad f,g \in C_0^{\infty}(\R^d),
\end{equation*}
is closable in $L^{2}(\R^d, \mu)$ and its closure $(\mathcal{E}^0, D(\mathcal{E}^0))$ is a symmetric  Dirichlet form in $L^2(\R^d, \mu)$ (see \cite[(II. 2.18)]{MR}). Denote the corresponding generator of $(\mathcal{E}^0, D(\mathcal{E}^0))$ by $(L^0,D(L^0))$. Let $f \in C_0^{\infty}(\R^d)$. Using integration by parts, for any $g \in C_0^{\infty}(\R^d)$,
\begin{eqnarray*}
\mathcal{E}^0(f,g) &=& \frac12 \int_{\R^d} \langle \rho  A \nabla f, \nabla g \rangle dx \\
&=& -\frac12 \int_{\R^d} \left( \rho \, \text{trace} (A\nabla^2f ) + \langle \rho \nabla A +  A\nabla \rho, \nabla f \rangle \right) g \,dx \\
&=&  - \int_{\R^d} \big( \frac12 \text{trace} (\widehat{A}\nabla^2f ) + \langle \underbrace{\frac{1}{2 \psi} \nabla A +  \frac{A\nabla \rho}{2 \rho \psi}}_{=:\beta^{\rho,A, \psi} }, \nabla f \rangle \big) g \,d\mu.
\end{eqnarray*}
Thus $f \in D(L^0)$. This implies $C_0^{\infty}(\R^d) \subset D(L^0)$ and 
$$
L^0f =\frac12\text{trace}(\widehat{A} \nabla^2 f) + \langle \beta^{\rho, A, \psi}, \nabla f \rangle  \in L^2(\R^d, \mu).
$$ 
Let $(T^0_t)_{t>0}$ be the sub-Markovian $C_0$-semigroup of contractions on $L^2(\R^d,\mu)$ associated with $(L^0,D(L^0))$. By Proposition \ref{extlempt}, \,$T^0_t|_{L^1(\R^d, \mu) \cap L^{\infty}(\R^d, \mu)}$ can be uniquely extended to a sub-Markovian  $C_0$-semigroup of contractions $(\overline{T^0_t})_{t>0}$ on $L^1(\R^d, \mu)$.\\
\quad Now let $\mathbf{B} \in L_{loc}^2(\R^d, \R^d, \mu)$ be weakly divergence free with respect to $\mu$, i.e.
\begin{equation} \label{divfree}
\int_{\R^d} \langle \mathbf{B}, \nabla u \rangle d\mu = 0, \quad \text{ for all } u \in C_0^{\infty}(\R^d).
\end{equation}
Moreover assume
\begin{equation} \label{adddivassump}
 \mathbf{ \rho \psi B} \in L_{loc}^2(\R^d, \R^d).
\end{equation}
Then using Lemma \ref{applempn},  \eqref{divfree} can be extended to all $u \in \widehat{H}_0^{1,2}(\R^d, \mu)_{0,b}$ and
$$
\int_{\R^d} \langle \mathbf{B}, \nabla u \rangle v d\mu = -\int_{\R^d} \langle \mathbf{B}, \nabla v \rangle u d\mu, \quad \text{ for all } u,v \in \widehat{H}_0^{1,2}(\R^d, \mu)_{0,b}.
$$
Define $Lu := L^0 u+ \langle \mathbf{B}, \nabla u \rangle, \; u \in D(L^0)_{0,b}$. Then  $(L, D(L^0)_{0,b})$ is an extension of
$$
\frac12 \text{trace}(\widehat{A} \nabla^2 u) + \langle \beta^{\rho, A,\psi} + \mathbf{B}, \nabla u \rangle, \quad u \in C_0^{\infty}(\R^d).
$$
For any bounded open subset $V$ of $\R^d$, 
$$
\mathcal{E}^{0,V}(f,g) := \frac12 \int_{V} \langle \widehat{A} \nabla f, \nabla g \rangle d\mu, \quad f,g \in C_0^{\infty}(V).
$$
is also closable on $L^2(V, \mu)$ by \cite[Subsection II.2b)]{MR}. Denote by $(\mathcal{E}^{0,V}, D(\mathcal{E}^{0,V}))$ the closure of $(\mathcal{E}^{0,V}, C_0^{\infty}(V))$ in $L^2(V, \mu)$. Using \eqref{ellipticity} and $0<\inf_V \rho \leq \sup_V \rho< \infty$, it is clear that $D(\mathcal{E}^{0,V})= \widehat{H}_0^{1,2}(V, \mu)$ since the norms $\| \cdot \|_{D(\mathcal{E}^{0,V})}$ and  $\|\cdot\|_{\widehat{H}^{1,2}_0 (V, \mu)}$ are equivalent.  Denote by $(L^{0,V}, D(L^{0,V}))$ the generator of $(\mathcal{E}^{0,V}, D(\mathcal{E}^{0,V}))$, by $(G^{0,V}_{\alpha})_{\alpha>0}$ the associated sub-Markovian $C_0$-resolvent of contractions on $L^2(V, \mu)$, by $(T_t^{0,V})_{t>0}$ the associated sub-Markovian $C_0$-semigroup of contractions on $L^2(V, \mu)$ and by $(\overline{T}_t^{0,V})_{t>0}$ the unique extension of $(T^{0,V}_t|_{L^1(V, \mu) \cap L^{\infty}(V, \mu)})_{t>0}$  on $L^1(V, \mu)$, which is a sub-Markovian  $C_0$-semigroup of contractions on $L^1(V, \mu)$. Let $(\overline{L}^{0,V}, D(\overline{L}^{0,V}))$ be the generator corresponding to $(\overline{T}_t^{0,V})_{t>0}$.
By Proposition \ref{extlempt}, $(\overline{L}^{0,V}, D(\overline{L}^{0,V}))$ is the closure of $(L^{0,V}, D(L^{0,V}))$  on $L^1(\R^d, \mu)$.

\subsection{The $L^1$-generator}\label{l1existencere}

In this section, we use all notations and assumptions from Section \ref{4.1ok}. All ideas and techniques used here are based on \cite[Chapter 1]{St99}. But the structure of the given symmetric Dirichlet form differs from that of \cite{St99} which will enable us to cover a degenerate diffusion matrix. Because of that subtle difference, we  check the details one by one that the methods of \cite[Chapter 1]{St99} can be adapted to our situation. The main difference between \cite[Chapter 1]{St99} and what is treated here is that we consider local convergence in the space $\widehat{H}^{1,2}_0(V, \mu)$, while \cite[Chapter 1]{St99} considers the space $H_0^{1,2}(V, \mu)$ where the pre-invariant density of \cite[Chapter 1]{St99} does not need to be locally bounded. Since $\widehat{H}^{1,2}(V,\mu)$ is naturally included in the Sobolev space $H^{1,2}(V)$, the arguments to derive our results are at times even easier than the ones of \cite[Chapter 1]{St99}. For instance, we can use the product and chain rules in $\widehat{H}^{1,2}(V, \mu)$ inherited from the Sobolev space structure (see Remark \ref{gensitce}). Moreover, assumption \eqref{adddivassump} will play an important role to apply the methods of \cite[Chapter 1]{St99}.

\begin{lem} \label{existlem}
Let $V$ be a bounded open subset of $\R^d$. Then
\begin{itemize}
\item[(i)]
$D(\overline{L}^{0,V})_b \subset \widehat{H}_0^{1,2}(V, \mu)$.
\item[(ii)]
$\lim_{t \rightarrow 0+}T_t^{0,V}u = u$ \;\, in $\widehat{H}_0^{1,2}(V, \mu)$  \;for all  $u \in D(\overline{L}^{0,V})_b$.

\item[(iii)]
$\mathcal{E}^0(u,v) = - \int_{V} \overline{L}^{0,V} u\, v \,d\mu$ \,\;for all $u \in D(\overline{L}^{0,V})_b$, $v \in \widehat{H}^{1,2}_0(V, \mu)_b$.

\item[(iv)]
Let $\varphi \in C^2(\R^d)$, $\varphi(0)=0$, and $u \in D(\overline{L}^{0,V})_b$. Then $\varphi(u) \in D(\overline{L}^{0,V})_b$ and
\begin{equation*} \label{cdsid}
\overline{L}^{0,V} \varphi(u) = \varphi'(u) \overline{L}^{0,V} u +  \frac12\varphi ''(u) \langle\widehat{A} \nabla u, \nabla u \rangle.
\end{equation*}
\end{itemize}
\end{lem}
\begin{proof}
Let $u \in D(\overline{L}^{0,V})_b$. Since $(T^{0,V}_t)_{t>0}$ is an analytic semigroup on $L^2(V, \mu)$, we get
$$
\overline{T}_t^{0,V} u =T^{0,V}_t u \in D(L^{0,V}) \;\; \text{ for all } t>0,
$$ 
hence by Proposition \ref{extlempt},
$$
L^{0,V} T_t^{0,V}u = L^{0,V} \overline{T}_t^{0,V} u = \overline{L}^{0,V}\overline{T}_t^{0,V}u = \overline{T}_t^{0,V} \overline{L}^{0,V} u.
$$
Therefore
\begin{eqnarray*}
&&\mathcal{E}^{0,V}(T^{0,V}_t u - T_s^{0,V} u, T^{0,V}_t u - T_s^{0,V} u ) \\
&& \qquad \qquad = - \int_{V} L^{0,V} ( T^{0,V}_t u - T_s^{0,V} u ) \cdot (T^{0,V}_t u - T_s^{0,V} u ) d\mu \\
&&\qquad \qquad   =-\int_{V} ( \overline{T}^{0,V}_t \overline{L}^{0,V} u - \overline{T}_s^{0,V} \overline{L}^{0,V}  u ) \cdot (T^{0,V}_t u - T_s^{0,V} u) d\mu \\
&&  \qquad \qquad  \leq \| \overline{T}^{0,V}_t \overline{L}^{0,V} u - \overline{T}_s^{0,V} \overline{L}^{0,V}  u \|_{L^1(V, \mu)} \cdot 2 \|u\|_{L^{\infty}(V, \mu)}\ \longrightarrow \ 0 \;\; \; \text{ as } \; t, s \rightarrow 0+.
\end{eqnarray*}
Thus $(T_t^{0,V} u)_{t>0}$ is an $\widehat{H}_0^{1,2}(V, \mu)$-Cauchy sequence as $t \rightarrow 0+$, which implies $u \in \widehat{H}_0^{1,2}(V, \mu)$ and  $\lim_{t \rightarrow 0+} T^{0,V}_t u = u \; \text{ in } \;  \widehat{H}_0^{1,2}(V, \mu)$.
Thus (i), (ii) are proved. \\ \\
Let $v \in \widehat{H}_0^{1,2}(V, \mu)_b$. Then
\begin{eqnarray*}
\mathcal{E}^{0,V}(u,v) &=& \lim_{t \rightarrow 0+} \mathcal{E}^{0,V} (T^{0,V}_t u, v) = \lim_{t \rightarrow 0+} - \int_{V} \left(L^{0,V} T_t^{0,V} u \right) v\, d\mu \\
&=& \lim_{t\rightarrow 0+} -\int_V \left(  \overline{T}^{0,V}_t \overline{L}^{0,V} u \right) \; v \,d \mu   =  -\int_V \overline{L}^{0,V}u  \; v \,d \mu,
\end{eqnarray*}
hence (iii) is proved.\\
\centerline{}
\quad (iv): Note $u \in D(\overline{L}^{0,V})_b \subset \widehat{H}_0^{1,2}(V, \mu)_b$.  Set $u_n:= nG^{0,V}_n u $, $M:= \|u\|_{L^{\infty}(V)}$. Then $\|u_n\|_{L^{\infty}(V)} \leq M$. By strong continuity,  $\lim_{n \rightarrow \infty} u_n =u \; \text{ in } \widehat{H}_0^{1,2}(V, \mu)$ and there exists a subsequence of $(u_n)_{n \geq 1}$, say $(u_n)_{n \geq 1}$ again, such that $\lim_{n \rightarrow \infty} u_n = u$ \; $\mu$-a.e. on $V$.
%%%%%%%%%%%%%%%%%%%%%%%%%%%%
\iffalse
Therefore $(\varphi(u_n) )_{n \geq 1}$ is a Cauchy sequence in $\widehat{H}_0^{1,2}(V, \mu)$ and $\varphi(u_n) \longrightarrow \varphi(u)$ in $L^2(V, \mu)$, which implies $\varphi(u) \in \widehat{H}_0^{1,2}(V, \mu)_b$. 
\fi
%%%%%%%%%%%%%%%%%%%%%%%%%%%%
Thus by Lebesgue's Theorem,  $\lim_{n \rightarrow \infty} \varphi(u_n) = \varphi(u)$ in $L^2(V, \mu)$. Observe that
\begin{eqnarray*}
\sup_{n \geq 1}\| \nabla \varphi (u_n) \|_{L^{2}(V, \R^d)}&=& \sup_{n \geq 1} \| \varphi'(u_n) \nabla u_n \|_{L^2(V, \R^d)} \\
&\leq& \|\varphi'\|_{L^{\infty}([-M, M  ])} \sup_{n \geq 1} \|u_n\|_{\widehat{H}^{1,2}_0(V, \mu)} <\infty.
\end{eqnarray*}
Thus by Banach-Alaoglu Theorem, $\varphi(u) \in \widehat{H}^{1,2}_0(V, \mu)$.\, Similarly, we  get $\varphi'(u) \in \widehat{H}_0^{1,2}(V, \mu)$. Let $v \in \widehat{H}_0^{1,2}(V, \mu)_b$.
Then by \cite[I. Corollary 4.15]{MR}, $v \varphi'(u) \in \widehat{H}_0^{1,2}(V, \mu)_b$ and 
\begin{eqnarray*}
\mathcal{E}^{0,V}(\varphi(u), v)  &=&  \frac12 \int_V \langle \widehat{A} \nabla \varphi(u), \nabla v \rangle d \mu \\
&=&  \frac12 \int_V \langle \widehat{A} \nabla u, \nabla v \rangle \varphi'(u)d \mu \\
&=&  \frac12 \int_V \langle \widehat{A} \nabla u, \nabla (v \varphi'(u))  \rangle d \mu  -  \frac12 \int_V \langle \widehat{A} \nabla u, \nabla u \rangle \varphi''(u)v d\mu \\
&=& -\int_V  \big(\varphi'(u)\overline{L}^{0,V} u   +   \frac12\varphi''(u) \langle\widehat{A} \nabla u, \nabla u \rangle \big) v\, d\mu.
\end{eqnarray*}
Since $ \varphi'(u)\overline{L}^{0,V} u   +   \frac12 \varphi''(u) \langle\widehat{A} \nabla u, \nabla u \rangle  \in L^1(V, \mu)$, (iv) holds by \cite[I. Lemma 4.2.2.1]{BH}.
\end{proof}
\centerline{}
Recall that a densely defined operator $(L,D(L))$ on a Banach space $X$ is called {\it dissipative} \;if for any $u \in D(L)$, there exists $l_u \in X'$ such that 
\begin{equation} \label{dissipativei}
\|l_u\|_{X'} = \|u\|_X,  \,\; l_u(u) = \|u\|_X^2  \;\text{ and  }\,  l_u(Lu) \leq 0.
\end{equation}

\begin{prop} \label{first1}
Let $V$ be a bounded open subset of $\R^d$. 
\begin{itemize}
\item[(i)]
The operator $(L^V, D(L^{0,V})_b)$ on $L^1(V, \mu)$ defined by
$$
L^{V} u  := L^{0,V} u  + \langle \mathbf{B}, \nabla u \rangle,\;\;  u \in D(L^{0,V})_b,
$$ 
is dissipative, hence closable on $L^1(V, \mu)$. The closure $(\overline{L}^V, D(\overline{L}^V))$ generates a sub-Markovian $C_0$-semigroup of contractions $(\overline{T}^V_t)_{t>0}$ on $L^1(V, \mu)$.

\item[(ii)]
$D(\overline{L}^V)_b \subset \widehat{H}^{1,2}_0(V, \mu)$ and
\begin{equation} \label{stamid}
\mathcal{E}^{0,V}(u,v) - \int_{V} \left \langle \mathbf{B}, \nabla u  \right \rangle v d\mu  = \int_V \overline{L}^V u \cdot v d \mu, \; \text{ for all } u \in D(\overline{L}^V)_b,\; v \in \widehat{H}_0^{1,2}(V, \mu)_b.
\end{equation}
\end{itemize}
\end{prop}
\begin{proof} (i)
\; {\bf Step 1:} For $u \in D(L^{0,V})_b$, we have $\int_V L^V u 1_{\{ u >1 \}} d \mu \leq 0$. \\
Let $\varphi_{\varepsilon}  \in C^2(\R)$, $\varepsilon>0$, be such that $\varphi_{\varepsilon}''\geq 0$,  $0 \leq \varphi_{\varepsilon} \leq 1$ and  $\varphi_{\varepsilon}(t) = 0$ if $t<1$, $\varphi'_{\varepsilon}(t) = 1$ if $t \geq 1+\varepsilon$. Then $\varphi_{\varepsilon}(u) \in D(\overline{L}^{0,V})$ by Lemma \ref{existlem}(iv) and 
\begin{eqnarray}
\int_{V} L^{0,V} u \, \varphi'_{\varepsilon}(u) d \mu 
&\leq& \int_{V} L^{0,V}u \,\varphi'_{\varepsilon}(u)\,d\mu +\int_{V} \frac12 \varphi''_{\varepsilon}(u) \langle \widehat{A} \nabla u, \nabla u \rangle d \mu  \nonumber \\
& = & \int_{V} \overline{L}^{0,V} \varphi_{\varepsilon} (u) d\mu  
= \lim_{t \rightarrow 0+}\int_V \frac{\overline{T}^{0,V}_t \varphi_{\varepsilon} (u)- \varphi_{\varepsilon} (u)     }{t} d \mu
\ \leq\  0, \label{subinvariant1}
\end{eqnarray}
where the last inequality followed by the $L^1(V, \mu)$-contraction property of $(\overline{T}_t^{0,V})_{t>0}$. \\
Since $\lim_{\varepsilon \rightarrow 0+} \varphi'_{\varepsilon}(t) = 1_{(0,\infty)}(t)$ for every $t \in \R$, we have
$$
\lim_{\varepsilon \rightarrow 0+}\varphi'_{\varepsilon}(u) = 1_{\{u>1\}} \,\;\; \mu \text{-a.e. on } V \; \text{ and } \; \| \varphi'_{\varepsilon} (u)\|_{L^{\infty}(V)} \leq 1.
$$
Thus by Lebesgue's Theorem 
$$
\int_V L^{0,V} u\, 1_{\{ u >1 \}} d \mu = \lim_{\varepsilon \rightarrow 0+} \int_{V} L^{0,V} u\;  \varphi_{\varepsilon}'(u) d \mu \leq 0.
$$
Similarly, since $\varphi_{\varepsilon}(u) \in \widehat{H}^{1,2}_0(V, \mu)$, using \eqref{divfree} we get
$$
\int_{V}\langle \mathbf{B}, \nabla u \rangle  1_{\{u >1\}} d \mu =  \lim_{\varepsilon \rightarrow 0+} \int_{V} \langle \mathbf{B}, \nabla u \rangle \varphi'_{\varepsilon}(u) d \mu  = \lim_{\varepsilon \rightarrow 0+} \int_{V} \langle \mathbf{B}, \nabla \varphi_{\varepsilon}(u) \rangle d \mu  = 0.
$$
Therefore $\int_{V} L^{V} u 1_{\{u >1 \}} d \mu \leq 0$ and Step 1 is proved.  Observe that by Step 1, for any $n \geq 1$
$$
\int_{V} \left(L^V n u \right)1_{\{ nu >1 \}} d \mu \leq 0 \, \Longrightarrow \, \int_{V} L^V  u \,1_{\{ u >\frac{1}{n} \}} d \mu \leq 0.
$$
Letting $n \rightarrow \infty$  it follows from Lebesgue's Theorem that
$\int_{V} L^V u \,1_{\{u >0 \}} d \mu \leq 0$. Replacing $u$ with $-u$, we have 
$$
-\int_{V} L^V u \,1_{\{u <0 \}} d \mu =  \int_{V} L^V (-u) \,1_{\{-u >0 \}} d \mu \leq 0,
$$
hence
$$
\int_{V} L^V u \,(1_{\{u >1 \}} - 1_{ \{ u <0 \}}) d \mu \leq 0.
$$
Setting $l_u := \|u\|_{L^1(V, \mu)}  (1_{\{u >1 \}} - 1_{ \{ u <0 \}})  \in L^{\infty}(V, \mu) = (L^1(V, \mu))'$, \eqref{dissipativei} is satisfied. Since $C_0^{\infty}(V) \subset   D(L^{0,V})_b$, $(L^{0,V}, D(L^{0,V})_b)$ is densely defined on $L^1(V, \mu)$. Consequently,
 $(L^{0,V}, D(L^{0,V})_b)$ is dissipative. \\
\centerline{}
{\bf Step 2:} \;We have $(1-L^V)(D(L^{0,V})_b) \subset L^1(V, \mu)$  densely. \\
Let $h \in L^{\infty}(V, \mu)  = (L^1(V, \mu))'$ be such that $\int_{V} (1-L^V) u \, h d\mu = 0$ for all $u \in D(L^{0,V})_b$.
Then $u \mapsto \int_{V} (1-L^{0,V}) u \, h \,d\mu $ is continuous with respect to the norm on $\widehat{H}^{1,2}_0(V, \mu)$ since
\begin{eqnarray*}
\big | \int_V (1-L^{0,V} u )u \, h d \mu   \big |  &=& \big| \int_{V} \langle \rho \psi \mathbf{B}, \nabla u \rangle h \,dx \big | \\
&\leq& \|h\|_{L^{\infty}(V)} \| \rho \psi \mathbf{B} \|_{L^2(V, \R^d)} \| \nabla u\|_{L^2(V, \R^d)} \\
&\leq& \|h\|_{L^{\infty}(V)} \| \rho \psi \mathbf{B} \|_{L^2(V, \R^d)} \|u\|_{\widehat{H}_0^{1,2}(V, \mu)}.
\end{eqnarray*}
Thus, by the Riesz representation Theorem, there exists $v \in \widehat{H}_0^{1,2}(V, \mu)$ such that
$$
\mathcal{E}_1^{0,V}(u,v) = \int_{V}  (1-L^{0,V}  )u \cdot h d \mu \quad \text{ for all } u \in D(L^{0,V})_b,
$$
which implies that
\begin{equation} \label{den12}
\int_V (1-L^{0,V})u \cdot  (h-v) d \mu = 0  \quad \text{ for all } u \in D(L^{0,V})_b.
\end{equation}
Since $(L^{0,V}, D(L^{0,V}))$ generates a sub-Markovian resolvent in $L^2(V, \mu)$,
$$
L^1(V, \mu) \cap L^{\infty}(V, \mu) \subset (1-L^{0,V})(D(L^{0,V})_b),
$$
hence $ (1-L^{0,V})(D(L^{0,V})_b) \subset L^1(V, \mu)$ densely. Therefore \eqref{den12} implies $h-v = 0$. In particular, $h \in \widehat{H}_0^{1,2}(V, \mu)$ and
\begin{eqnarray*}
\mathcal{E}^{0,V}_1(h,h) &=& \lim_{\alpha \rightarrow \infty} \mathcal{E}^{0,V}_1(\alpha G^{0,V}_{\alpha} h, h) \\
&=& \lim_{\alpha \rightarrow \infty} \int_V (1-L^{0,V})(\alpha G^{0,V}_{\alpha} h) \, h d\mu \\
&=& \lim_{\alpha \rightarrow \infty} \int_{V} \big \langle \rho \psi \mathbf{B}, \nabla (\alpha G_{\alpha} h) \big \rangle h dx  \\
&=& \int_{V} \big \langle \rho \psi \mathbf{B}, \nabla h \big \rangle h dx =  \frac12 \int_{V} \big \langle  \mathbf{B}, \nabla h^2 \big \rangle d\mu = 0,
\end{eqnarray*}
therefore $h = 0$. Then applying the Hahn-Banach Theorem \cite[Proposition 1.9]{BRE}, Step 2 is proved.  By the Lumer-Phillips Theorem \cite[Theorem 3.1]{LP61}, the closure $(\overline{L}^V, D(\overline{L}^V))$ of $(L^V, D(L^{0,V})_b)$ generates a contraction $C_0$-semigroup $(\overline{T}^V_t)_{t>0}$  on $L^1(V, \mu)$.\\
\centerline{}
{\bf Step 3:} $(\overline{T}^V_t)_{t > 0}$ is sub-Markovian. \\
Let $(\overline{G}_{\alpha}^V)_{\alpha>0}$ be the associated resolvent. It is enough to show that $(\overline{G}_{\alpha}^V)_{\alpha>0}$ is sub-Markovian since $T^V_t u = \displaystyle \lim_{\alpha \rightarrow \infty} \exp\big( t \alpha ( \alpha \overline{G}^V_{\alpha}u -u )   \big)$ in $L^1(V, \mu)$ by the proof of Hille-Yosida (cf. \cite[I. Theorem 1.12]{MR}). Observe that by construction 
$$
D(L^{0,V})_b \subset D(\overline{L}^V)  \text{ densely with respect to the graph norm } \| \cdot \|_{D(\overline{L}^V)}.
$$
Let $u \in D(\overline{L}^V)$ and take $u_n \in D(\overline{L}^{0,V})_b$ satisfying $\lim_{n \rightarrow \infty}u_n =u$ in $D(\overline{L}^V)$ and $\lim_{n \rightarrow \infty} u_n =u$, \, $\mu$-a.e. on $V$. Let $\varepsilon>0$ and $\varphi_{\varepsilon}$ be as in Step 1. Then by \eqref{subinvariant1}
$$
\int_{V} \overline{L}^V u \,1_{\{ u>1 \}} d\mu =\lim_{\varepsilon \rightarrow 0+}\int_{V} \overline{L}^V  u \, \varphi_{\varepsilon}'(u) d \mu \leq 0.
$$
Let $f \in L^1(V, \mu)$ and $u:= \alpha \overline{G}^V_{\alpha} f \in D(\overline{L}^V)$. If $f \leq 1$, then
\begin{eqnarray*}
\alpha \int_V u 1_{\{ u >1 \}} d \mu \leq \int_V (\alpha u - \overline{L}^V u )1_{\{ u >1 \}} d \mu = \alpha \int_V f 1_{\{ u >1 \}} d \mu \leq \alpha \int_V 1_{\{ u >1 \}} d \mu. 
\end{eqnarray*}
Therefore, $\alpha \int_V (u-1)1_{\{ u >1 \}} d \mu \leq 0$, which implies $u \leq 1$. If $f \geq 0$, then $-nf \leq 1$ for all $n \in \N$, hence $-nu \leq 1$ for all $n \in \N$. Thus $u \geq 0$. Therefore  $(\overline{G}^V_{\alpha})_{\alpha>0}$ is sub-Markovian. \\
\centerline{}
(ii) \; {\bf Step 1: }\;  It holds\, $D(\overline{L}^{0,V})_b \subset D(\overline{L}^V) \text{ \;and\; } \overline{L}^V u = \overline{L}^{0,V} u + \langle   \mathbf{B}, \nabla u \rangle, \; u \in D(\overline{L}^{0,V})_b.$
 Let $u \in D(\overline{L}^{0,V})_b$. \;Since $(T_t^{0,V})_{t>0}$ is an analytic semigroup, $T_t^{0,V} u \in D(L^{0,V})_b \subset D(\overline{L}^V)$ and $\overline{L}^V T_t^{0,V} u = L^{0,V} T_t^{0,V} u + \langle \mathbf{B}, \nabla T_t^{0,V}u \rangle  = \overline{T}_t^{0,V} \overline{L}^{0,V} u + \langle  \mathbf{B}, \nabla T_t^{0,V} u  \rangle$. By Lemma \ref{existlem}(ii), $\lim_{t \rightarrow 0+}T^{0,V}_t u =u$ in $\widehat{H}^{1,2}_0(V, \mu)$, which implies that 
\begin{equation*}
\lim_{t \rightarrow 0+} \overline{L}^V T_t^{0,V} u = \overline{L}^{0,V} u + \langle \mathbf{B}, \nabla u \rangle \quad \text{ in }\; L^1(V, \mu),
\end{equation*}
by \eqref{adddivassump}. \,Since $\lim_{t \rightarrow 0+} T^{0,V}_t u = \lim_{t \rightarrow 0+} \overline{T}^{0,V}_t u =u$\;  in \,$L^1(V, \mu)$ and $(\overline{L}^V, D(\overline{L}^V))$ is a closed operator on $L^1(V, \mu)$, we obtain
$$
u \in D(\overline{L}^V), \quad  \; \overline{L}^V u = \overline{L}^{0,V} u + \langle \mathbf{B}, \nabla u \rangle.
$$ 
\centerline{}
{\bf Step 2: } Let $u \in D(\overline{L}^V)_b$ and take $u_n \in D(L^{0,V})_b$ as in Step 3 of the proof of Proposition \ref{first1}(i). Let $M_1, M_2>0$ be such that $\|u\|_{L^{\infty}(V)}<M_1 < M_2$.  Then
\begin{equation} \label{imlemce}
\lim_{n \rightarrow \infty} \int_{\{ M_1 \leq |u_n| \leq M_2 \}} \langle \widehat{A} \nabla u_n, \nabla u_n \rangle d \mu = 0.
\end{equation}
Indeed, let $\varphi \in C^1(\R)$ be such that $\varphi'(t):= (t-M_1)^+ \wedge (M_2-M_1)$ with $\varphi(0)=0$. Then by Lemma \ref{existlem}(i) (iv), we have $\varphi (u_n) \in \widehat{H}_0^{1,2}(V, \mu)$. Observe that $\varphi'(u) = 0$, \, $\mu$-a.e. on $V$ and
\begin{eqnarray*}
&&\int_{\{ M_1 \leq u_n \leq M_2 \}} \varphi''(u_n)  \langle \widehat{A} \nabla u_n, \nabla u_n \rangle d \mu  = \int_{V} \big \langle \widehat{A} \nabla u_n, \nabla \varphi'(u_n) \big \rangle d \mu  \\
&&= \mathcal{E}^{0,V}(u_n, \varphi'(u_n)) = -\int_V L^{0,V} u_n \, \varphi'(u_n) d\mu \\
&&= - \int_V L^{0,V} u_n \,\varphi '(u_n)d \mu - \int_V \langle \mathbf{B}, \nabla \varphi(u_n) \rangle d \mu  \\
&&= - \int_V L^V u_n\, \varphi '(u_n)d \mu \longrightarrow - \int_V  \overline{L}^V u\, \varphi '(u)d \mu  =  0, \quad \text{ as } n \rightarrow \infty,
\end{eqnarray*}
where the convergence of the last limit holds by Lebesgue's Theorem, since
$$
\lim_{n \rightarrow \infty} \varphi'(u_n) = \varphi'(u) = 0, \; \text{ \;$\mu$-a.e. on $\R^d$ }
$$
and
\begin{eqnarray*}
&&\big| \int_V L^V u_n\cdot \varphi '(u_n)d \mu - \int_V \overline{L}^V u\cdot \varphi '(u)d \mu \big| \\
&\leq& \|\varphi' \|_{L^{\infty}(V)}\int_{V}   \big|  (L^V u_n-\overline{L}^V u)\big|  d\mu + \int_{V} |\overline{L}^V u|\cdot |\varphi'(u_n)-\varphi'(u)| d\mu \\
&& \; \longrightarrow 0 \; \text{ as } n \rightarrow \infty.
\end{eqnarray*}
Similarly, 
$$
\int_{\{-M_2 \leq u_n \leq -M_1\}} \langle \widehat{A} \nabla u_n, \nabla u_n \rangle d\mu =\int_{\{M_1 \leq -u_n \leq M_2\}} \langle \widehat{A} \nabla (-u_n), \nabla (-u_n) \rangle d\mu =0,
$$
hence \eqref{imlemce} is proved. \\
\centerline{}
{\bf Step 3: }   Let $u$, $u_n$, $n \geq 1$ be as in Step 2.  Let $\varphi \in C_0^2(\R)$ be such that $\varphi(t) = t$ if $|t|< \|u\|_{L^{\infty}(V)}+1$ and $\varphi(t) = 0$ if $|t| \geq \|u\|_{L^{\infty}(V)}+2$.  Using  Step 2 and Lebesgue's Theorem
\begin{eqnarray*}
\overline{L}^V \varphi (u_n)  = \varphi'(u_n)L^V u_n + \varphi''(u_n) \langle A \nabla u_n, \nabla u_n \rangle \longrightarrow \overline{L}^V u \;\; \text{ in } L^1(V, \mu) \; \text{ as } n \rightarrow \infty.
\end{eqnarray*}
Therefore
\begin{eqnarray*}
&&\mathcal{E}^{0,V}(\varphi(u_n)-\varphi(u_m), \,\varphi(u_n)-\varphi(u_m)) \\
&& \qquad \quad = - \int_V \overline{L}^V \big(  \varphi(u_n)- \varphi(u_m)  \big)   \cdot \big(  \varphi(u_n)- \varphi(u_m)  \big) d\mu \\
&& \qquad \quad \leq 2 \|\varphi \|_{L^{\infty}(\R^d)}  \| \overline{L}^V \varphi(u_n) - \overline{L}^V \varphi(u_m)  \|_{L^1(V, \mu)} \longrightarrow 0 \;\; \text{ as } \; n,m \rightarrow \infty.
\end{eqnarray*}
Thus $\lim_{n \rightarrow \infty} \varphi(u_n) = u$  \,in\, $\widehat{H}_0^{1,2}(V, \mu)$ by the completeness of $\widehat{H}_0^{1,2}(V, \mu)$. Then using \eqref{adddivassump}, for any $v \in \widehat{H}_0^{1,2}(V, \mu)_b$, 
\begin{eqnarray*}
&&\mathcal{E}^{0,V}(u,v) - \int_{V} \langle  \mathbf{B}, \nabla u  \rangle \, v d \mu = \lim_{n \rightarrow \infty} (\mathcal{E}^{0,V}\left(\varphi(u_n), v \right) - \int_{V} \langle  \mathbf{B}, \nabla \varphi(u_n) \rangle d \mu ) \\
&& \qquad \qquad = - \lim_{n \rightarrow \infty} \int_{V} \overline{L}^V \varphi(u_n)  \cdot v d\mu = - \int_{V} \overline{L}^V u \cdot v d\mu,
\end{eqnarray*}
which completes the proof of (ii).
\end{proof}

\begin{rem} \label{gensitce}
One can generalize the assumptions of Proposition \ref{first1} to more general positive functions $\rho$ as follows. Consider $\psi$ as in Section \ref{4.1ok} and assume $\phi \in H_{loc}^{1,2}(\R^d)$ with $\phi>0$ a.e. on $\R^d$ and let $\rho:= \phi^2$, $\mu:=\rho \psi dx$. Let $A=(a_{ij})_{1 \leq i,j \leq d}$ be a symmetric matrix of functions that is locally uniformly strictly elliptic on $\R^d$ and $a_{ij} \in H_{loc}^{1,2}(\R^d, \rho dx)$ for all $1\leq i,j \leq d$. Assume $\mathbf{B}$ satisfies \eqref{divfree} and $\psi \mathbf{B} \in L_{loc}^2(\R^d, \R^d,\rho dx)$. Let $(\mathcal{E}^0, D(\mathcal{E}^0))$, $(L^0, D(L^0))$ be defined in the same manner as in Section \ref{4.1ok}. For an open set $U$ in $\R^d$, define $\widehat{H}_{0, \rho}^{1,2}(U, \mu)$ as the closure of $C_0^{\infty}(U)$ in $L^2(U, \mu)$ with respect to the norm $\left(\int_U \|\nabla u\|^2 \rho dx + \int_{U} u^2 d\mu \right)^{1/2}$. Then replacing $\widehat{H}^{1,2}_0(U, \mu)$ with $\widehat{H}_{0, \rho}^{1,2}(U, \mu)$, one obtains the same results as in Lemma \ref{existlem} and Proposition \ref{first1}. Especially, if $\psi \equiv 1$, it reduces to the framework of \cite{St99}. But considering a future goal in Theorem \ref{helholmop}, where we obtain $\rho \in H_{loc}^{1,p}(\R^d) \cap C_{loc}^{0,1-d/p}(\R^d)$ with $\rho(x)>0$ for all $x\in \R^d$, which in particular means that $\rho \in L_{loc}^{\infty}(\R^d)$ and $\frac{1}{\rho} \in L^{\infty}_{loc}(\R^d)$, we maintain our present assumptions in Section \ref{4.1ok} because it makes the arguments in the proofs more simple.
\end{rem}

\begin{rem} \label{remimpco}
Let $V$ be a bounded open subset of $\R^d$. Define 
$$
{L^{*\,V}} u:= L^{0,V}+\langle -\mathbf{B}, \nabla u \rangle, \quad u \in D(L^{0,V})_b.
$$
Note that $-\mathbf{B}$ has the same structural properties as $\mathbf{B}$ since \eqref{divfree} and \eqref{adddivassump} hold. Thus Proposition \ref{first1} holds equally with $\mathbf{B}$ replaced by $-\mathbf{B}$. In particular, there exists a  closed extension $(\overline{L}^{*\,V}, D(\overline{L}^{*\,V}))$ of $({L^{*\,V}}, D(L^{0,V})_b)$ on $L^1(V, \mu)$, which generates a sub-Markovian $C_0$-resolvent of contractions $(\overline{G}_{\alpha}^{*\,V} )$ on $L^1(V, \mu)$ and
\begin{equation*}
\mathcal{E}^{0,V}(u,v) + \int_{V} \langle \mathbf{B}, \nabla u \rangle v d \mu= - \int_{V} \overline{L}^{*\,V} u\, v d\mu, \quad u \in D(\overline{L}^{*\,V} )_b, \; v \in \widehat{H}^{1,2}_0(V, \mu).
\end{equation*}
Let $(L^{*\,V}, D(L^{*\,V}))$ be the part of $(\overline{L}^{*\,V}, D(\overline{L}^{*\,V}))$ on $L^2(V, \mu)$ and $(L^V, D(L^V))$ be the part of  $(\overline{L}^V, D(\overline{L}^V))$ on $L^2(V, \mu)$. Then for any $u \in D(L^V)_b$, $v \in D(L^{*\,V})_b$
\begin{eqnarray}
-\int_{V} L^V u\cdot v d \mu &=& \mathcal{E}^{0,V}(u,v) - \int_{V} \langle \mathbf{B}, \nabla u \rangle vd \mu \nonumber \\
&=& \mathcal{E}^{0,V}(v,u) + \int_{V} \langle \mathbf{B}, \nabla v \rangle u d \mu  \nonumber\\
&=& - \int_{V}  L^{*\,V} v \cdot ud \mu \label{fewstard}
\end{eqnarray}
Let $(G_{\alpha}^V)_{\alpha>0}$ and $({G^*_{\alpha}}^V)_{\alpha>0}$ be the resolvent associated to $(L^V, D(L^V))$, $({L^*}^V, D({L^*}^V))$ on $L^2(V, \mu)$, respectively. Then for any $f, g \in L^2(V, \mu) \cap L^{\infty}(V, \mu)$,
\begin{eqnarray}
\int_{V} G_{\alpha}^V f \cdot g d\mu  &=& \int_{V} G_{\alpha}^V f \cdot (\alpha - {L^{*}}^V ) {G^*_{\alpha}}^V g \,d\mu \nonumber  \\
&\underset{\eqref{fewstard}}{=}& \int_{V} (\alpha-L^V) G^V_{\alpha} f \cdot {G^*_{\alpha}}^V g \,d\mu \nonumber  \\
&=& \int_{V} f \cdot {G_{\alpha}^{*}}^{V} g \,d\mu. \label{adjoinidko}
\end{eqnarray}
By denseness of $L^2(V, \mu) \cap L^{\infty}(V, \mu)$ in $L^2(V,\mu)$, \eqref{adjoinidko} extends to all $f,g \in L^2(V, \mu)$.  Thus for each $\alpha>0$,\; ${G_{\alpha}^*}^V$ is the adjoint operator of $G_{\alpha}^V$ on $L^2(V, \mu)$.
\end{rem}
Now let $V$ be a bounded open subset of $\R^d$.
Denote by $(\overline{G}_{\alpha}^V)_{\alpha>0}$ the resolvent associated with $(\overline{L}^V, D(\overline{L}^V))$ on $L^1(V, \mu)$. Then $(\overline{G}^V_{\alpha})_{\alpha>0}$ can be extended  to $L^1(\R^d, \mu)$ by
\begin{equation} \label{extdefpj}
\overline{G}_{\alpha}^{V} f := 
\; \; \left\{\begin{matrix}
\overline{G}^V_{\alpha} (f 1_V) \quad \text{ on }\; V \\ 
\quad 0 \qquad \; \quad \text{ on } \, \R^d \setminus V,
\end{matrix}\right. \qquad
f \in L^1(\R^d, \mu),
\end{equation}
Let $g \in L^1(\R^d, \mu)_b$. Then $\overline{G}^V_{\alpha}(g1_V)\in D(\overline{L}^V)_b \subset \widehat{H}^{1,2}_0(V, \mu)$, hence $\overline{G}^V_{\alpha} g \in \widehat{H}^{1,2}_0(V, \mu)$.\\
Note that if $u \in D(\mathcal{E}^{0,V})$, then by definition it holds $u \in D(\mathcal{E}^0)$ and $\mathcal{E}^{0,V}(u,u) = \mathcal{E}^{0}(u,u)$. Therefore we obtain
\begin{equation} \label{fkppmeew}
\mathcal{E}^0(\overline{G}_{\alpha}^{V_n}g,  \, \overline{G}_{\alpha}^{V_n}g ) = \mathcal{E}^{0,V_n}\big(\overline{G}_{\alpha}^{V_n}(g1_{V_n}),  \, \overline{G}_{\alpha}^{V_n} (g1_{V_n}) \big).
\end{equation}

\begin{lem} \label{increlem}
Let $V_1$, $V_2$ be bounded open subsets of $\R^d$ and $\overline{V}_1 \subset V_2$. Let $u \in L^1(\R^d, \mu)$, $u \geq 0$, and $\alpha>0$. Then $\overline{G}^{V_1}_{\alpha} u \leq \overline{G}^{V_2}_{\alpha} u$.
\end{lem}
\begin{proof}
Using the denseness in $L^1(\R^d, \mu)$, we may assume $u \in L^1(\R^d, \mu)_b$. Let $w_{\alpha}:= \overline{G}_{\alpha}^{V_1} u- \overline{G}^{V_2}_{\alpha} u$.  Then clearly $w_{\alpha} \in \widehat{H}_0^{1,2}(V_2, \mu)$. Observe that $w^+_{\alpha} \leq \overline{G}_{\alpha}^{V_1}u$ on $\R^d$, so that $w^+_{\alpha} \in \widehat{H}^{1,2}_0(V_1, \mu)$ by Lemma \ref{tecpjp}.
By \cite[Theorem 4.4 (iii)]{EG15}, we obtain
\begin{equation} \label{evpwoek}
\int_{V_2} \langle \mathbf{B}, \nabla w_{\alpha} \rangle w_{\alpha}^+ d \mu = \int_{V_2} \langle \mathbf{B}, \nabla w_{\alpha}^+ \rangle w^{+}_{\alpha}d \mu = 0.
\end{equation}
Since $\mathcal{E}^{0,V_2}$ is a symmetric Dirichlet form, $\mathcal{E}^{0,V_2}(w_{\alpha}^{-}, w^+_{\alpha}) = \mathcal{E}^{0,V_2}(w_{\alpha}^{+}, w^-_{\alpha}) \leq 0$.
Therefore
\begin{eqnarray*}
\mathcal{E}^{0,V_2}( w^+_{\alpha},  w^+_{\alpha}) &\leq& \mathcal{E}^{0,V_2}_{\alpha}( w_{\alpha},  w^+_{\alpha}) - \int_{V_2} \langle \mathbf{B}, \nabla w_{\alpha} \rangle w_{\alpha}^+ d \mu \\
&\leq& \Big( \mathcal{E}^{0,V_1}_{\alpha}( \overline{G}_{\alpha}^{V_1} u,  w^+_{\alpha}) - \int_{V_1} \langle \mathbf{B}, \nabla \overline{G}_{\alpha}^{V_1}u \rangle w_{\alpha}^+ d \mu \Big) \\
&& \qquad \qquad - \Big(  \mathcal{E}^{0,V_2}_{\alpha}( \overline{G}_{\alpha}^{V_2} u,  w^+_{\alpha}) - \int_{V_2} \langle \mathbf{B}, \nabla \overline{G}_{\alpha}^{V_2} u \rangle w_{\alpha}^+ d \mu \Big) \\
&\leq& \int_{V_1}  (\alpha- \overline{L}^{V_1}) \overline{G}_{\alpha}^{V_1} u \, w^{+}_{\alpha} d \mu - \int_{V_2}  (\alpha- \overline{L}^{V_2}) \overline{G}_{\alpha}^{V_2} u \, w^{+}_{\alpha} d \mu \\
&=& \int_{V_1} u w^{+}_{\alpha} d \mu -\int_{V_2} u w^{+}_{\alpha} d \mu =0.
\end{eqnarray*}
Therefore $w_{\alpha}^+= 0$ in $\R^d$, hence $\overline{G}_{\alpha}^{V_1} u \leq   \overline{G}_{\alpha}^{V_2} u$  on $\R^d$.
\end{proof}

\begin{rem}
Since $w_{\alpha}, w_{\alpha}^+ \in H^{1,2}(V_2)$ in Lemma \ref{increlem}, we could directly get \eqref{evpwoek} using \cite[Theorem 4.4 (iii)]{EG15}. However, in the general situation as in Remark \ref{gensitce}, if $\rho$ is not bounded below by a strictly positive constant, then $w_{\alpha}, w_{\alpha}^+$ may not be contained in $H^{1,2}(V_2)$. In that case by Lemma \ref{bddapxlem}, we can take a sequence $(f_n)_{n \geq 1} \subset C_0^{\infty}(V_2)$ such that $\sup_{n \geq 1} \|f_n\|_{L^{\infty}(V_2)} \leq \|w_{\alpha}\|_{L^{\infty}(V_2)}$ and
\begin{eqnarray*}
&&\lim_{n \rightarrow \infty} f_n = w_{\alpha} \; \text{ in } D(\mathcal{E}^{0,V_2}), \;\quad \lim_{n \rightarrow \infty} f^+_n = w^+_{\alpha} \; \text{ weakly in } D(\mathcal{E}^{0,V_2}), \;\quad \\
&& \quad \lim_{n \rightarrow \infty} f_n^+ = w_{\alpha}^+ \quad \; \mu \text{-a.e. in } V_2.
\end{eqnarray*}
By \cite[Theorem 4.4 (iii)]{EG15}
\begin{equation} \label{appdivme}
\int_{V_2} \langle \mathbf{B}, \nabla f_n \rangle f_n^+ d \mu = \int_{V_2} \langle \mathbf{B}, \nabla f_{n}^+ \rangle f^{+}_{n}d \mu,
\end{equation}
and %Observe that in the right hand side of \eqref{appdivme}
\begin{eqnarray*}
&&\big| \int_{V_2} \langle \mathbf{B}, \nabla w_{\alpha}^+ \rangle w_{\alpha}^{+}d\mu  - \int_{V_2} \langle \mathbf{B}, \nabla f_{n}^+ \rangle f^{+}_{n}d\mu \big|  \\
&\leq& \underbrace{ \big| \int_{V_2 }  \langle \mathbf{B}, \nabla (w_{\alpha}^{+}-f^+_n )\rangle w_{\alpha}^{+} d \mu \big|}_{:=I_n} + \underbrace{\big| \int_{V_2} \langle \mathbf{B}, \nabla f_{n}^+ \rangle (w_{\alpha}^{+}- f^{+}_{n}) d\mu   \big|}_{=:J_n}.
\end{eqnarray*}
Since $\lim_{n \rightarrow \infty} f^+_n = w^+_{\alpha} \; \text{ weakly in } D(\mathcal{E}^{0,V_2})$, we have $\lim_{n \rightarrow \infty} I_n = 0$. Using the Cauchy-Schwarz inequality, it holds
\begin{eqnarray*}
J_n &\leq& \int_{V_2} \|\mathbf{B}\| \|\nabla f_n^+\|\, |w_{\alpha}^+ - f^+_n| d\mu \\
&\leq&  \big( \int_{V_2} \|\nabla f_n^+\|^2\, |w_{\alpha}^+ -f^+_n|\, \rho dx \big)^{1/2} \big(\int_{V_2}  \|\psi \mathbf{B}\|^2 |w_{\alpha}^+ - f^+_n| \,\rho dx \big)^{1/2} \\
&\leq& \sqrt{2}\|w_{\alpha}\|^{1/2}_{L^{\infty}(V)} \sup_{n \geq 1} \|f_n^+\|_{\widehat{H}^{1,2}_{0, \rho}(V_2, \mu)}
 \big(\int_{V_2}  \|\psi \mathbf{ B}\|^2 |w_{\alpha}^+ - f^+_n| d\mu \big)^{1/2} \\
&& \quad \longrightarrow 0 \;\; \text{ as } n \rightarrow \infty
\end{eqnarray*}
by Lebesgue's Theorem. Applying the same method for the left hand side of \eqref{appdivme}, we obtain
\begin{equation*} 
\int_{V_2} \langle \mathbf{B}, \nabla w_{\alpha} \rangle w_{\alpha}^+ d \mu = \int_{V_2} \langle \mathbf{B}, \nabla w_{\alpha}^+ \rangle w^+_{\alpha} d \mu.
\end{equation*}
\end{rem}
\text{}\\
By means of Proposition \ref{first1}, we can derive the following Theorem \ref{mainijcie}.
\begin{theo} \label{mainijcie}
There exists a  closed extension $(\overline{L}, D(\overline{L}))$ of $Lu:= L^0 u + \langle \mathbf{B}, \nabla u \rangle$,  \,$u \in D(L^0)_{0,b}$ on $L^1(\R^d, \mu)$ satisfying the following properties:

\begin{itemize}
\item[(a)] 
$(\overline{L}, D(\overline{L}))$ generates a sub-Markovian $C_0$-semigroup of contractions $(\overline{T}_t)_{t>0}$ on $L^1(\R^d, \mu)$.

\item[(b)]
Let $(U_n)_{n \geq 1}$ is a family of bounded open subsets of $\R^d$ satisfying $\overline{U}_n \subset U_{n+1}$ and $\R^d= \bigcup_{n \geq 1}  U_n$. Then \, $\lim_{n \rightarrow \infty} \overline{G}_{\alpha}^{U_n} f = (\alpha- \overline{L})^{-1} f$ \; in $L^1(\R^d, \mu)$, \; for all $f \in L^1(\R^d, \mu)$ and $\alpha>0$.

\item[(c)]
$D(\overline{L})_b \subset D(\mathcal{E}^0)$ and  $\text{ for all }  u \in D(\overline{L})_b,\, v \in \widehat{H}^{1,2}_0(\R^d, \mu)_{0,b}$ it holds 
\begin{eqnarray*}
\mathcal{E}^0(u,u)  &\leq& -\int_{\R^d} \overline{L}u\cdot u d \mu,  \qquad \qquad \\
\hspace{-5em}\mathcal{E}^{0}(u,v) - \int_{\R^d} \langle \mathbf{B}, \nabla u \rangle v d\mu &=& - \int_{\R^d} \overline{L} u\cdot v d \mu. \quad  
\end{eqnarray*}
\end{itemize}
\end{theo}
\begin{proof}
Let $f \in L^1(\R^d, \mu)$ with $f \geq 0$. Let $(V_n)_{n \geq 1}$ be a family of bounded open subsets of $\R^d$ satisfying $\overline{V}_n \subset V_{n+1}$ for all $n \in \N$.
Using Lemma \ref{increlem}, we can define for any $\alpha>0$
$$
\overline{G}_{\alpha} f := \lim_{n \rightarrow \infty} \overline{G}_{\alpha}^{V_n} f \quad \; \mu \text{-a.e. on } \R^d. 
$$
Using the $L^1$-contraction property, $\int_{\R^d} \alpha \overline{G}_{\alpha}^{V_n} f d\mu  = \int_{V_n} \alpha \overline{G}_{\alpha}^{V_n} (f1_{V_n}) d\mu \leq \int_{V_n} f d\mu \leq \int_{\R^d} f d\mu$. Thus by monotone integration, $\overline{G}_{\alpha} f \in L^1(\R^d, \mu)$ with
\begin{eqnarray*} \label{contracl1}
\int_{\R^d} \alpha \overline{G}_{\alpha} f d \mu \leq \int_{\R^d} f d\mu,
\end{eqnarray*}
and by Lebesgue's Theorem, $\lim_{n \rightarrow \infty} \overline{G}_{\alpha}^{V_n} f  =  \overline{G}_{\alpha} f$ \; in $L^1(\R^d, \mu)$.\\
For any $f \in L^1(\R^d, \mu)$, define $\overline{G}_{\alpha} f := \overline{G}_{\alpha} f^+ - \overline{G}_{\alpha} f^-$. Then $\alpha \overline{G}_{\alpha}$ is a contraction on $L^1(\R^d,\mu)$, since
$$
\int_{\R^d} | \alpha \overline{G}_{\alpha} f | d \mu \leq \int_{\R^d} (\alpha \overline{G}_{\alpha} f^+ + \alpha \overline{G}_{\alpha} f^- )d \mu \leq \int_{\R^d} f^+ d \mu + \int_{\R^d} f^- d \mu = \int_{\R^d} |f| d\mu.
$$
Thus
$$
\lim_{n \rightarrow \infty} \overline{G}_{\alpha}^{V_n} f  =  \overline{G}_{\alpha} f \; \text{ in  }L^1(\R^d, \mu), \quad \lim_{n \rightarrow \infty} \overline{G}_{\alpha}^{V_n} f = \overline{G}_{\alpha} f \quad \; \mu \text{-a.e. on } \R^d. 
$$ 
Clearly, $(\overline{G}_{\alpha})_{\alpha>0}$ is sub-Markovian, since $(\overline{G}_{\alpha}^{V_n})_{\alpha>0}$ is sub-Markovian on $L^1(V_n, \mu)$ for any $n \geq 1$. By the $L^1(\R^d, \mu)$-contraction property, for any $\alpha, \beta >0$
\begin{equation} \label{limitidt}
\lim_{n \rightarrow \infty} \| \overline{G}_{\alpha}^{V_n} \overline{G}_{\beta} f  - \overline{G}_{\alpha}^{V_n} \overline{G}_{\beta}^{V_n} f \|_{L^1(\R^d, \mu)} \leq \lim_{n \rightarrow \infty} \frac{1}{\alpha} \| \overline{G}_{\beta} f - \overline{G}^{V_n}_{\beta} f  \|_{L^1(\R^d, \mu)} = 0.
\end{equation}
Using \eqref{limitidt} and the resolvent equation for $(\overline{G}^{V_n}_{\alpha})_{\alpha>0}$, we obtain for any $\alpha, \beta>0$
\begin{eqnarray*}
(\beta- \alpha) \overline{G}_{\alpha} \overline{G}_{\beta} f &=& \lim_{n \rightarrow \infty} (\beta- \alpha) \overline{G}_{\alpha}^{V_n} \overline{G}_{\beta} f = \lim_{n \rightarrow \infty} (\beta-\alpha) \overline{G}_{\alpha}^{V_n} \overline{G}_{\beta}^{V_n} f  \\
&=& \lim_{n \rightarrow \infty} \overline{G}^{V_n}_{\alpha} f - \overline{G}_{\beta}^{V_n} f = \overline{G}_{\alpha} f - \overline{G}_{\beta} f \quad\text{ in } L^1(\R^d, \mu).
\end{eqnarray*}
Let $f \in L^1(\R^d, \mu)_b$ and $\alpha>0$.\; By \eqref{stamid}, $\overline{G}_{\alpha}^{V_n}(f 1_{V_n}) \in D(\overline{L}^V)_b \subset \widehat{H}^{1,2}_0(V_n, \mu)_b$. Using \eqref{fkppmeew},
\begin{eqnarray}
\mathcal{E}^0_{\alpha}(\overline{G}_{\alpha}^{V_n}f,  \, \overline{G}_{\alpha}^{V_n}f ) &=& \mathcal{E}^{0,V_n}_{\alpha}\big(\overline{G}_{\alpha}^{V_n}(f 1_{V_n}),  \, \overline{G}_{\alpha}^{V_n}(f 1_{V_n}) \big) \nonumber \\
&=& - \int_{V_n} \overline{L}^{V_n}\overline{G}_{\alpha}^{V_n}(f 1_{V_n}) \cdot \overline{G}_{\alpha}^{V_n}(f1_{V_n}) d \mu + \int_{V_n} \alpha \overline{G}_{\alpha}^{V_n}(f1_{V_n}) \cdot \overline{G}_{\alpha}^{V_n}(f1_{V_n}) d \mu \nonumber \\
&=& \int_{V_n}  (f1_{V_n}) \cdot \overline{G}_{\alpha}^{V_n}(f1_{V_n}) d \mu \nonumber \\
&\leq& \int_{\R^d} f  \cdot \overline{G}_{\alpha} f d \mu \label{boundwq} \\
&\leq& \frac{1}{\alpha} \|f\|_{L^{\infty}(\R^d, \mu)} \|f\|_{L^1(\R^d, \mu)}.  \nonumber
\end{eqnarray}
Observe that $\lim_{ n \rightarrow \infty} \overline{G}^{V_{n}}_{\alpha} f  = \overline{G}_{\alpha} f$  \; in $L^2(\R^d, \mu)$ by Lebesgue's Theorem. Thus by the Banach-Alaoglu Theorem, $\overline{G}_{\alpha} f \in D(\mathcal{E}^0)$ and
there exists subsequence of $(\overline{G}_{\alpha}^{V_n} f)_{n \geq 1}$, say again $(\overline{G}_{\alpha}^{V_n} f)_{n \geq 1}$, such that
\begin{equation} \label{weakcomp}
\lim_{n \rightarrow \infty} \overline{G}^{V_{n}}_{\alpha} f = \overline{G}_{\alpha} f \quad \text{ weakly in } D(\mathcal{E}^0).
\end{equation}
Using the property of weak convergence and \eqref{boundwq}
\begin{equation} \label{weakineqsom}
\mathcal{E}^{0}_{\alpha}(\overline{G}_{\alpha} f, \overline{G}_{\alpha} f ) \leq \liminf_{n \rightarrow \infty} \mathcal{E}^0_{\alpha} (\overline{G}^{V_{n}}_{\alpha} f,\overline{G}^{V_{n}}_{\alpha} f) \leq \int_{\R^d} f \overline{G}_{\alpha} f d\mu.
\end{equation}
Let $v \in \widehat{H}^{1,2}_0(\R^d, \mu)_{0,b}$. Then by Lemma \ref{applempn}, $v \in D(\mathcal{E}^0)$.  Using \eqref{weakcomp},
\begin{eqnarray} 
&&\qquad \mathcal{E}_{\alpha}^{0}(\overline{G}_{\alpha}f, v) - \int_{\R^d} \langle \mathbf{B}, \nabla \overline{G}_{\alpha} f \rangle v \, d \mu  \nonumber \\
&&=\lim_{n \rightarrow \infty}\big( \mathcal{E}_{\alpha}^{0}(\overline{G}^{V_{n}}_{\alpha}f, v) - \int_{\R^d} \langle \rho \psi \mathbf{B}, \nabla \overline{G}^{V_{n}}_{\alpha} f \rangle v \, dx \big)   \nonumber \\
&&=\lim_{n \rightarrow \infty}\big( \mathcal{E}_{\alpha}^{0,V_{n}}(\overline{G}^{V_{n}}_{\alpha}(f 1_{V_n}), v) - \int_{V_n} \big \langle \mathbf{B}, \nabla \overline{G}^{V_{n}}_{\alpha} (f1_{V_n}) \big \rangle\, v \, d \mu \big) \nonumber \\
&&= \lim_{n \rightarrow \infty} \int_{V_n} (\alpha-\overline{L}^{V_n}) \overline{G}^{V_n}_{\alpha} (f1_{V_n}) \cdot v d\mu = \lim_{n \rightarrow \infty} \int_{V_n} f v d\mu= \int_{\R^d} f v d\mu.  \label{stampale}
\end{eqnarray}
Let $u \in D(L^0)_{0,b}$ be given and take $j \in \N$ satisfying $\text{supp}\, u \subset V_j$. Then by Lemma \ref{applempn}, $u \in \widehat{H}^{1,2}_0(V_j, \mu)$. Observe that $\text{supp}\, (Lu) \subset V_j$ and for any $n \geq j$, $u1_{V_n} \in D(L^{0, V_n})_b$, $L^{V_n} (u1_{V_n})  = L u$ on  $V_n$, hence $\overline{G}^{V_n}_{\alpha} (\alpha-L) u  = u$ on $\R^d$. Letting $n \rightarrow \infty$ we have
\begin{equation} \label{extiden}
 u = \overline{G}_{\alpha} (\alpha-L) u.
\end{equation}
Note that
\begin{eqnarray}
\| \alpha \overline{G}_{\alpha} u  - u \|_{L^1(\R^d, \mu)} 
&=& \left \| \alpha \overline{G}_{\alpha} u - \overline{G}_{\alpha} (\alpha-L) u  \right \|_{L^1(\R^d, \mu)} \ =\  \left \| \overline{G}_{\alpha} L u \right \|_{L^1(\R^d, \mu)}  \nonumber \\
&\leq& \frac{1}{\alpha} \| Lu \|_{L^1(\R^d, \mu)} \longrightarrow 0 \;\; \text{ as } \alpha \rightarrow \infty. \label{stconr}
\end{eqnarray}
Since $C_0^{\infty}(\R^d) \subset D(L^0)_{0,b}$, \eqref{stconr} extends to all $u \in L^1(\R^d, \mu)$, which shows the strong continuity of $(\overline{G}_{\alpha})_{\alpha>0}$ on $L^1(\R^d, \mu)$.
Let $(\overline{L}, D(\overline{L}))$ be the generator of $(\overline{G}_{\alpha})_{\alpha>0}$. Then \eqref{extiden} implies $\overline{L} u = Lu$  for all $u \in D(L^0)_{0,b}$. Thus $(\overline{L}, D(\overline{L}))$ is a closed extension of $(L, D(L^0)_{0,b})$ on $L^1(\R^d, \mu)$. By the Hille-Yosida Theorem, $(\overline{L}, D(\overline{L}))$ generates a $C_0$-semigroup of contractions $(\overline{T}_t)_{t> 0}$ on $L^1(\R^d, \mu)$. \\
Since $\overline{T}_t u = \lim_{\alpha \rightarrow \infty} \exp \Big( t\alpha (\alpha \overline{G}_{\alpha}u - u ) \Big) $ in $L^1(\R^d, \mu)$,  $(\overline{T}_t)_{t>0}$ is also sub-Markovian, hence (a) is proved.

Next we will show (b). Let $(U_n)_{n \geq 1}$ be a family of bounded open subsets of $\R^d$ such that $\overline{U}_{n} \subset U_{n+1}$ for all $n \in \N$ and $\R^d  = \bigcup_{n \geq 1} U_n$. Let $f \in L^1(\R^d, \mu)$ with $f \geq 0$. By the compactness of $\overline{V}_n$ in $\R^d$, there exists $n_0 \in \N$ such that $\overline{V}_n \subset U_{n_0}$, so that $\overline{G}^{V_n} f \leq \overline{G}^{U_{n_0}} f \leq \lim_{n \rightarrow \infty} \overline{G}^{U_n}_{\alpha}f$. 
Letting $n \rightarrow \infty$, we obtain $\overline{G}_{\alpha} f \leq \lim_{n \rightarrow \infty} \overline{G}^{U_n}_{\alpha}f$. Similarly we have $\lim_{n \rightarrow \infty} \overline{G}^{U_n}_{\alpha}f \leq \overline{G}_{\alpha} f$, which shows (b).

Finally we will show (c). Let $u \in D(\overline{L})_b$ be given. Then by \eqref{weakcomp}, $\alpha \overline{G}_{\alpha} u \in D(\mathcal{E}^0)$ and by \eqref{weakineqsom}
\begin{eqnarray}
\mathcal{E}^{0}(\alpha \overline{G}_{\alpha} u, \alpha \overline{G}_{\alpha} u  ) &\leq& \int_{\R^d} \alpha u \cdot \alpha  \overline{G}_{\alpha} u d \mu - \alpha \int_{\R^d} \alpha \overline{G}_{\alpha} u \cdot \alpha \overline{G}_{\alpha} u d \mu  \nonumber \\
&=& \int_{\R^d} \alpha \big( u - \alpha \overline{G}_{\alpha} u  \big) \cdot \alpha \overline{G}_{\alpha} u \, d \mu  \nonumber \\
&=& \int_{\R^d} - \alpha \overline{L}\,  \overline{G}_{\alpha} u \cdot \alpha \overline{G}_{\alpha} u \,d \mu \nonumber \\
&=& \int_{\R^d} - \alpha \overline{G}_{\alpha} \overline{L} u \cdot \alpha \overline{G}_{\alpha} u d\mu \label{bddjwnvr14} \\
&\leq& \| \overline{L} u \|_{L^1(\R^d, \mu)} \|u \|_{L^{\infty}(\R^d, \mu)}. \nonumber
\end{eqnarray}
Therefore $\sup_{\alpha>0} \mathcal{E}^{0}(\alpha \overline{G}_{\alpha}u, \alpha \overline{G}_{\alpha}u ) < \infty$. By Banach-Alaoglu theorem, there exists a subsequence of $(\alpha \overline{G}_{\alpha}u)_{\alpha>0}$, say again $(\alpha \overline{G}_{\alpha}u)_{\alpha>0}$, such that $u \in D(\mathcal{E}^0)$ and $\lim_{\alpha \rightarrow \infty} \alpha \overline{G}_{\alpha} u = u$ weakly in $D(\mathcal{E}^0)$. Moreover by the property of weak convergence, \eqref{bddjwnvr14} and Lebesgue's Theorem,
\begin{eqnarray*}
\mathcal{E}^{0}(u,u) &\leq& \liminf_{\alpha \rightarrow \infty} \mathcal{E}^{0}(\alpha \overline{G}_{\alpha} u, \alpha \overline{G}_{\alpha} u  )  \leq \liminf_{\alpha \rightarrow \infty} \big(-\int_{\R^d} \alpha \overline{G}_{\alpha} \overline{L} u \cdot \alpha \overline{G}_{\alpha} u d \mu\big)\\
&=& -\int_{\R^d} \overline{L}u \, u d \mu. \;\; \;
\end{eqnarray*}
\iffalse
$$
L^{V_n} u =
\left\{\begin{matrix}
Lu  &\text{ on } V_i \\ 
0 & \quad \text{ on } \R^d \setminus V_i 
\end{matrix}\right.
$$
\fi
If $v \in \widehat{H}^{1,2}(\R^d, \mu)_{0,b}$, then by \eqref{stampale}
\begin{eqnarray*}
&&\mathcal{E}^0(u,v) - \int_{\R^d} \langle \mathbf{B}, \nabla u \rangle v d \mu = \lim_{\alpha \rightarrow \infty} \big( \mathcal{E}^0(\alpha \overline{G}_{\alpha} u, v)- \int_{\R^d} \langle  \rho \psi \mathbf{B}, \nabla \alpha \overline{G}_{\alpha} u   \rangle v dx  \big) \\
&& =  \lim_{\alpha \rightarrow \infty} \big( \mathcal{E}_{\alpha}^0(\alpha \overline{G}_{\alpha} u, v)- \int_{\R^d} \langle  \mathbf{B}, \nabla \alpha \overline{G}_{\alpha} u   \rangle v d\mu   - \alpha \int_{\R^d} \alpha \overline{G}_{\alpha} u \cdot v d \mu \big) \\
&& = \lim_{\alpha \rightarrow \infty} \int_{\R^d} \alpha(u- \alpha \overline{G}_{\alpha} u) v d \mu =  \lim_{\alpha \rightarrow \infty} \int_{\R^d} -\alpha \overline{G}_{\alpha} \overline{L} u \cdot v d \mu  = - \int_{\R^d} \overline{L} u \cdot v d \mu,
\end{eqnarray*}
as desired. \vspace{-0.7em}
\end{proof}

\begin{rem} \label{resestdual}
In the same way as in Theorem \ref{mainijcie}, one can construct an $L^1(\R^d, \mu)$ closed extension $(\overline{L}^*, D(\overline{L}^*))$ of $L^0 u + \langle -\mathbf{B}, \nabla u \rangle$, $u \in D(L^0)_{0,b}$ which generates a sub-Markovian  $C_0$-resolvent of contractions $(\overline{G}_{\alpha}^{*} )_{\alpha>0}$ an $L^1(\R^d, \mu)$. Let  $(U_n)_{n \geq 1}$ be as in Theorem \ref{mainijcie}(b). Observe that by Remark \ref{remimpco} 
\begin{equation} \label{localresid}
\int_{\R^d} \overline{G}_{\alpha}^{U_n} u \cdot v \, d \mu  = \int_{\R^d} u \cdot {\overline{G}_{\alpha}^{* \,U_n}} v \,d\mu, \;\; \text{ for all } u,v \in L^1(\R^d, \mu) \cap L^{\infty}(\R^d, \mu),
\end{equation}
where $(\overline{G}_{\alpha}^{* \,U_n})_{\alpha>0}$ is the resolvent associated to $(\overline{L}^{*\,U_n}, D(\overline{L}^{*\,U_n}))$ on $L^1(U_n, \mu)$, which is trivially extended to $\R^d$ as in \eqref{extdefpj}.
\;Letting $n \rightarrow \infty$ in \eqref{localresid}, 
\begin{equation*} \label{glbresid}
\int_{\R^d} \overline{G}_{\alpha} u \, v \, d \mu  = \int_{\R^d} u \, \overline{G}_{\alpha}^{*} v d\mu, \;\; \text{ for all } u,v \in L^1(\R^d, \mu) \cap L^{\infty}(\R^d, \mu). \quad 
\end{equation*}
\end{rem}
\bigskip
The following Theorem \ref{pojjjde} which shows that $D(\overline{L})_b$ is an algebra is one of the ingredients to construct a Hunt process corresponding to the strict capacity (see, {\bf SD3} in \cite{Tr5}). It will be used later. The proof of Theorem \ref{pojjjde} is based on \cite[Remark 1.7 (iii)]{St99}, but we include its proof checking in detail some approximation arguments.
\begin{lem} \label{pojjjde}
$D(\overline{L})_b$ is an algebra and $\overline{L}u^2 = 2u \overline{L}u +  \langle \widehat{A} \nabla u, \nabla u \rangle$ for any $u \in D(\overline{L})_b$.
\end{lem}
\begin{proof}
Let $u \in D(\overline{L})_b$. Since $D(\overline{L})_b$ is a linear space, it suffices to show $u^2 \in D(\overline{L})_b$. Let $(\overline{L}^*, D(\overline{L}^*))$, $(\overline{G}_{\alpha}^{*} )_{\alpha>0}$ be as in Remark \ref{resestdual} and set $g:=2u \overline{L}u + \langle \widehat{A} \nabla u, \nabla u \rangle$. \,If we can show
\begin{equation} \label{sufficeoce}
\int_{\R^d} (\overline{L}^*\, \overline{G}^*_1 h) \, u^2 d\mu  = \int_{\R^d} g\, \overline{G}^*_1 h \,d\mu, \;\; \text{ for all } h \in L^1(\R^d, \mu)_b,
\end{equation}
then 
\begin{eqnarray*}
\int_{\R^d} \overline{G}_1(u^2-g) \, h d\mu  &=& \int_{\R^d} (u^2-g) \overline{G}^*_{1} h\, d\mu \underset{\eqref{sufficeoce}}{=} \int_{\R^d} u^2  (  \overline{G}^*_{1} h - \overline{L}^* \,\overline{G}_{1}^* h  ) d\mu\\
&=& \int_{\R^d} u^2 h \,d \mu, \;\; \text{ for all } \; h \in L^1(\R^d, \mu)_b,
\end{eqnarray*}
hence $u^2  = \overline{G}_1(u^2-g) \in D(\overline{L})_b$  and $\overline{L}u^2= (1-\overline{L}) \overline{G}_1(g-u^2)- \overline{G}_1(g-u^2) = g-u^2+u^2=g$, as desired. \\
\text{}\\
{\bf Step 1: } To prove \eqref{sufficeoce}, first assume $u = \overline{G}_1 f$ for some $f \in L^1(\R^d, \mu)_b$. Fix $v =\overline{G}^*_1 h$ for some $h \in L^1(\R^d, \mu)_b$ with $h \geq 0$. Let $(U_n)_{n \geq 1}$ be as in Theorem \ref{mainijcie}(b) and $u_n:= \overline{G}^{U_n}_1 f$, \;$v_n:= \overline{G}^{*\, U_n}_{1} h$. By Proposition \ref{first1} and Theorem \ref{mainijcie}
\begin{eqnarray}
&&\int_{\R^d} (\overline{L}^{*\, U_n} v_n) \, u u_n d \mu \nonumber \\ 
&=& - \mathcal{E}^0(v_n,  u u_n) - \int_{\R^d} \langle \mathbf{B}, \nabla v_n \rangle u u_n d \mu, \qquad \text{( since $v_n \in D(\mathcal{E}^0)$ and $u u_n \in D(\mathcal{E}^0)$ )} \nonumber  \\
&=& - \frac12 \int_{\R^d} \langle \widehat{A} \nabla v_n, \nabla u \rangle u_n  d\mu- \frac12 \int_{\R^d} \langle \widehat{A} \nabla v_n, \nabla u_n \rangle u  d\mu  + \int_{\R^d} \langle \mathbf{B}, \nabla (u u_n) \rangle v_n d \mu  \nonumber \\
&=& - \frac12 \int_{\R^d} \langle \widehat{A} \nabla (v_n u_n), \nabla u    \rangle d\mu + \frac12 \int_{\R^d} \langle \widehat{A} \nabla u_n,  \nabla u \rangle v_n d\mu - \frac12\int_{\R^d} \langle \widehat{A} \nabla v_n, \nabla u_n \rangle u  d\mu \nonumber \\
&& \qquad +\int_{\R^d} \langle \mathbf{B}, \nabla u \rangle v_n u_n \,d\mu+ \int_{\R^d} \langle \mathbf{B}, \nabla u_n \rangle v_n u d\mu \nonumber \\
&=& - \frac12 \int_{\R^d} \langle \widehat{A} \nabla u, \nabla (v_n u_n)  \rangle d \mu +\int_{\R^d} \langle \mathbf{B}, \nabla u \rangle v_n u_n \,d\mu  + \frac12 \int_{\R^d} \langle \widehat{A} \nabla u_n, \nabla u \rangle v_n  d\mu  \nonumber \\
&& \qquad  - \frac12 \int_{\R^d} \langle \widehat{A} \nabla u_n, \nabla (v_n u)\rangle d \mu+ \int_{\R^d} \langle \mathbf{B}, \nabla u_n \rangle v_n u d\mu + \frac12 \int_{\R^d} \langle \widehat{A} \nabla u_n, \nabla u \rangle v_n d\mu \nonumber \\
&=& \int_{\R^d} \overline{L} u \cdot v_n u_n d\mu + \int_{\R^d} \overline{L}^{U_n} u_n \cdot v_n u d \mu + \int_{\R^d} \langle \widehat{A} \nabla u_n, \nabla u \rangle v_n d \mu. \label{intepce}
\end{eqnarray}
Observe that
\begin{eqnarray}
&&  \big |\int_{\R^d} \langle \widehat{A} \nabla u, \nabla u   \rangle v d\mu - \int_{\R^d} \langle \widehat{A} \nabla u_n, \nabla u \rangle v_n d \mu \big| \nonumber \\
&& \leq \underbrace{\big | \int_{\R^d} \langle \widehat{A} \nabla (u-u_n), \nabla u \rangle v   d\mu  \big|}_{=:I_n} + \underbrace{\big| \int_{\R^d}  \langle \widehat{A} \nabla u_n, \nabla u    \rangle \, (v-v_n) \, d\mu \big| }_{=:J_n} \label{weakconvlen} 
\end{eqnarray}
Since $\lim_{n \rightarrow \infty} u_n = u$ \, weakly in $D(\mathcal{E}^0)$ and $v$ is bounded on $\R^d$, $\lim_{n \rightarrow \infty} I_n =0$.
Note that $v_n = \overline{G}^{*\,U_n}_1 h \leq \overline{G}^{*}_1 h=v$,\; $\sup_{n \in \N}\mathcal{E}^0(u_n, u_n)<\infty$,\; $|v_n| \leq |v| \in L^{\infty}(\R^d, \mu)$ and 
$$
\lim_{n \rightarrow \infty}u_n = u \; \text{ $\mu$-a.e. on $\R^d$,  }
$$ 
hence we obtain by the Cauchy–Schwarz inequality,
\begin{eqnarray*}
J_n&&\leq   \big(  \int_{\R^d}  \langle \widehat{A} \nabla u_n, \nabla u_n    \rangle \, (v-v_n) \, d\mu \big)^{1/2} \big(  \int_{\R^d}  \langle \widehat{A} \nabla u, \nabla u    \rangle \, (v-v_n) \, d\mu \big)^{1/2} \\
&&\leq \sqrt{2}\|v\|_{L^{\infty}(\R^d, \mu)}^{1/2} \sup_{n \geq 1}\mathcal{E}^0(u_n, u_n)^{1/2} \big(  \int_{\R^d}  \langle \widehat{A} \nabla u, \nabla u    \rangle \, (v-v_n) \, d\mu \big)^{1/2} \\
&&\;\; \longrightarrow 0 \quad \text{ as } n \rightarrow \infty,
\end{eqnarray*}
where the latter convergence to zero followed by Lebesgue's Theorem for which we use
$$
\big|\langle \widehat{A} \nabla u, \nabla u    \rangle \, (v-v_n)  \big| \leq 2\|v\|_{L^{\infty}(\R^d, \mu)}  \langle \widehat{A} \nabla u, \nabla u    \rangle \in L^1(\R^d, \mu),  \; \text{\,  $\mu$-a.e. on $\R^d$  }
$$
and
$$
\lim_{n \rightarrow \infty}  \langle \widehat{A} \nabla u, \nabla u    \rangle \, (v-v_n) = 0, \;\; \text{ $\mu$-a.e. on $\R^d$}.
$$
Therefore it follows by \eqref{weakconvlen} that
\begin{eqnarray} \label{limitcje}
\lim_{n \rightarrow \infty} \int_{\R^d} \langle \widehat{A} \nabla u_n, \nabla u \rangle v_n d \mu =\int_{\R^d} \langle \widehat{A} \nabla u, \nabla u \rangle v d \mu.
\end{eqnarray}
By \eqref{intepce}, \eqref{limitcje} and Lebesgue's Theorem
\begin{eqnarray}
&&\int_{\R^d} \overline{L}^{*}v\cdot u^2 d\mu  \nonumber \\
&= &\int_{\R^d} \big( \overline{G}^{*}_1 h - h\big) u u_n d\mu \nonumber \\
&=& \lim_{n \rightarrow \infty} \int_{\R^d} \big(\overline{G}^{* \, U_n}_1 h - h \big) u u_n \,d\mu \nonumber \\
&=& \lim_{n \rightarrow \infty} \int_{\R^d} (\overline{L}^{*\,U_n} v_n)  u u_n d\mu \nonumber  \\
&\underset{\eqref{intepce}}{=}&\lim_{n \rightarrow \infty} \int_{\R^d} \overline{L} u \cdot v_n u_n d\mu + \int_{\R^d} (\overline{G}_1^{U_n} f -f) \cdot v_n u d \mu +\int_{\R^d} \langle \widehat{A} \nabla u_n, \nabla u \rangle v_n d \mu  \nonumber \\
&\underset{\eqref{limitcje}}{=}&  \int_{\R^d} \overline{L} u \cdot v u d\mu + \int_{\R^d} \overline{L}u\cdot vu d \mu + \int_{\R^d} \langle \widehat{A} \nabla u, \nabla u \rangle v d \mu  \nonumber \\
&=& \int_{\R^d} g v d\mu. \label{gvidendt}
\end{eqnarray}
In the case of general $h \in L^1(\R^d, \mu)_b$, we also obtain \eqref{gvidendt} using $h=h^+-h^-$ and linearity.\\
\centerline{}
{\bf Step 2: } Let $ u \in D(\overline{L})_b$ be arbitrary. Set
$$
g_{\alpha}:= 2(\alpha \overline{G}_{\alpha} u) \overline{L} (\alpha \overline{G}_{\alpha} u) +  \langle \widehat{A} \nabla \alpha  \overline{G}_{\alpha} u, \nabla \alpha \overline{G}_{\alpha} u \rangle, \quad \alpha>0.
$$
By Theorem \ref{mainijcie}(c),
\begin{eqnarray*}
\mathcal{E}^0(\alpha \overline{G}_{\alpha} u - u, \alpha \overline{G}_{\alpha} u - u ) &\leq& -\int_{\R^d} \overline{L} (\alpha \overline{G}_{\alpha} u - u) \cdot (\alpha \overline{G}_{\alpha} u - u) d\mu \\
&\leq& 2\|u\|_{L^{\infty}(\R^d,\mu)} \|\alpha \overline{G}_{\alpha} \overline{L}u -\overline{L}u \|_{L^1(\R^d, \mu)} \\
&& \;\longrightarrow 0 \quad \text{ as } \alpha \rightarrow \infty,
\end{eqnarray*}
hence $\lim_{\alpha \rightarrow \infty} g_{\alpha} = g$ \; in $L^1(\R^d, \mu)$. Observe that by the resolvent equation
\begin{eqnarray*}
\overline{G}_{\alpha} u =\overline{G}_1\big( (1-\alpha) \overline{G}_{\alpha} u + u  \big)
\end{eqnarray*}
and  $(1-\alpha) \overline{G}_{\alpha} u +  u \in L^1(\R^d, \mu)_b$. Replacing $u$ in \eqref{gvidendt} with $\alpha \overline{G}_{\alpha} u$
\begin{eqnarray*}
\int_{\R^d} \overline{L}^{*}v\, \big(\alpha \overline{G}_{\alpha}u \big)^2 d\mu=\int_{\R^d} g_{\alpha} v d\mu.
\end{eqnarray*}
Letting $\alpha \rightarrow \infty$, we finally obtain by Lebesgue's Theorem
$$
\int_{\R^d} \overline{L}^{*}v\cdot u^2 d\mu=\int_{\R^d} g v d\mu,
$$
so that our assertion holds.
\end{proof}

\subsection{Existence of a pre-invariant measure and strong Feller properties} \label{fewfore}

Here we state some conditions which will be used as our assumptions.
\begin{itemize}
\item[\bf (A1)]
$p>d$ is fixed and $A=(a_{ij})_{1 \leq i,j \leq d}$ is a symmetric matrix of functions which is  locally uniformly strictly elliptic on $\R^d$ such that $a_{ij} \in H_{loc}^{1,p}(\R^d) \cap C_{loc}^{0,1-d/p}(\R^d)$ for all $1 \leq i,j \leq d$.\, $\psi \in L_{loc}^1(\R^d)$ is a positive function such that $\frac{1}{\psi} \in L_{loc}^{\infty}(\R^d)$ and  $\mathbf{G}$ is a Borel measurable vector field on $\R^d$ satisfying $\psi \mathbf{G} \in L^p_{loc}(\R^d, \R^d)$.

\item[\bf (A2)]
$\psi \in L_{loc}^q(\R^d)$ with $q\in (\frac{d}{2},\infty]$. Fix $s\in (\frac{d}{2},\infty]$ such that $\frac{1}{q}+ \frac{1}{s}< \frac{2}{d}$.

\item[\bf (A3)]
$q \in [\frac{p}{2} \vee 2,\infty]$.
\end{itemize}

\begin{theo} \label{helholmop}
Under the assumption {\bf(A1)}, there exists $\rho \in H_{loc}^{1,p}(\R^d) \cap C_{loc}^{0,1-d/p}(\R^d)$ satisfying $\rho(x)>0$ for all $x \in \R^d$ such that 
\begin{equation} \label{helm}
\int_{\R^d} \langle \mathbf{G}-\beta^{\rho, A, \psi}, \nabla \varphi \rangle \rho \psi dx =0, \quad \text{for all } \varphi \in C_0^{\infty}(\R^d).
\end{equation}
Moreover $\rho \psi  \,\mathbf{B} \in L_{loc}^p(\R^d, \R^d)$, where $\mathbf{B}:= \mathbf{G}-\beta^{\rho, A, \psi}$.
\end{theo}
\begin{proof}
By \cite[Theorem 3.6]{LT19}, there exists $\rho \in H_{loc}^{1,p}(\R^d) \cap C_{loc}^{0,1-d/p}(\R^d)$ satisfying $\rho(x)>0$ for all $x \in \R^d$ such that 
$$
\int_{\R^d}\big \langle \frac12 A \nabla \rho + (\frac12 \nabla A - \psi \mathbf{G}) \rho, \nabla\varphi \big \rangle dx =0, \quad \text{ for all } \varphi \in C_0^{\infty}(\R^d),
$$
hence
$$
\int_{\R^d} \big \langle \mathbf{G} - \frac{\nabla A}{2\psi} - \frac{A \nabla \rho}{2 \rho \psi} , \nabla \varphi \big \rangle \rho \psi dx =0, \quad \text{ for all } \varphi \in C_0^{\infty}(\R^d),
$$
and moreover
$$
\rho \psi\, \mathbf{B} = \rho \psi \mathbf{G}-\frac{\rho}{2} \nabla A-\frac{A \nabla \rho}{2} \in L_{loc}^p(\R^d, \R^d).
$$
\end{proof}\\ \\
{\bf From now on}, we assume that {\bf (A1)} holds and fix $A$, $\psi$, $\rho$, $\mathbf{B}$ as in Theorem \ref{helholmop}. Then $A$, $\psi$, $\rho$, $\mathbf{B}$ satisfy all assumptions of Section \ref{4.1ok}. Set as in Section \ref{4.1ok} $\mu:=\rho \psi \,dx$, $\widehat{A}:= \frac{1}{\psi} A$, $\widehat{\rho}:= \rho \psi$, $\widehat{a}_{ij}=\frac{1}{\psi}a_{ij}$ for all $1 \leq i,j \leq d$.

\begin{rem} \label{repremeasno}
If $\psi \in H^{1,2}(V) \cap L^{\infty}(V)$ \,for some bounded open set $V$ in $\R^d$, then by the chain and product rules for weakly differentiable functions,
$$
\frac12 \nabla \widehat{A} =  \frac{\nabla A }{2\psi} + \frac{-A \nabla \psi}{2\psi^2}, \qquad \frac{\widehat{A} \nabla \widehat{\rho}}{2 \widehat{\rho}} = \frac{A \nabla \psi}{2\psi^2}+\frac{A \nabla \rho}{2 \rho \psi} \;\;\; \text{  on } V.
$$
Set $\displaystyle \beta^{\widehat{\rho}, \widehat{A}} := \frac12 \nabla \widehat{A}+\frac{\widehat{A}\nabla \widehat{\rho}}{2\widehat{\rho}}$\, on $V$. Then it holds $\beta^{\widehat{\rho}, \widehat{A}} = \beta^{\rho. A, \psi}$\, {\rm({\it a.e.})} on $V$. If we assume $\psi \in H^{1,p}(V)$, then it holds
$$
\widehat{\mathbf{F}}:=\frac{1}{2}\nabla \widehat{A}+\mathbf{G}-2\beta^{\widehat{\rho}, \widehat{A}}\in L^{p}(V,\R^d).
$$
\end{rem}
By Theorem \ref{mainijcie} there exists a closed extension $(\overline{L}, D(\overline{L}))$ of
$$
Lf = L^0 f + \langle \mathbf{B}, \nabla f \rangle, \;\quad  f \in D(L^0)_{0,b},
$$
on $L^1(\R^d, \mu)$ which generates the sub-Markovian $C_0$-semigroup of contractions $(\overline{T}_t)_{t>0}$ on $L^1(\R^d, \mu)$.  Restricting $(\overline{T}_t)_{t> 0}$ to $L^1(\R^d, \mu)_b$, it is well-known by Riesz-Thorin interpolation that $(\overline{T}_t)_{t> 0}$ can be extended to a sub-Markovian $C_0$-semigroup of contractions $(T_t)_{t>0}$ on each $L^r(\R^d, \mu)$, $r\in [1,\infty)$. Denote by $(L_r, D(L_r))$ the corresponding closed generator with graph norm
$$
\|f\|_{D(L_r)}:=\|f\|_{L^r(\R^d,\mu)}+ \|L_r f\|_{L^r(\R^d, \mu)},
$$
and by $(G_{\alpha})_{\alpha>0}$ the corresponding resolvent. Also $(T_t)_{t>0}$ and $(G_{\alpha})_{\alpha>0}$ can be uniquely defined on $L^{\infty}(\R^d,\mu)$, but are no longer strongly continuous there.\\ \\
For $f \in C_0^{\infty}(\R^d)$, we have
$$
Lf = L^0 f + \langle \mathbf{B}, \nabla f \rangle=\frac12 \text{trace}(\widehat{A} \nabla ^2 f) + \langle \mathbf{G}, \nabla f \rangle.
$$
Define
\begin{eqnarray*}
L^*f : &=& L^0 f- \langle \mathbf{B}, \nabla f \rangle= \frac12\text{trace}(\widehat{A}\nabla^2 f)+\langle \mathbf{G}^*, \nabla f \rangle,
\end{eqnarray*}
with
$$
\mathbf{G}^*:=(g^*_1, \cdots, g^*_d)=2\beta^{\rho, A, \psi}-\mathbf{G} = \beta^{\rho, A, \psi}-\mathbf{B} \in L^{2}_{loc}(\R^d, \R^d, \mu).
$$ 
We see that $L$ and $L^*$ have the same structural properties, i.e. they are given as the sum of a symmetric second order elliptic differential operator $L^0$ and a divergence free first order perturbation\,$\langle \mathbf{B}, \nabla \cdot \,\rangle$ \,or $\langle -\mathbf{B}, \nabla \cdot \, \rangle$, respectively, with same integrability condition\,$\rho \psi \mathbf{B} \in L_{loc}^p(\R^d, \R^d)$. Therefore all what will be derived below for $L$ will hold analogously for $L^*$. Denote by $(L^*_r, D(L^*_r))$ the operators corresponding to $L^*$ for the co-generator on $L^r(\R^d,\mu)$, $r\in [1,\infty)$, $(T^*_t)_{t>0}$ for the co-semigroup, $(G^*_{\alpha})_{\alpha>0}$ for the co-resolvent. As in  \cite[Section 3]{St99}, we obtain a corresponding bilinear form with domain 
$D(L_2) \times L^2(\R^d,\mu) \cup L^2(\R^d,\mu) \times D(L^*_2)$ by
\begin{equation} \label{fwpeokce}
{\mathcal{E}}(f,g):= \left\{ \begin{array}{r@{\quad\quad}l}
  -\int_{\R^d} L_2 f \cdot g \,d\mu & \mbox{ for}\ f\in D(L_2), \ g\in L^2(\R^d,\mu),  \\ 
            -\int_{\R^d} f\cdot L^*_2 g \,d\mu  & \mbox{ for}\ f\in L^2(\R^d,\mu), \ g\in D(L^*_2). \end{array} \right .
\end{equation}
 $\mathcal{E}$ is called the {\it generalized Dirichlet form associated with} $(L_2,D(L_2))$. 

\begin{theo} \label{rescondojc}
Assume {\bf (A1)}, {\bf(A2)} and let $f \in \cup_{r \in [s,\infty]} L^r(\R^d,\mu)$. Then $G_{\alpha} f$ has a locally H\"{o}lder continuous $\mu$-version $R_{\alpha} f$ on $\R^d$. Furthermore for any open balls $B$, $B'$ satisfying $\overline{B} \subset B'$, we have the following estimate
\begin{equation} \label{resestko}
\|  R_{\alpha} f \|_{C^{0, \gamma}(\overline{B})} \le c_2 \left (\| f \|_{L^s(B',\mu)} + \| G_{\alpha}f \|_{L^1(B',\mu)}\right ),
\end{equation}
where $c_2>0$, $\gamma \in (0,1)$ are constants which are independent of $f$.
\end{theo}
\begin{proof}
Let $f \in C_0^{\infty}(\R^d)$ and $\alpha>0$. Then by Theorem \ref{mainijcie}, $G_{\alpha} f  \in D(\overline{L})_b \subset D(\mathcal{E}^0)$ and
\begin{eqnarray}
&&\mathcal{E}^0(G_{\alpha} f, \varphi) - \int_{\R^d} \langle \mathbf{B}, \nabla G_{\alpha} f \rangle  \varphi d \mu \nonumber \\
&&\quad =  -\int_{\R^d} \left (\overline{L}\, \overline{G}_{\alpha} f \right) \varphi \,d \mu \nonumber  \\
&& \quad = \int_{\R^d} (f - \alpha G_{\alpha} f)\varphi \,d\mu, \quad \text{ for all } \varphi \in C_0^{\infty}(\R^d). \label{dfeefes}
\end{eqnarray}
Thus \eqref{dfeefes} implies
\begin{eqnarray}
&&\int_{\R^d} \big \langle \frac12 \rho A \nabla G_{\alpha}f, \nabla \varphi \big \rangle dx - \int_{\R^d} \langle \rho \psi \mathbf{B}, \nabla G_{\alpha} f \rangle \varphi\, dx  +\int_{\R^d} (\alpha \rho \psi G_{\alpha} f) \,\varphi dx  \nonumber \\
&& = \int_{\R^d} (\rho \psi f) \,\varphi dx, \quad \; \text{ for all } \varphi \in C_0^{\infty}(\R^d). \label{feppwikmee}
\end{eqnarray}
Note that 
$\rho$ is locally bounded below and above on $\R^d$ and $\rho \psi \mathbf{B} \in L_{loc}^p(\R^d, \R^d)$,  $\alpha \rho \psi \in L_{loc}^q(\R^d)$. Let $B$, $B'$ be open balls in $\R^d$ satisfying $\overline{B} \subset B'$. Since $\frac{1}{\psi} \in L^{\infty}(B')$, $G_{\alpha} f \in H^{1,2}(B')$. Thus by Theorem \ref{holreg}, there exists a H\"{o}lder continuous $\mu$-version $R_{\alpha} f$ of $G_{\alpha} f$ on $\R^d$ and constants $\gamma \in (0,1)$, $c_1>0$, which are independent of $f$ such that 
\begin{eqnarray}
\|R_{\alpha} f \|_{C^{0,\gamma}(\overline{B})} &\leq& c_1 \big( \|G_{\alpha} f \|_{L^1(B')}+ \| \rho \psi  f\|_{L^{  (\frac{1}{q}+\frac{1}{s})^{-1}}(B')}    \big) \nonumber \\ 
&\leq& c_2 \big( \|G_{\alpha} f \|_{L^1(B', \mu)}+\| f\|_{L^{s}(B', \mu)}    \big), \label{jfpojen} 
\end{eqnarray}
where $c_2:=c_1\big(\frac{1}{\inf_{B'}\rho\psi} \vee  \frac{\|\rho \psi \|_{L^q(B')}}{(\inf_{B'}\rho\psi)^{1/s}} \big)$. Using the H\"{o}lder inequality and the contraction property, \eqref{jfpojen} extends to $f \in \cup_{r \in [s, \infty)}L^r(\R^d, \mu)$. In order to extend \eqref{jfpojen} to  $f \in L^{\infty}(\R^d, \mu)$, let $f_n:=1_{B_n}  \cdot f \in L^{q}(\R^d, \mu)_0$,  $n\ge 1$. Then 
$\|f-f_n\|_{L^s(B', \mu)} + \| G_{\alpha}(f-f_n) \|_{L^1(B',\mu)} \to 0$ as $n\to \infty$ by Lebesgue's Theorem. Hence \eqref{jfpojen} also extends  to $f \in L^\infty(\R^d,\mu)$.
\end{proof}\\ \\
Let  $f \in  D(L_r)$ for some $r \in [s, \infty)$. Then $f = G_1(1-L_r) f$, hence by Theorem \ref{rescondojc}, $f$ has a locally H\"{o}lder continuous $\mu$-version on $\R^d$ and
\begin{eqnarray*}
\|f\|_{C^{0,\gamma}(\overline{B})}
&\leq& c_3 \|f\|_{D(L_r)},
\end{eqnarray*}
where $c_3>0$, $\gamma \in (0,1)$ are constants, independent of $f$. In particular, $T_t f \in D(L_r)$ and $T_t f$  has hence a continuous $\mu$-version, say $P_{t}f$, with
\begin{equation} \label{ptgraphnorm}
\|P_t f\|_{C^{0, \gamma}(\overline{B})} \leq c_3\|P_t f \|_{D(L_r)}.
\end{equation}
Note that $c_3$ is independent of $t\geq 0$ as well as of $f$.
The following Lemma will be quite important for later to show joint continuity of $P_{\cdot}g(\cdot)$ for $g \in \cup_{\nu \in[\frac{2p}{p-2},\infty]} L^\nu(\R^d, \mu)$.

\begin{lem} \label{contiiokc}
Assume {\bf (A1)}, {\bf(A2)}.
For any $f\in \bigcup_{r\in [s,\infty)} D(L_r)$ the map
$$
(x,t)\mapsto P_t f(x)
$$
is continuous on $ \R^d\times [0,\infty)$.
\end{lem}
\begin{proof}
Let $f\in D(L_r)$ for some $r\ge s$ and $\left ((x_n,t_n)\right )_{n\ge 1}$ be a sequence in $\R^d\times [0,\infty)$ that converges to $(x_0,t_0)\in \R^d\times [0,\infty)$. Note that $P_{t_0}f \in C(\R^d)$. Then there exists an open ball $B$ such that $x_n\in \overline{B}$ for all $n\ge 0$ and using \eqref{ptgraphnorm}
\begin{eqnarray*}
\left | P_{t_n} f(x_n)-P_{t_0} f(x_0) \right | &\leq& \left | P_{t_n} f(x_n)-P_{t_0} f(x_n) \right |+\left | P_{t_0} f(x_n)-P_{t_0} f(x_0) \right | \\[3pt]
&\leq& \| P_{t_n} f- P_{t_0}f \,\|_{C(\overline{B})} +\left | P_{t_0} f(x_n)-P_{t_0} f(x_0) \right | \\[3pt]
&\leq& c_3 \| P_{t_n} f- P_{t_0}f \,\|_{L^r(\R^d,\mu)} + c_3 \| P_{t_n} L_r f- P_{t_0} L_r f \,\|_{L^r(\R^d, \mu)}\\[3pt]
&&\;+\left | P_{t_0} f(x_n)-P_{t_0} f(x_0) \right | \longrightarrow 0  \;\;  \text{ as } n\to \infty.
\end{eqnarray*}
\end{proof}
\begin{rem}
If $(\mathcal{E}, C_0^{\infty}(\R^d))$ satisfies the weak sector condition, then $(T_t)_{t >0}$ is an analytic semigroup on $L^r(\R^d, \mu)$, $r \in [2, \infty)$ by Stein interpolation. If $f \in D(L_r)$ with $r \in [2, \infty)$, then 
\[
T_t f \in D(L_r),\ \ \ \text{and }\ \ \  \|L_r T_t f \|_{L^r(\R^d,\mu)} \le \frac{c}{t}\|f\|_{L^r(\R^d,\mu)},
\]
where $c>0$ is a constant which is independent of $f$ and $t>0$. 
Thus for any $r\in [s\vee 2,\infty)$, $t>0$, $f\in L^r(\R^d,\mu)$ and any open ball $B$
\begin{eqnarray*}
\|P_t f\|_{C^{0,\beta}(\overline{B})} &\le& c_3 \left(\|P_t f \;\|_{L^r(\R^d, \mu)} + \|L_r P_t f \;\|_{L^r(\R^d, \mu)}  \right) \nonumber \\
&\leq& c_3 \left( 1+\frac{c}{t}  \right) \|f\|_{L^r(\R^d, \mu)}. \label{tt2-3} \vspace{1em}
\end{eqnarray*}
However, it is in general difficult to show a weak sector condition and moreover it does not need to hold.
 Thus we have to develop another way to show the joint continuity of $P_{\cdot} f(\cdot)$ where $f$ is in some suitable class.
\end{rem}
\begin{theo}\label{1-3reg3}
Assume {\bf (A1)}, {\bf (A2)}, {\bf (A3)} and let $ f \in \bigcup_{\nu \in [\frac{2p}{p-2}, \infty]}L^\nu(\R^d, \mu)$, $t>0$. Then $T_t f$ has a continuous $\mu$-version $P_t f$ on $\R^d$ and furthermore $P_{\cdot} f(\cdot) $ is continuous on $\R^d \times (0, \infty)$. For any bounded open sets $U$, $V$ in $\R^d$ with $\overline{U} \subset V$ and $0<\tau_3<\tau_1<\tau_2<\tau_4$, i.e. $[\tau_1, \tau_2] \subset (\tau_3, \tau_4)$, we have the following estimate for all $f \in \cup_{\nu \in[\frac{2p}{p-2},\infty]} L^\nu(\R^d, \mu)$
\begin{equation} \label{thm main est}
\|P_{\cdot} f(\cdot)\|_{C(\overline{U} \times [\tau_1, \tau_2])} \leq  C_1 \| P_{\cdot} f(\cdot) \|_{L^{\frac{2p}{p-2},2}( V \times (\tau_3, \tau_4)) },
\end{equation}
where $C_1$ is a constant that depend on $\overline{U} \times [\tau_1, \tau_2], V \times (\tau_3, \tau_4)$, but is independent of $f$. 
\end{theo}
\begin{proof}
First assume $f \in D(\overline{L})_b \cap D(L_s) \cap D(L_2)$. By means of Lemma \ref{contiiokc}, define $u \in C_b(\R^d \times [0, \infty))$ by $u(x,t) := P_t f(x)$. Note that for any bounded open set $O\subset \R^d$ and $T>0$,  it holds $u \in H^{1,2}(O\times (0,T))$ by Theorem \ref{pjpjwoeij} below.  Let $\varphi_1 \in C_{0}^{\infty}(\R^d)$, $\varphi_{2} \in C_{0}^{\infty}\left((0,T)\right)$. Observe that $T_t f \in D(\overline{L})_b$, hence
\begin{align}
&\iint_{\R^d \times (0,T)} \langle \frac{1}{2} \rho A  \nabla u,  \nabla (\varphi_1 \varphi_2)  \rangle-  \langle  \rho \psi \bold{B}, \nabla \left(T_t f \right)  \rangle  \varphi_1 \varphi_2 \;dx dt  \nonumber \\
=&\int_{0}^{T} \varphi_2\big( \int_{\R^d}  \langle \frac{1}{2} \rho A  \nabla \left( T_t f \right),  \nabla \varphi_1   \rangle-  \big \langle  \rho \psi \bold{B}, \nabla \left(T_t f \right) \big \rangle  \varphi_1 \;dx \big) dt \nonumber \\
=&\int_{0}^{T} \varphi_2 ( \mathcal{E}^{0}(T_t f, \varphi_1)- \int_{\R^d} \langle \bold{B}, \nabla T_t f \rangle \varphi_1\, d\mu) dt \nonumber \\
=& \int_{0}^{T} - \varphi_2  (  \int_{\R^d}  \varphi_1  \overline{L}\, \overline{T}_t f \, d\mu  ) dt \nonumber \\
=&  \int_{0}^{T}  -\varphi_2 \big( \frac{d}{dt} \int_{\R^d} \varphi_1 T_t f \;\rho \psi dx  \big) dt \nonumber \\
=&  \int_{0}^{T}  ( \frac{d}{dt} \varphi_2 ) \big (  \int_{\R^d} \varphi_1 T_t f \, \rho\psi dx \big ) dt  \nonumber \\
=&\iint_{\R^d \times (0,T)} u \;\partial_t (  \varphi_1 \varphi_2  ) \rho \psi dx dt. \label{maindivdpl}
\end{align}
By Lemma \ref{stoneweier}, \eqref{maindivdpl} extends to
\begin{eqnarray}
&&\iint_{\R^d \times (0,T)}  \langle \frac{1}{2} \rho A  \nabla u,  \nabla \varphi\rangle-  \langle  \rho \psi \bold{B}, \nabla \left(T_t f \right) \rangle  \varphi  \;dx dt \nonumber  \\
&&\qquad  =\iint_{\R^d \times (0,T)} u \;\partial_t  \varphi  \cdot \rho \psi dx dt \quad \text{ for all } \varphi \in C_0^{\infty}(\R^d \times (0,T)). \label{maindivkjbie-3}
\end{eqnarray}
Let $\tau_2^*:=\frac{\tau_2+\tau_4}{2}$ and take $r>0$ so that 
$$
r<\frac{\sqrt{\tau_1 - \tau_3}}{2} \quad \text{ and  }\quad R_{\bar{x}}(2r)\subset V, \ \forall \bar{x}\in \overline{U}.
$$ 
Then for all $(\bar{x}, \bar{t}) \in \overline{U} \times [\tau_1, \tau_2^*]$, we have $R_{\bar{x}}(2r) \times (\bar{t}-(2r)^2, \bar{t})\subset V \times (\tau_3, \tau_4)$.
Using the compactness of $\overline{U} \times [\tau_1, \tau_2]$, there exist $(x_i, t_i) \in  \overline{U} \times [\tau_1, \tau^{*}_2]$, $i=1, \dots, N$, such that
\[
\overline{U} \times [\tau_1, \tau_2] \subset \bigcup_{i=1}^{N} R_{x_i}(r) \times (t_i-r^2, t_i). 
\]
Using Theorem \ref{weoifo9iwjeof},
\begin{eqnarray*}
\|u\|_{C(\overline{U} \times [\tau_1, \tau_2])} &=& \sup_{\overline{U} \times [\tau_1, \tau_2]} |u| \\
&\leq&  \max_{i=1,\dots,N}\, \sup_{R_{x_i}(r) \times (t_i-r^2, t_i)} |u| \\
&\leq&  \max_{i=1,\dots, N}\, c_i\|u\|_{L^{\frac{2p}{p-2},2}\big(R_{x_i}(2r) \times (t_i-(2r)^2, t_i)\big)} \\
&\leq& \underbrace{(\max_{i=1,\dots,N}\, c_i)}_{=:C_1} \,\|u\|_{L^{\frac{2p}{p-2},2}\big(V \times (\tau_3, \tau_4)\big)},
\end{eqnarray*}
where $c_i>0$ ($1 \leq i \leq N$) are constants which are independent of $u$. Thus for $\nu \geq \frac{2p}{p-2}$
\begin{eqnarray}
\|P_{\cdot }f\|_{C(\overline{U} \times [\tau_1, \tau_2])}  &\leq& C_1 \|P_{\cdot} f \|_{L^{\frac{2p}{p-2},2}\big(V \times (\tau_3, \tau_4)\big)}   \label{firstjkcme} \\
&=& C_1 \big( \int_{\tau_3}^{\tau_4} \big(  \int_{V} |T_t f |^{\frac{2p}{p-2}} dx \big)^{\frac{p-2}{p}}dt   \big)^{1/2} \nonumber \\
&\leq&  C_1 \big( \frac{1}{ \inf_{V} \rho \psi} \big)^{\frac{p-2}{2p}}  \big( \int_{\tau_3}^{\tau_4} \big(  \int_{V} |T_t f |^{\frac{2p}{p-2}} d\mu \big)^{\frac{p-2}{p}}dt   \big)^{1/2} \nonumber   \\
&\leq&   C_1 \big( \frac{1}{ \inf_{V} \rho \psi} \big)^{\frac{p-2}{2p}} \big(\int_{\tau_3}^{\tau_4} \|T_t f \|^2_{L^{\frac{2p}{p-2}}(V, \mu)} dt \big)^{1/2} \nonumber \\
&\leq&  C_1 \underbrace{\big( \frac{1}{ \inf_{V} \rho \psi} \big)^{\frac{p-2}{2p}} \mu(V)^{\frac{1}{2}-\frac{1}{p}-\frac{1}{\nu}} }_{=:C_2} \,\big(\int_{\tau_3}^{\tau_4} \|T_t f \|^2_{L^{\nu}(V, \mu)} dt \big)^{1/2} \nonumber \\
&\leq& C_1 C_2 (\tau_4-\tau_3)^{1/2} \|f\|_{L^{\nu}(\R^d, \mu)}. \label{conestlmet}
\end{eqnarray}
Now assume $f \in L^1(\R^d, \mu) \cap L^{\infty}(\R^d, \mu)$. Then $nG_n f \in D(\overline{L})_b \cap D(L_s) \cap D(L_2)$ for all $ n \in \N$ and $\lim_{n \rightarrow \infty} nG_n f = f$ \; in $L^{\nu}(\R^d, \mu)$. Thus \eqref{conestlmet} extends to all $f \in L^1(\R^d, \mu) \cap L^{\infty}(\R^d, \mu)$. If $\nu \in [\frac{2p}{p-2}, \infty)$, the above  again extends to all $f \in L^\nu(\R^d, \mu)$ using the denseness of $L^1(\R^d, \mu) \cap L^{\infty}(\R^d, \mu)$ in $L^\nu(\R^d, \mu)$. Finally assume $f \in L^{\infty}(\R^d, \mu)$ and let $f_n:= 1_{B_n}\cdot  f$ for $n \geq 1$. Then $\lim_{n \rightarrow \infty} f_n = f$ \; $\mu$-a.e. on $\R^d$ and
\begin{equation}\label{ae2-3}
T_t f= \lim_{n \rightarrow \infty}T_t f_n=\lim_{n \rightarrow \infty}P_t f_n, \; \mu \text{-a.e.\; on } \R^d.  
\end{equation}
 Thus using the sub-Markovian property and Lebesgue's Theorem in \eqref{firstjkcme}, $(P_{\cdot} f_n (\cdot ))_{n\ge 1}$ is a Cauchy sequence in 
$C(\overline{U} \times [\tau_1, \tau_2])$.
Hence we can again define
$$
P_{\cdot} f:=\lim_{n\to \infty}P_{\cdot} f_n(\cdot) \ \text{ in }\  C(\overline{U} \times [\tau_1, \tau_2]).
$$
For each $t>0$, $P_t f_n$ converges uniformly to $P_t f$ in $U$, hence in view of \eqref{ae2-3}, $T_t f$ has continuous $\mu$-version $P_t f$ and $P_{\cdot}f \in  C(\overline{U} \times [\tau_1, \tau_2])$. Therefore \eqref{conestlmet}  extends to all $f \in L^{\infty}(\R^d, \mu)$. Since $U$ and $[\tau_1, \tau_2]$ were arbitrary, it holds for any $f\in \cup_{\nu\in[\frac{2p}{p-2},\infty]} L^{\nu}(\R^d, \mu)$, $P_{\cdot} f(\cdot)$ is continuous on $\R^d \times (0, \infty)$ and for each $t>0$, $P_t f = T_t f$ \,$\mu$-a.e. on $\R^d$.  \vspace{-1.11em}
\end{proof}
\begin{rem}
\begin{itemize}
\item[(i)] By Theorem \ref{rescondojc}, we get a resolvent kernel and a resolvent kernel density for any $x\in \R^d$. 
Indeed, for any $\alpha>0$, $x\in \R^d$, \eqref{resestko} implies that
\begin{eqnarray*}\label{resker}
R_{\alpha}(x,A):=\lim_{l\to \infty}R_{\alpha}(1_{B_l\cap A})(x), \;\; A\in  \mathcal{B}(\R^d).
\end{eqnarray*}
defines a sub-probability measure $\alpha R_{\alpha}(x,dy)$ on $(\R^d, \mathcal{B}(\R^d))$ that is absolutely continuous with respect to $\mu$. Using the Radon-Nikodym derivative, the resolvent kernel density is defined by
\begin{eqnarray*}\label{reskernext}
r_\alpha (x, \cdot):=\frac{R_{\alpha}(x,dy)}{\mu(dy)}, \qquad x \in \R^d.
\end{eqnarray*} 
\item[(ii)]
By Theorem \ref{1-3reg3}, we also get a heat kernel and a heat kernel density for any $x\in \R^d$. 
Indeed, for any $t>0$, $x\in \R^d$, (\ref{tt2-3}) implies that
\begin{eqnarray*}\label{heatker-3}
P_{t}(x,A):=\lim_{l\to \infty}P_{t}(1_{B_l\cap A})(x), \;\; A\in  \mathcal{B}(\R^d),
\end{eqnarray*}
defines a sub-probability measure $P_{t}(x,dy)$ on $(\R^d, \mathcal{B}(\R^d))$ that is absolutely continuous with respect to $\mu$. Using the Radon-Nikodym derivative, the heat kernel density is defined by
\begin{eqnarray*}\label{rnd}
p_t (x, \cdot):=\frac{P_{t}(x,dy)}{\mu(dy)}, \qquad x \in \R^d.
\end{eqnarray*}
\end{itemize}
\end{rem}
\begin{prop}\label{regular2-3}
Assume {\bf (A1)}, {\bf(A2)}, {\bf(A3)} and let $t, \alpha>0$. Then it holds:
\begin{itemize}
\item[(i)] $G_{\alpha}g$ has a locally H\"older continuous $\mu$-version
\begin{eqnarray}\label{tgyrr-3}
R_{\alpha}g=\int_{\R^d}g(y) R_{\alpha}(\cdot,dy)=\int_{\R^d}g(y)r_{\alpha}(\cdot,y)\mu(dy)  
, \ \  \forall g\in \bigcup_{r\in [s,\infty]} L^r(\R^d,\mu). \qquad 
%\label{tgy-3}
\end{eqnarray}
In particular,  \eqref{tgyrr-3} extends by linearity  to all $g\in L^s(\R^d,\mu)+L^\infty(\R^d,\mu)$, i.e.  $(R_{\alpha})_{\alpha>0}$ is $L^{[s,\infty]}(\R^d,\mu)$-strong Feller.
\item[(ii)] $T_t f$ has a continuous $\mu$-version   
\begin{eqnarray}\label{tgy2-3}
P_t f= \int_{\R^d} f(y) P_t(\cdot,dy)=\int_{\R^d}f(y)p_{t}(\cdot,y)\mu(dy)
,\ \  \forall f\in \bigcup_{\nu \in [\frac{2p}{p-2},\infty]}L^\nu(\R^d,\mu).  \qquad
%\label{tgy-3}
\end{eqnarray}
In particular,  \eqref{tgy2-3} extends by linearity  to all $f\in L^{\frac{2p}{p-2}}(\R^d,\mu)+L^\infty(\R^d,\mu)$, i.e.   $(P_{t})_{t>0}$ is $L^{[\frac{2p}{p-2},\infty]}(\R^d,\mu)$-strong Feller.
\end{itemize}
Finally, for any $\alpha>0, x\in \R^d$, $g\in L^s(\R^d,\mu)+L^\infty(\R^d,\mu)$
$$
R_{\alpha}g(x)=\int_0^{\infty} e^{-\alpha t} P_t g(x)\,dt.
$$
\end{prop}

\section{Some auxiliary results}
In this Section, %From Proposition \ref{extlempt} to Lemma \ref{tecpjp}, 
we use all notations and assumptions from Section \ref{l1existencere}
\begin{prop} \label{extlempt}
$(T^0_t)_{t>0}$ restricted to $L^1(\R^d, \mu) \cap L^{\infty}(\R^d, \mu)$ can be extended to a sub-Markovian $C_0$-semigroup of contractions $(\overline{T_t}^0)_{t>0}$ with generator $(\overline{L}^0, D(\overline{L}^0) )$  on $L^1(\R^d, \mu)$. If $f \in D(L^0)$ and $f, L^0 f \in L^1(\R^d, \mu)$, then $f \in D(\overline{L}^0)$ and $\overline{L}^0 f = L^0 f$. Set $\mathcal{A}:= \{ u \in D(L^0) \cap L^1(\R^d, \mu) \mid L^0 u \in L^1(\R^d, \mu) \}$. Then $(\overline{L}^0, D(\overline{L}^0))$ is  the closure of $(L^0, \mathcal{A})$ on $L^1(\R^d, \mu)$. \\
Similarly, for a bounded open subset $V$ of $\R^d$, $(T^{0,V}_t)_{t>0}$ restricted to $L^1(V, \mu) \cap L^{\infty}(V, \mu)$ can be extended to a sub-Markovian $C_0$-semigroup of contractions $(\overline{T}^{0,V}_t)_{t>0}$  on $L^1(V, \mu)$. Also if $f \in D(L^{0,V})$ and $f, \, L^{0,V} f \in L^1(V, \mu)$, then $f \in D(\overline{L}^{0,V})$ and $\overline{L}^{0,V} f = L^{0,V} f$. Finally \,$(\overline{L}^{0,V}, D(\overline{L}^{0,V}))$ is the closure of $(L^{0,V}, D(L^{0,V}))$ on $L^1(V, \mu)$.
\end{prop}
\begin{proof}
Since the proof for the case of  $(T^{0,V}_t)_{t>0}$ is exactly same with the case of $(T^{0}_t)_{t>0}$, we will only prove the case of $(T^{0}_t)_{t>0}$. Since $(\mathcal{E}^0, D(\mathcal{E}^0))$ is a regular Dirichlet from, there exists a Hunt process 
$$
\mathbb{M}^0 = (\Omega^0, \mathcal{F}^0, (\mathcal{F}^0_{t})_{t\geq0}, (X^0_t)_{t\geq0}, (\mathbb{P}^0_x)_{x \in \R^d \cup \Delta} )
$$
with life time $\zeta^0 = \inf \{ t>0 \mid X^0_t = \Delta \}$ such that for any $g \in L^2(\R^d, \mu)$
$$
x \mapsto \mathbb{E}^0_x \left[ g(X^0_t)  \right] \; \text{ is a quasi-continuous $\mu$-version of } T^0_t g.
$$ 
Let $f \in L^1(\R^d, \mu) \cap L^{\infty}(\R^d, \mu)$. Using Jensen inequality and sub-Markovian property of $(T^0_t)_{t>0}$
\begin{eqnarray*}
\int_{\R^d} |T^0_t f| d \mu  &=& \int_{\R^d}  \left | \mathbb{E}^0_{\cdot} \left[ f(X^0_t)  \right]  \right |d \mu \\
&\leq&  \int_{\R^d}  \mathbb{E}^0_{\cdot} \left[  \left | f(X^0_t) \right | \right]  d \mu \\
&=& \lim_{n \rightarrow \infty} \int_{\R^d} T^0_t |f| \cdot 1_{B_n}d \mu \\
&=& \lim_{n \rightarrow \infty} \int_{\R^d} |f| \cdot T_t^0 1_{B_n} d \mu \\
&\leq& \int_{\R^d} |f| d \mu.
\end{eqnarray*}
Since $L^1(\R^d, \mu) \cap L^{\infty}(\R^d, \mu)$ is dense in $L^1(\R^d, \mu)$, \;$(T_t^0)_{t>0}$ restricted to $L^1(\R^d, \mu) \cap L^{\infty}(\R^d, \mu)$ uniquely extend to the sub-Markovian contraction semigroup $(\overline{T}^0_t)_{t>0}$
on $L^1(\R^d, \mu)$.
Define 
\begin{eqnarray*}
&& \mathcal{D}:= L^{\infty}(\R^d, \mu) \cap \{g \mid  g \geq 0 \text{ and there exists } A \in \mathcal{B}(\R^d) \\
&&\hspace{12em }\text{ with  }\mu(A) < \infty  \text{ and } g = 0  \text{ on } \R^d \setminus A \}.
\end{eqnarray*}
Since $ \mathcal{D}$ is dense in $L^1(\R^d, \mu)^+$,\, $\mathcal{D}-\mathcal{D}$ is dense in $L^1(\R^d, \mu)$. Let $f \in \mathcal{D}-\mathcal{D}$. Then there exists $A \in \mathcal{B}(\R^d)$ with $\mu(A)< \infty$ such that $\text{supp}(f) \subset A$ and $f \in L^1(\R^d, \mu) \cap L^{\infty}(\R^d, \mu)$. By strong continuity of $(T^0_t)_{t>0}$ on $L^2(\R^d, \mu)$
$$
\lim_{t \rightarrow 0+}\int_{\R^d} 1_A |T^0_t f| d \mu =  \int_{\R^d} 1_A  |f|  d\mu = \|f\|_{L^1(\R^d, \mu)},
$$
hence using the contraction property on $L^1(\R^d, \mu)$,
\begin{eqnarray*}
0 &\leq& \int_{\R^d} 1_{\R^d \setminus A} \,|T^0_t f| d \mu  =  \int_{\R^d} |T^0_t f | d\mu - \int_{\R^d} 1_{A} |T^0_t f| d \mu  \\
&\leq& \|f\|_{L^1(\R^d, \mu)} - \int_{\R^d} 1_A |T^0_t f|d \mu \longrightarrow 0 \;\; \text{ as } t \rightarrow 0+.
\end{eqnarray*}
Therefore
\begin{eqnarray*}
\lim_{t \rightarrow 0+} \int_{\R^d} |T^0_t f -f | d \mu &=& \lim_{t \rightarrow 0+} \big(\int_{\R^d} 1_A |T^0_t f - f| d \mu + \int_{\R^d} 1_{\R^d \setminus A} |T^0_t f | d \mu \big) \\
&\leq& \mu(A)^{1/2} \lim_{t \rightarrow 0+} \|T_t f - f \|_{L^2(\R^d, \mu)} =0. 
\end{eqnarray*}
By the denseness of $\mathcal{D}-\mathcal{D}$ in $L^1(\R^d, \mu)$, we get the strong continuity of $(\overline{T}_t^0)_{t>0}$ on $L^1(\R^d, \mu)$.
Now let $f \in D(L^0)$ and $f,\, L^0 f \in L^1(\R^d, \mu)$. Then $f \in L^{1}(\R^d, \mu) \cap L^2(\R^d, \mu)$, $L^0 f \in  L^{1}(\R^d, \mu) \cap L^2(\R^d, \mu)$, hence we get \,$\overline{T}_t^0 f = T^0_t f$, \,  $\overline{T}_t^0 L^0f = T^0_t L^0f$\; for every $t>0$. Using the Fundamental Theorem of Calculus (on Banach space) and strong continuity of $(\overline{T}_t^0)_{t>0}$ on $L^1(\R^d, \mu)$
\begin{eqnarray*}
\frac{\overline{T}_t^0 f- f}{t} &=& \frac{T_t^0 f- f}{t}  = \frac{1}{t} \int_{0}^t T^0_s L^0 f \,ds  \\
&=& \frac{1}{t} \int_{0}^t \overline{T}_s^0 L^0 f \,ds   \longrightarrow L^0 f \quad \text{ in } L^1(\R^d, \mu) \;\,\text{ as } t \rightarrow 0+.
\end{eqnarray*}
Consequently, $f \in D(\overline{L}^0)$ and $\overline{L}^0 f = L^0 f$. \\
\centerline{}
\text{} \quad Let $(\overline{G}_{\alpha}^0)_{\alpha>0}$ be the resolvent generated by $(\overline{L}^0, D(\overline{L}^0))$.\;
Set $\mathcal{C}:=\big \{ \overline{G}_{1}^0 g \mid  g \in C_0^{\infty}(\R^d) \big \}$. Then $\mathcal{C} \subset \mathcal{A}$ and one can directly check that $\mathcal{C}  \subset D(\overline{L}^0)$ is dense with respect to graph norm $\| \cdot \|_{D(\overline{L}^0)}$, hence it completes our proof.
\end{proof}

\begin{lem}\label{bddapxlem}
Let $V$ be a bounded open subset of $\R^d$ and $f \in \widehat{H}^{1,2}_0(V, \mu)_b$. Then there exists a sequence $(f_n)_{n \geq 1} \subset C_0^{\infty}(V)$ and a constant $M>0$ such that $\|f_n\|_{L^{\infty}(V)} \leq M$ for all $n \geq 1$ and
\begin{eqnarray*}
&&\lim_{n \rightarrow \infty} f_n = f \quad \text{ in } \; \widehat{H}^{1,2}_0(V, \mu), \;\qquad \lim_{ n \rightarrow \infty} f_n =f \quad \mu \; \text{-a.e. on } V.
\end{eqnarray*}
\end{lem}
\begin{proof}
Take $(g_n)_{n \geq 1} \subset C_0^{\infty}(V)$ such that 
\begin{equation} \label{appchar}
\lim_{n \rightarrow \infty} g_n = f \; \text{ in } \widehat{H}^{1,2}_0(V, \mu)  \;\; \text{ and  }\; \lim_{n \rightarrow \infty} g_n  = f \;\; \;\mu \text{ -a.e. on  } V.
\end{equation}
Define $\varphi \in C^{\infty}_0(\R)$ such that $\varphi(t) = t$\, if\, $|t| \leq \|f\|_{L^{\infty}(\R^d)}+1$ and $\varphi(t)= 0$ \, if \, $|t|\geq \|f\|_{L^{\infty}(\R^d)}+2$. \; Let $M:= \|\varphi \|_{L^{\infty}(\R)}$ and $\widetilde{f}_n:= \varphi(g_n)$. Then $\widetilde{f}_n \in C_0^{\infty}(V)$ and \;$\| \widetilde{f}_n\|_{L^{\infty}(V)} \leq M$ for all $n \geq 1$.
By Lebesgue's Theorem and \eqref{appchar}, 
$$
\; \quad  \lim_{n \rightarrow \infty} \widetilde{f}_n =\lim_{n \rightarrow \infty} \varphi (g_n) = \varphi(f)  = f \quad \text{ in } L^2(V, \mu). \; 
$$
Using the chain rule and \eqref{appchar}
\begin{eqnarray*}
\sup_{n \geq 1} \| \nabla \widetilde{f}_n \|_{L^{2}(V, \R^d)} &=&  \sup_{n \geq 1} \|\nabla  \varphi(g_n) \|_{L^2(V, \R^d)} \\
&\leq& \|\varphi'\|_{L^{\infty}(\R)}  \sup_{n \geq 1} \|\nabla g_n \|_{L^2(V, \R^d)} \\
&<& \infty.
\end{eqnarray*}
Thus by the Banach-Alaoglu Theorem and the Banach-Saks Theorem, there exists a subsequence of $(\widetilde{f}_n)_{n \geq 1}$, say again $(\widetilde{f}_n)_{n \geq 1}$, such that for the Cesaro mean
$$
f_N:= \frac{1}{N} \sum_{n=1}^{N} \widetilde{f}_n \longrightarrow f \;\; \;\text{ in } \;  \widehat{H}^{1,2}_0(V, \mu) \quad \text{ as } N \rightarrow \infty.
$$
Note that $f_N \in C_0^{\infty}(V)$, $\|f_N\|_{L^{\infty}(V)} \leq M$ for all $N \in \N$. \;Since the Cesaro mean of a convergent sequence in $\R$ is also converges, $(f_n)_{n \geq 1}$ is the desired sequence.
\end{proof}

\begin{lem}\label{applempn}
Let $f \in \widehat{H}^{1,2}_0(\R^d, \mu)_{0,b}$ and $V$ be a bounded open subset of $\R^d$ with $\text{supp}(f) \subset V$. Then $f \in \widehat{H}^{1,2}_0(V, \mu)_b$. Moreover there exists $(f_n)_{n \geq 1} \subset C_0^{\infty}(\R^d)$ and a constant $M>0$ such that $\text{supp}(f_n) \subset V$, $\|f_n\|_{L^{\infty}(V)} \leq M$ for all $n \geq 1$ and
\begin{eqnarray*}
&&\lim_{n \rightarrow \infty} f_n = f \quad \text{ in } \; \widehat{H}^{1,2}_0(\R^d, \mu), \;\qquad \lim_{ n \rightarrow \infty} f_n =f \quad \mu \; \text{-a.e. on } \R^d.
\end{eqnarray*}
\end{lem}
\begin{proof}
Let $W$ be an open subset of $\R^d$ satisfying $\text{supp}(f) \subset W \subset \overline{W} \subset V$.
Take a cut-off function $\chi \in C_0^{\infty}(\R^d)$ satisfying $\text{supp}(\chi) \subset V$ and $\chi \equiv 1$ on $W$. Since $f \in \widehat{H}^{1,2}_0(\R^d, \mu)$, there exists $\widetilde{g}_n \in C_0^{\infty}(\R^d)$ such that
$$
\lim_{n \rightarrow \infty} \widetilde{g}_n =f \;\; \text{ in } \widehat{H}^{1,2}_0(\R^d, \mu).
$$
Thus $\chi \widetilde{g}_n \in C_0^{\infty}(\R^d)$ with  $\text{supp}(\chi \widetilde{g}_n) \subset V$ and
\begin{eqnarray*}
\|\chi \widetilde{g}_n -  f  \|_{L^{2}(\R^d, \mu)} &=& \|\chi \widetilde{g}_n -  \chi f  \|_{L^{2}(\R^d, \mu)} \\
&\leq& \| \chi \|_{L^{\infty}(\R^d)} \|\widetilde{g}_n - f \|_{L^2(\R^d, \mu)} \longrightarrow 0 \;\; \text{ as } n \rightarrow \infty.
\end{eqnarray*}
Note that $\chi \widetilde{g}_n \in C_0^{\infty}(\R^d) \subset \widehat{H}^{1,2}_0(V,\mu)$ and
\begin{eqnarray*}
\sup_{n \geq 1} \|\nabla (\chi \widetilde{g}_n) \|_{L^2(V, \R^d)} &=& \sup_{n \geq 1} \left(\| \widetilde{g}_n \nabla \chi  \|_{L^2(V, \R^d)} + \|\chi \nabla \widetilde{g}_n  \|_{L^2(V, \R^d)}\right) \\
&\leq& \sup_{n \geq 1} \big(\frac{\| \nabla \chi \|_{L^{\infty}(V, \R^d)}  }{\inf (\rho \psi)} \|g_n\|_{L^2(\R^d, \mu)} + \| \chi \|_{L^{\infty}(\R^d)} \|\nabla \widetilde{g}_n\|_{L^2(\R^d, \R^d)}\big) \\
&<& \infty.
\end{eqnarray*}
%By Banach-Alaouglu Theorem
Since bounded sequences in Hilbert spaces have a weakly convergent subsequence, $f \in \widehat{H}^{1,2}_0(V, \mu)$. Taking $(f_n)_{n \geq 1} \subset C_0^{\infty}(V)$ as in Lemma \ref{bddapxlem} and extending it trivially to $C_0^{\infty}(\R^d)$, our assertion holds.
\end{proof}

\begin{lem} \label{tecpjp}
Let $V_1$, $V_2$ be bounded open subsets of $\R^d$ satisfying $\overline{V}_1 \subset V_2$. Assume $f \in \widehat{H}^{1,2}_0(V_2, \mu)$, $g \in \widehat{H}_0^{1,2}(V_1, \mu)$ with $g =0$ on $V_2 \setminus V_1$. If $0 \leq f \leq g$, then $f \in \widehat{H}^{1,2}_0(V_1, \mu)$.
\end{lem}
\begin{proof}
Take $(g_n)_{n \geq 1} \subset C_0^{\infty}(V_2)$ satisfying $\text{supp}(g_n) \subset V_1$ for all $n \in \N$ and 
$$
\lim_{n \rightarrow \infty} g_n = g \quad \text{ in } \widehat{H}^{1,2}_0(V_2, \mu).
$$
\;Observe that for all $n \in \N$
$$
\text{supp}(f \wedge g_n) \subset V_1 \;\text{ and }\; f \wedge g_n = \frac{f+g_n}{2}-\frac{|f-g_n|}{2}\in \widehat{H}^{1,2}_0(V_2, \mu).
$$
By Lemma \ref{applempn}, $f \wedge g_n \in \widehat{H}^{1,2}_0(V_1, \mu)$ for all $n \in \N$.
\;Moreover 
\begin{eqnarray*}
\lim_{n \rightarrow \infty} f \wedge g_n  = \lim_{n \rightarrow \infty} \big(\frac{f+g_n}{2}-\frac{|f-g_n|}{2} \big)=\frac{f+g}{2}-\frac{|f-g|}{2} = f \wedge g = f  \;\; \text{ in }  L^2(V_1, \mu).
\end{eqnarray*}
Since $\big(  \langle \cdot,  \cdot \rangle_{\widehat{H}^{1,2}_0(V_2, \mu)}, \widehat{H}^{1,2}_0(V_2, \mu)  \big)$ is a Dirichlet form,
\begin{eqnarray*}
&\, & \sup_{n \geq 1} \|f \wedge g_n \|_{\widehat{H}^{1,2}_0(V_1, \mu)} \\
&=&  \sup_{n \geq 1} \|f \wedge g_n \|_{\widehat{H}^{1,2}_0(V_2, \mu)}  \\
&=&  \sup_{n \geq 1} \big\| \frac{f+g_n}{2}-\frac{|f-g_n|}{2}   \big \|_{\widehat{H}^{1,2}_0(V_2, \mu)} \\
&\leq& \frac{1}{2} \sup_{n \geq 1} \big( \|f\|_{\widehat{H}^{1,2}_0(V_2, \mu)}  +\|g_n\|_{\widehat{H}^{1,2}_0(V_2, \mu)} +\big \| |f| \big\|_{\widehat{H}^{1,2}_0(V_2, \mu)} +\big \||g_n| \big \|_{\widehat{H}^{1,2}_0(V_2, \mu)}    \big)   \\
&\leq& \sup_{n \geq 1} \big( \|f\|_{\widehat{H}^{1,2}_0(V_2, \mu)}  +\|g_n\|_{\widehat{H}^{1,2}_0(V_2, \mu)}  \big) < \infty.
\end{eqnarray*}
Thus by the Banach-Alaoglu Theorem,  $f \in \widehat{H}^{1,2}_0(V_1, \mu)$. \vspace{-0.4em}
\end{proof}
\\[10pt]
For a bounded open set $U$ in $\R^d$ and $T>0$, $C^{2}(\overline{U} \times [0,T])$ denotes the space of all twice continuously differentiable functions on $\overline{U} \times [0,T]$ with the norm defined by
$$
\|u\|_{C^{2}(\overline{U} \times [0,T])}:= \| u\|_{C(\overline{U} \times [0,T])} + \sum_{i=1}^{d+1} \|\partial_i u\|_{C(\overline{U} \times [0,T])} + \sum_{i,j=1}^{d+1} \|\partial_i \partial_j u \|_{C(\overline{U} \times [0,T])}.
$$
\begin{lem} \label{stoneweier}
Let $U$ be a bounded open subset of $\R^d$ and $T>0$. Set
\begin{eqnarray*}
&&\mathcal{S}:= \big \{h \in C_0^{\infty}(U \times (0,T))\mid \text{there exists $N \in \N$ such that $h= \sum_{i=1}^N f_i g_i$,} \\
&& \qquad \qquad \qquad \qquad \text{where $f_i \in C_0^{\infty}(U)$, $g_i \in C_0^{\infty}((0,T))$ for all i=1,\dots, N }    \big \}.
\end{eqnarray*}
Then $C_0^{2}(U \times (0,T)) \subset \overline{\mathcal{S}}|_{C^2(\overline{U} \times [0,T])}$.
\end{lem}
\begin{proof} {\bf Step 1}: Let $V$ be a bounded open set in $\R^d$ and $T_1, T_2 \in \R$ with $T_1<T_2$.
Define 
\begin{eqnarray*}
&&\mathcal{R}:= \Big \{h \in C_0^{\infty}(V \times (T_1,T_2))\mid \text{there exists $N \in \N$ such that $h= \sum_{i=1}^N f_i g_i$,} \\
&& \qquad \qquad \qquad \qquad \text{where $f_i \in C_0^{\infty}(V)$, $g_i \in C_0^{\infty}((T_1,T_2))$ for all $i=1,\dots, N$} \Big \}   
\end{eqnarray*}
We claim that 
\begin{equation} \label{step1fr}
C_0^{2}(V \times (T_1,T_2)) \subset \overline{\mathcal{R}}|_{C(\overline{V} \times [T_1,T_2])}.
\end{equation}
Note that $V \times (T_1,T_2)$ is a locally compact space and $\overline{\mathcal{R}}|_{C(\overline{V} \times [T_1,T_2])}$ is a closed subalgebra of $C_{\infty}(V \times (T_1,T_2))$. We can easily check that for each $(x,t) \in V \times (T_1,T_2)$, there exists $\widetilde{h} \in \mathcal{R}$ such that $\widetilde{h}(x,t) \neq 0$. For $(x,t), (y,s) \in V \times (T_1,T_2)$ and $(x,t) \neq (y,s)$, there exists $\widehat{h} \in \mathcal{R}$ such that $\widehat{h}(x,t)=1$ and $\widehat{h}(y,s)=0$. Therefore by \cite[Chapter V, 8.3 Corollary]{Conw},  we obtain $\overline{\mathcal{R}}|_{C(\overline{V} \times [T_1,T_2])}= C_{\infty}(V \times (T_1,T_2))$ (the continuous functions on $V \times (T_1,T_2)$ that vanish at infinity, i.e. given $\varepsilon>0$, there exists a compact set $K \subset V \times (T_1,T_2)$ such that $|f(x)|< \varepsilon $ for all $x \in V \times (T_1,T_2) \setminus K$), so that our claim \eqref{step1fr} holds. \\
\centerline{}
{\bf Step 2}: $C_0^{2}(U \times (0,T)) \subset \overline{\mathcal{S}}|_{C^2(\overline{U} \times [0,T])}$. \\
For $n \in \N$, let $\eta_{n}$ be a standard mollifier on $\R^d$ and $\theta_{n}$ be a standard mollifier on $\R$. Then $\xi_{n}:= \eta_{n} \theta_{n}$ is a standard mollifier on $\R^d \times \R$. Let $h \in C_0^{2}(U \times (0,T))$ be given. Then there exists a bounded open subset $V$ of $\R^d$ and $T_1, T_2 \in \R$ with $0<T_1<T_2$ such that
$$
\text{supp}(h) \subset V \times (T_1, T_2) \subset \overline{V} \times [T_1, T_2] \subset U \times (0,T).
$$
Take $N \in \N$ such that $f*\xi_N \in C_0^{\infty}(U \times (0,T))$ for all $f \in C_0^{\infty}(V \times (T_1,T_2))$.\\
Note that by \cite[Proposition 4.20]{BRE}, it holds
\begin{eqnarray*}
&&\partial_t(h * \xi_{\varepsilon}) = \partial_t  h * \xi_{\varepsilon}, \; \partial^2_t (h * \xi_{\varepsilon}) = \partial^2_t  h * \xi_{\varepsilon}, \; \partial_t \partial_i  (h * \xi_{\varepsilon}) = \partial_t \partial_i  h * \xi_{\varepsilon},  \\
&& \partial_i (h * \xi_{\varepsilon}) = \partial_i  h * \xi_{\varepsilon}, \,\;\; \partial_i \partial_j (h * \xi_{\varepsilon}) =  \partial_i \partial_j  h * \xi_{\varepsilon}\, \text{ for any $1 \leq i,j \leq d$.}
\end{eqnarray*}
Hence by \cite[Proposition 4.21]{BRE}, $\lim_{n \rightarrow \infty}h *\xi_{\varepsilon} = h$ in $C^2(\overline{U} \times [0,T])$. 
Thus given $\varepsilon>0$, there exists $n_{\varepsilon} \in \N$ with $n_{\varepsilon} \geq N$ such that
$$
\|h - h* \xi_{n_{\varepsilon}}\|_{C^2(\overline{U} \times [0,T])} < \frac{\varepsilon}{2}.
$$
Let $\mathcal{R}$ be as in Step 1. By \eqref{step1fr}, there exists $h_{\varepsilon} \in \mathcal{R} \subset  C_0^{\infty}(V \times (T_1,T_2))$ such that
$$
\|h-h_{\varepsilon}\|_{C(\overline{U} \times [0,T])} < \frac{\varepsilon}{2 \|\xi_{n_{\varepsilon}}\|_{C^2(\overline{U} \times [0,T])}}.
$$
Thus using \cite[Proposition 4.20]{BRE} and Young's inequality,
$$
\|h * \xi_{n_{\varepsilon}} - h_{\varepsilon}* \xi_{n_{\varepsilon}}\|_{C^2(\overline{U} \times [0,T])} \leq \|\xi_{n_{\varepsilon}}\|_{C^2(\overline{U} \times [0,T])}\|h-h_{\varepsilon}\|_{C(\overline{U} \times [0,T])}< \frac{\varepsilon}{2}.
$$
Therefore
$$
\|h-h_{\varepsilon}* \xi_{n_{\varepsilon}}\|_{C^2(\overline{U} \times [0,T])} <\varepsilon.
$$
Since $h_{\varepsilon}* \xi_{n_{\varepsilon}} \in \mathcal{S}$, we have $h \in  \overline{\mathcal{S}}|_{C^2(\overline{U} \times [0,T])}$, as desired. 
\end{proof}

\section{Well-posedness} \label{effwegewe}
\subsection{Weak existence}
The following assumption will in particular be necessary to obtain a Hunt process with transition function $(P_t)_{t\ge 0}$ (and consequently a weak solution to the corresponding SDE for every starting point). It will be first used in Theorem \ref{existhunt4} below.
\begin{itemize}
\item[\bf(A4)] 
$\mathbf{G} \in L_{loc}^s(\R^d, \R^d, \mu)$, where $s$ is as in {\bf (A2)}.
\end{itemize}
The condition {\bf (A4)} is not necessary to get a Hunt processes (and consequently a weak solution to the corresponding SDE for merely quasi-every starting point) as in  the following proposition.
\begin{prop}\label{Huntex}
There exists a Hunt process 
$$
\tilde{\M} = (\tilde{\Omega}, \tilde{\F}, (\tilde{\F})_{t \ge 0}, (\tilde{X}_t)_{t \ge 0}, (\tilde{\P}_x)_{x \in \R^d \cup \{ \Delta \} })
$$ 
with life time $\tilde\zeta:=\inf\{t\ge 0\,|\,\tilde{X}_t=\Delta\}$ and cemetery $\Delta$ such that $\mathcal{E}$ is (strictly properly) associated with $\tilde{\M}$ and for strictly $\mathcal{E}$-q.e. $x \in \R^d$,
$$
\tilde{\P}_{x} \big( \big\{ \omega \in \tilde{\Omega} \mid  \tilde{X}_{\cdot}(\omega) \in C\big([0, \infty), \R^d_{\Delta}\big),\, \tilde{X}_{t}(\omega) = \Delta, \, \forall t \geq \zeta(\omega)   \big\} \big) =1.
$$
\end{prop}
\begin{proof}
First one shows the quasi-regularity of the generalized Dirichlet form $(\mathcal{E},D(L_2))$ associated with $(L_2,D(L_2))$, and the existence of an $\mu$-tight special standard process associated with $(\mathcal{E},D(L_2))$. This can be done exactly as in \cite[Theorem 3.5]{St99}. One only has to take care that the space $\cal{Y}$ as defined in the proof of \cite[Theorem 3.5]{St99} is replaced because of a seemingly uncorrected version of the paper by the following one
$$
\text{}\; {\cal{Y}}:=\{u\in D(\overline{L}))_b\,|\, \exists f,g\in L^1(\R^d,\mu)_b,\ f,g\ge 0, \text{ such that } u\le G_1f \text{ and } -u\le G_1g \}
$$
in order to guarantee the convergence at the end of the proof. Then the assertion will follow  exactly as in \cite[Theorem 6]{Tr5}, using for the proof instead $\cal{G}$ there the space $\cal{Y}$ defined above and defining $E_k\equiv \R^d$, $k\ge 1$.
\end{proof}
\text{}\\
\begin{rem}
\begin{itemize}
\item[(i)]
Assume  {\bf (A1)}, {\bf (A2)}, {\bf (A3)} and $\mathbf{G} \in L^{\frac{sq}{q-1}}_{loc}(\R^d,\R^d)$. Then for any bounded open subset $V$ of $\R^d$, it holds
$$
\int_{V}\|\mathbf{G}\|^{s} d\mu \leq \|\mathbf{G}\|^s_{L^\frac{sq}{q-1}(V)}\|\rho \psi \|_{L^q(V)},
$$
hence {\bf (A4)} is satisfied. 
\item[(ii)]
Two simple examples where  {\bf (A1)}, {\bf (A2)}, {\bf (A3)}, {\bf (A4)} are satisfied are given as follows: for the first example let $A$, $\psi$ satisfy the assumptions of {\bf (A1)}, $\psi \in L^p_{loc}(\R^d)$, $s=\frac{dp}{2p-d}+\varepsilon$, and $\mathbf{G} \in L^{\infty}_{loc}(\R^d, \R^d)$ and for the second let $A$, $\psi$ satisfy the assumptions of {\bf (A1)}, $\psi \in L^{2p}_{loc}(\R^d)$, $s=\frac{2pd}{4p-d}+\varepsilon$ and $\mathbf{G} \in L^{2p}_{loc}(\R^d, \R^d)$. In both cases $\varepsilon>0$ can be chosen to be arbitrarily small.
\end{itemize}
\end{rem}
\centerline{}
Analogously to \cite[Theorem 3.12]{LT18}, we obtain:
\begin{theo}\label{existhunt4}
Under the assumptions {\bf(A1)}, {\bf(A2)}, {\bf(A3)}, {\bf(A4)}, there exists a Hunt process
\[
\M =  (\Omega, \F, (\F_t)_{t \ge 0}, (X_t)_{t \ge 0}, (\P_x)_{x \in \R^d\cup \{\Delta\}}   )
\]
with state space $\R^d$ and life time 
$$
\zeta=\inf\{t\ge 0\,:\,X_t=\Delta\}=\inf\{t\ge 0\,:\,X_t\notin \R^d\}, 
$$
having the transition function $(P_t)_{t \ge 0}$ as transition semigroup, such that $\M$ has continuous sample paths in the one point compactification $\R^d_{\Delta}$ of $\R^d$ with the cemetery $\Delta$ as point at infinity, i.e. for any $x \in \R^d$,
$$
\P_{x} \big( \left\{ \omega \in \Omega \mid  X_{\cdot}(\omega) \in C\big([0, \infty), \R^d_{\Delta}\big) ,\, X_{t}(\omega) = \Delta, \, \forall t \geq \zeta(\omega) \right\} \big) =1.
$$
\end{theo}
\begin{rem}\label{mostresults4}
The analogous results to \cite[Lemma 3.14, Lemma 3.15, Proposition 3.16, Proposition 3.17, Theorem 3.19]{LT18}  hold in the situation of this paper. One of the main differences is that $q=\frac{dp}{d+p}>\frac{d}{2}$  of \cite{LT18} is replaced by $s>\frac{d}{2}$ of {\bf (A2)}.  Especially, a Krylov type estimate for $\M$ of Theorem \ref{existhunt4} holds as stated in \eqref{sfpolkd} right below.  Let $g\in L^r(\R^d,\mu)$ for some $r\in [s,\infty]$ be given. Then for any ball $B$, there exists a constant $C_{B,r}$, depending in particular on $B$ and $r$, such that for all $t \ge  0$,
\begin{equation} \label{sfpolkd}
\sup_{x\in \overline{B}}\E_x\left [ \int_0^t |g|(X_s) \, ds \right ] < e^t C_{B,r} \|g\|_{L^r(\R^d, \mu)}.
\end{equation} 
The derivation of \eqref{sfpolkd} is based on Theorem \ref{rescondojc}, whose proof  uses the elliptic H\"{o}lder estimate of Theorem \ref{holreg}. This differs from the proof of the Krylov type estimates in \cite{LT18} and \cite{LT19}, which are based on an elliptic $H^{1,p}$-estimate.  Finally, one can get the analogous conservativeness and moment inequalities to \cite[Theorem 4.2, Theorem 4.4(i)]{LT18} in the situation of this paper. 
\end{rem}
The following theorem can be proved exactly as in \cite[Theorem 3.19]{LT18}.
\begin{theo}\label{weakexistence4}
Assume {\bf (A1)}, {\bf (A2)}, {\bf (A3)}, {\bf (A4)} are satisfied. Consider the Hunt process $\M$ from Theorem \ref{existhunt4} with coordinates $X_t=(X_t^1,...,X_t^d)$. Let $(\sigma_{ij})_{1 \le i \le d,1\le j \le m}$, $m\in \N$ arbitrary but fixed, be any locally uniformly strictly elliptic matrix consisting of continuous functions for all $1 \leq i \leq d$, $1 \leq j \leq m$,  such that $A=\sigma \sigma^T$,  i.e. 
$$
a_{ij}(x)=\sum_{k=1}^m \sigma_{ik}(x) \sigma_{jk}(x), \ \  \forall x\in \R^d, \ 1\le i,j\le d.
$$
Set 
$$
\widehat{\sigma}=\sqrt{\frac{1}{\psi}}\cdot \sigma\text{ , i.e. }\widehat{\sigma}_{ij}=\sqrt{\frac{1}{\psi}}\cdot \sigma_{ij}, \ 1 \leq i \leq d,\ 1 \leq j \leq m.
$$
(Recall that the expression $\frac{1}{\psi}$ denotes an arbitrary Borel measurable function satisfying $\psi\cdot \frac{1}{\psi}=1$ a.e.).\\
Then on a standard extension 
of $(\Omega, \F, (\F_t)_{t\ge 0}, \P_x )$, $x\in \R^d$, that we denote for notational convenience again 
by $(\Omega, \F, (\F_t)_{t\ge 0}, \P_x )$, $x\in \R^d$, there exists a standard  $m$-dimensional Brownian motion $W = (W^1,\dots,W^m)$ starting from zero such that 
$\P_x$-a.s. for any $x=(x_1,...,x_d)\in \R^d$, $i=1,\dots,d$
\begin{equation}\label{weaksolutiondc-3}
X_t^i = x_i+ \sum_{j=1}^m \int_0^t \widehat{\sigma}_{ij} (X_s) \, dW_s^j +   \int^{t}_{0}   g_i(X_s) \, ds, \quad 0\le  t <\zeta,
\end{equation}
in short
$$
X_t = x+  \int_0^t \widehat{\sigma} (X_s) \, dW_s+   \int^{t}_{0}   \mathbf{G}(X_s) \, ds, \quad 0\le  t <\zeta.
$$
If \eqref{consuniquelaw} holds a.e. outside an arbitrarily large compact set, then $\mathbb{P}_x(\zeta=\infty)=1$ for all $x\in \mathbb{R}^d$ (cf. \cite[Theorem 4.2]{LT18}). 
\end{theo}
\begin{exam} \label{imporexam}
Given $p>d$, let $A=(a_{ij})_{1\leq i,j\leq d}$ be a symmetric matrix of functions on $\R^d$ which is locally uniformly strictly elliptic and $a_{ij} \in H^{1,p}_{loc}(\R^d) \cap C_{loc}^{0, 1-d/p}(\R^d)$ for all $1 \leq i,j \leq d$.\, Given $m \in \N$, let $\sigma=(\sigma_{ij})_{1 \le i \le d,1\le j \le m}$ be a matrix of functions satisfying $\sigma_{ij}\in C(\R^d)$ for all $1 \le i \le d,1\le j \le m$, such that $A =\sigma \sigma^T$. Let $\phi \in L^{\infty}_{loc}(\R^d)$ be such that for any open ball $B$, there exist strictly positive constants $c_B$, $C_B$ such that
$$
c_B \leq \phi(x) \leq C_B \quad \text{for every  }x\in  B.
$$
Let $\frac{1}{\psi}(x):=\frac{\|x\|^{\alpha}}{\phi(x)}$, $x\in \R^d$, for some $\alpha>0$ and consider following conditions.
\begin{itemize}
\item[(a)]
$\alpha p<d$,\; $\mathbf{G}\in L^{\infty}(B_{\varepsilon}(\mathbf{0})) \cap L^p_{loc}(\R^d \setminus \overline{B}_{\varepsilon}(\mathbf{0}))$ for some $\varepsilon>0$,
\item[(b)]
$2\alpha p <d$,\; $\mathbf{G}\in L^{2p}(B_{\varepsilon}(\mathbf{0})) \cap L^p_{loc}(\R^d \setminus \overline{B}_{\varepsilon}(\mathbf{0}))$ for some $\varepsilon>0$,
\item[(c)]
$\alpha \cdot (\frac{p}{2} \vee 2)<d$, \,$\mathbf{G}\equiv 0$ on $B_{\varepsilon}(\mathbf{0})$ and $\mathbf{G} \in L^{s}_{loc}(\R^d \setminus \overline{B_{\varepsilon}}(\mathbf{0}))$ for some $\varepsilon>0$, where $s>d$ so that $(\frac{p}{2} \vee 2)^{-1}+\frac{1}{s}<\frac{2}{d}$.
\end{itemize}
Either of the conditions (a), (b), or (c) imply {\bf (A1)}, {\bf (A2)}, {\bf (A3)}, {\bf (A4)}. Indeed, for arbitrary $\varepsilon>0$ take $q=p$, $s=\frac{pd}{2p-d}+\varepsilon$  in the case of (a), $q=2p$, $s=\frac{2pd}{4p-d}+\varepsilon$ in the case of (b), and $q=\frac{p}{2} \vee 2$, $s>d$ defined by (c) in the case of (c).
Assuming (a), (b) or (c), the Hunt process $\M$ as in Theorem \ref{weakexistence4} solves weakly \,$\P_x$-a.s.  for any $x \in \R^d$, 
\begin{equation}\label{weaksolutiondc}
X_t= x+\int_0^t \|X_s\|^{\alpha/2}  \cdot \frac{\sigma}{\sqrt{\phi}}  (X_s) \, dW_s +   \int^{t}_{0}   \mathbf{G}(X_s) \, ds, \quad 0\le  t <\zeta
\end{equation}
and is non-explosive if \eqref{consuniquelaw} holds a.e. outside an arbitrarily large compact set.
\end{exam}

\subsection{Uniqueness in law} \label{uniquenessinlaw3}

Consider 
\begin{itemize}
\item[{\bf (A4)$^\prime$}:] {\bf (A1)} holds with $p=2d+2$, {\bf(A2)} holds with some $q\in (2d+2,\infty]$,
$s\in (\frac{d}{2},\infty]$ is fixed, such that $\frac{1}{q}+ \frac{1}{s}< \frac{2}{d}$, and $\mathbf{G} \in L^{\infty}_{loc}(\R^d,\R^d)$.
\end{itemize}
%Note that if we assume {\bf (A4$^\prime$)}, then {\bf (A1)}, {\bf (A2)}, {\bf (A3)}, and {\bf (A4)} hold.
\begin{defn}\label{weaksolution}
Suppose {\bf (A1)}, {\bf (A2)}, {\bf (A3)}, {\bf (A4)} hold (for instance if {\bf (A4)$^\prime$} holds). Let the expression $\frac{1}{\psi}$ denote an arbitrary but fixed Borel measurable function satisfying $\psi\cdot \frac{1}{\psi}=1$ a.e. and $\frac{1}{\psi}(x)\in [0,\infty)$ for any $x\in \R^d$.
%and let $\widehat{\sigma}$ be as in Theorem \ref{weakexistence4}.
%, and consider fixed (real-valued) Borel measurable verdsions $g_i, 1\le i\le d, \frac{1}{\psi}$ of the functions appearing in {\bf (A4$^\prime$)}
Let
$$
\widetilde{\M} =  
(\widetilde{\Omega}, \widetilde{\F}, (\widetilde{\F}_t)_{t \ge 0}, (\widetilde{X}_t)_{t \ge 0}, (\widetilde{W}_t)_{t \ge 0}, (\widetilde{\P}_x)_{x \in \R^d})
$$
be such that for any $x=(x_1,...,x_d)\in \R^d$
\begin{itemize}
\item[(i)] $(\widetilde{\Omega}, \widetilde{\F}, (\widetilde{\F}_t)_{t \ge 0}, \widetilde{\P}_x)$ is a filtered probability space, satisfying the usual conditions,
\item[(ii)] $(\widetilde{X}_t=(\widetilde{X}_t^1,..., \widetilde{X}_t^d))_{t \ge 0}$ is an $(\widetilde{\F}_t)_{t \ge 0}$-adapted continuous $\R^d$-valued stochastic process, 
\item[(iii)] $(\widetilde{W}_t=(\widetilde{W}^1_t,...,\widetilde{W}^m_t))_{t \ge 0}$ is a standard $m$-dimensional  $((\widetilde{\F}_t)_{t \ge 0},\widetilde{\P}_x)$-Brownian motion starting from zero,
\item[(iv)]  for the (real-valued) Borel measurable functions $\widehat{\sigma}_{ij}, g_i, \frac{1}{\psi}$, $\widehat{\sigma}_{ij}=\sqrt{\frac{1}{\psi}}\sigma_{ij}$, with $\sigma$ is as in Theorem \ref{weakexistence4}, it holds 
$$
\widetilde{\P}_{x}\Big (\int_0^t (\widehat{\sigma}_{ij}^2(\widetilde{X}_s)+| g_i(\widetilde{X}_s)| \Big )ds<\infty,\ 1\le i\le  d,\  1\le j\le m,\  t\in [0,\infty),
$$
and for any $1\le i\le d$,
\begin{equation*}
\widetilde{X}_t^i = x_i+ \sum_{j=1}^m \int_0^t \widehat{\sigma}_{ij} (\widetilde{X}_s) \, d\widetilde{W}_s^j +   \int^{t}_{0}   g_i(\widetilde{X}_s) \, ds,
\quad 0\leq t < \infty, \quad \; \text{$\widetilde{\P}_{x}$-a.s.},
\end{equation*}
in short 
\begin{equation} \label{fepojpow}
\widetilde{X}_t = x + \int_0^t \widehat{\sigma}(\widetilde{X}_s) d\widetilde{W}_s + \int_0^t \mathbf{G}(\widetilde{X}_s) ds, \quad 0\leq t < \infty, \quad \; \text{$\widetilde{\P}_{x}$-a.s.}
\end{equation}
\end{itemize}
Then $\widetilde{\M}$ is called a {\bf weak solution} to \eqref{fepojpow}.
Note that in this case, $(t, \widetilde{\omega}) \mapsto \widehat{\sigma}(\widetilde{X}_t(\widetilde{\omega}))$ and  $(t, \widetilde{\omega}) \mapsto \mathbf{G}(\widetilde{X}_t(\widetilde{\omega}))$ are progressively measurable with respect to $(\widetilde{\F}_t)_{t \ge 0}$  and that
$$
\widetilde{D}_R:=\inf \{ t \geq 0 \mid \widetilde{X}_t \in \R^d \setminus B_R  \}\nearrow \infty \quad \widetilde{\P}_{x}\text{-a.s. for any }x\in \R^d.
$$ 
\end{defn}
\begin{rem}\label{explanationweaksolution}
(i) In Definition \ref{weaksolution} the (real-valued) Borel measurable functions $\widehat{\sigma}_{ij}, g_i, \frac{1}{\psi}$ are fixed. In particular, the solution and the integrals involving the solution in \eqref{fepojpow} may depend on the versions that we choose.
When we fix the Borel measurable version $\frac{1}{\psi}$ with $\frac{1}{\psi}(x)\in [0,\infty)$ for all $x\in \R^d$,  as in Definition \ref{weaksolution}, we always consider the corresponding extended Borel measurable function $\psi$ defined by 
$$
\psi(x):=\frac{1}{\frac{1}{\psi}(x)}, \quad \text{ if }\ \frac{1}{\psi}(x)\in (0,\infty), \qquad \psi(x):=\infty, \quad \text{ if }\ \frac{1}{\psi}(x)=0.
$$
Thus the choice of the special version for $\psi$ depends on the previously chosen Borel measurable version $\frac{1}{\psi}$.\\
(ii) If $\M$ of Theorem \ref{weakexistence4} is non-explosive (has infinite lifetime for any starting point), then it is a weak solution to  \eqref{fepojpow}. Thus a weak solution to \eqref{fepojpow} exists just under assumptions {\bf (A1)}, {\bf (A2)}, {\bf (A3)}, {\bf (A4)} and a suitable growth condition (cf. Remark \ref{mostresults4}) on the coefficients. For this special weak solution we know that integrals involving the solution do not depend on the chosen Borel versions. This follows similarly to \cite[Lemma 3.14(i)]{LT18}.
\end{rem}
\begin{theo}[\bf Local Krylov type estimate] \label{krylovtype}
Assume {\bf (A4)$^\prime$} and let $\widetilde{\M}$ be a weak solution to \eqref{fepojpow}. Let
$$
Z^{\widetilde{\M}}(\widetilde{\omega}):=\{ t \ge 0 \mid \sqrt{\frac{1}{\psi}} (\widetilde{X}_t(\widetilde{\omega})) = 0  \}
$$
and
$$
\Lambda(Z^{\widetilde{\M}}):=  \big\{ \widetilde{\omega} \in \widetilde{\Omega} \mid dt\big (Z^{\widetilde{\M}}(\widetilde{\omega})\big ) =0 \big \}.
$$
and assume that
\begin{eqnarray}\label{zeros}
\widetilde{\P}_x(\Lambda(Z^{\widetilde{\M}})) =1 \quad \text{ for all }  x \in \R^d. 
\end{eqnarray}
Let $x \in \R^d$, $T>0$, $R>0$ and $f \in L^{2d+2, d+1}(B_R \times (0,T))$. Then there exists a constant $C>0$ which is independent of $f$ such that
$$
\widetilde{\E}_{x} \left[ \int_0^{T\wedge \widetilde{D}_R} f(\widetilde{X}_s,s) ds \right] \leq C \| f\|_{L^{2d+2, d+1}(B_R \times (0,T))},
$$
where $\widetilde{\E}_{x}$ is the expectation w.r.t. $\widetilde{\P}_{x}$.
\end{theo}
\begin{proof}
Let $g \in L^{d+1}(B_R \times (0,T))$. (Note: all functions defined on $B_R \times (0,T)$ are trivially extended on  $\R^d \times (0, \infty) \setminus B_R \times (0,T)$.) \, Using \cite[2. Theorem (2), p. 52]{Kry},  there exists a constant $C_1>0$ which is independent of $g$, such that
\begin{eqnarray*}
&& \widetilde{\E}_x \left[ \int_{(0, {T \wedge \widetilde{D}_R}) \setminus  Z^{\widetilde{\M}}} \big (2^{-\frac{d}{d+1}}  \text{det}(A)^{\frac{1}{d+1}}\cdot  \Big (\frac{1}{\psi}\Big )^{\frac{d}{d+1}}  g\big )(\widetilde{X}_s,s) ds \right] \\
&=&\widetilde{\E}_x \left[ \int_0^{T \wedge \widetilde{D}_R}  \big (2^{-\frac{d}{d+1}}  \text{det}(A)^{\frac{1}{d+1}}\cdot  \Big (\frac{1}{\psi}\big )^{\frac{d}{d+1}}  g \big  )(\widetilde{X}_s,s) ds \right] \\
&\leq& e^{ T \|\mathbf{G} \|_{L^{\infty}(B_{R})} } \cdot \widetilde{\E}_x \left[ \int_0^{T \wedge \widetilde{D}_R} e^{-\int_0^s \|\mathbf{G}(\widetilde{X}_u)\| du}\cdot \text{det}\big(\widehat{A}/2\big)^{\frac{1}{d+1}}g(\widetilde{X}_s,s) ds \right] \\
&\leq& e^{ T \|\mathbf{G} \|_{L^{\infty}(B_{R})} }  \cdot C_1 \|g\|_{L^{d+1}(B_R \times (0,\infty))} \\
&=& e^{ T \|\mathbf{G} \|_{L^{\infty}(B_{R})} }   \cdot C_1 \|g\|_{L^{d+1}(B_R \times (0,T))}.
\end{eqnarray*}
Let $f \in L^{2d+2, d+1}(B_R \times (0,T))$. Let $\psi$ denote the extended Borel measurable version as explained in Remark \ref{explanationweaksolution}(i). Note that 
$$
Z^{\widetilde{\M}}(\widetilde{\omega})=\{s\ge 0\,| \Big (\frac{1}{\psi}\Big )^{\frac{d}{d+1}} (\widetilde{X}_s(\widetilde{\omega}))\psi^{\frac{d}{d+1}} (\widetilde{X}_s(\widetilde{\omega}))\not= 1\}.
$$
Hence by \eqref{zeros}  
$$
\widetilde{\P}_x\big (   dt (\{ s\ge 0\,| \Big (\frac{1}{\psi}\Big )^{\frac{d}{d+1}} (\widetilde{X}_s)\psi^{\frac{d}{d+1}} (\widetilde{X}_s)\not= 1\})    =0 \big )=1.
$$
Thus   replacing $g$ with $2^{\frac{d}{d+1}} \cdot \text{det}(A)^{-\frac{1}{d+1}} \psi^{\frac{d}{d+1}} f$, we get
\begin{eqnarray*}
&&\widetilde{\E}_{x} \left[ \int_0^{T\wedge \widetilde{D}_R} f(\widetilde{X}_s,s) ds \right] =\widetilde{\E}_{x} \left[ \int_{(0, {T \wedge \widetilde{D}_R}) \setminus Z^{\widetilde{\M}}} f(\widetilde{X}_s,s) ds \right] \\
&\leq& e^{ T \|\mathbf{G} \|_{L^{\infty}(B_{R})} }   \cdot C_1 \|2^{\frac{d}{d+1}} \cdot \text{det}(A)^{-\frac{1}{d+1}} \psi^{\frac{d}{d+1}} f\|_{L^{d+1}(B_R \times (0,T))} \\
&\leq& \underbrace{2^{\frac{d}{d+1}}e^{ T \|\mathbf{G} \|_{L^{\infty}(B_{R})} }   \cdot C_1 \|\text{det}(A)^{-\frac{1}{d+1}} \|_{L^{\infty}(B_R)} \|\psi\|^{\frac{2d}{2d+2}}_{L^{2d}(B_R)}}_{=:C} \|f\|_{L^{2d+2,d+1}(B_R \times (0,T))}.\\
\end{eqnarray*}
\end{proof}
\centerline{}
Using Theorem \ref{krylovtype} and \eqref{sfpolkd}, the proof of the following lemma is straightforward.
 \begin{lem}\label{condforpsitohold}
Let $\widetilde{\M}$ be a weak solution to \eqref{fepojpow} . 
Then either of the following conditions implies \eqref{zeros}:
\begin{itemize}
\item[(i)] $\frac{1}{\psi}(x)\in (0, \infty)$ for all $x \in \R^d$.
\item[(ii)] For each $n\in\N, T>0$ and $x\in \R^d$ it holds
$$
 \widetilde{\E}_x \left[\int_0^T 1_{B_n}\psi(\widetilde{X}_s)ds \right ]<\infty,
$$
where $\psi$ denotes the extended Borel measurable version as explained in Remark \ref{explanationweaksolution}(i). Moreover \eqref{psiconditionintro} is equivalent to \eqref{zeros}.
\end{itemize}
In particular, 
if the weak solution that is constructed in Theorem \ref{weakexistence4} is non-explosive, then (ii) always holds for this solution  and 
\eqref{zeros} implies in general that integrals of the form $\int_0^t f(\widetilde{X}_s,s)ds$ are, whenever they are well-defined, independent of the particular Borel version that is chosen for $f$.
\end{lem}

\begin{theo}[\bf Local It\^{o}-formula]\label{fefere}
Assume {\bf (A4)$^\prime$} and let $\widetilde{\M}$ be a weak solution to \eqref{fepojpow} such that \eqref{zeros} holds.
%$\widetilde{\P}_{x}(\Lambda(Z^{\widetilde{\M}}))=1$ for all $x \in \R^d$.
Let $R_0>0$, $T>0$. Let $u \in W^{2,1}_{2d+2}(B_{R_0} \times (0,T)) \cap C(\overline{B}_{R_0} \times [0, T] )$ be such that $\|\nabla u\| \in L^{4d+4}(B_{R_0} \times (0,T))$. Let $R>0$ with $R<R_0$. Then $\widetilde{\P}_{x}$-a.s. for any $x \in \R^d$, 
$$
u(\widetilde{X}_{T\wedge \widetilde{D}_R}, T \wedge \widetilde{D}_R) - u(x,0) = \int_0^{ T \wedge \widetilde{D}_R} \nabla u(\widetilde{X}_s,s) \widehat{\sigma}(\widetilde{X}_s) d\widetilde{W}_s+
\int_0^{ T\wedge \widetilde{D}_R} (\partial_t u + Lu) (\widetilde{X}_s,s) ds, 
$$
where $Lu:= \frac12 {\rm trace}(\widehat{A} \nabla^2 u)+ \langle \mathbf{G}, \nabla u \rangle$.

\end{theo}
\begin{proof}
Take $T_0>0$ satisfying $T_0>T$. Extend $u$ to $\overline{B}_{R_0} \times [-T_0,T_0]$ by
$$
u(x, t)  = u(x, 0) \; \text{ for } -T_0 \leq t<0, \quad  u(x, t)  = u(x, T) \; \text{ for } T<t \leq T_0, \, x \in \overline{B}_{R_0}.
$$
Then it holds 
$$
u  \in W^{2,1}_{2d+2}(B_{R_0} \times (0,T)) \cap C(\overline{B}_{R_0} \times [-T, T] ) \; \text{ and } \;\|\nabla u\| \in L^{4d+4}(B_{R_0} \times (-T_0,T_0)).
$$
For sufficiently large $n \in \N$, let $\zeta_{n}$ be a standard mollifier on $\R^{d+1}$ and $u_n:= u * \zeta_n$. Then it holds $u_n \in C^{\infty}(\overline{B}_R \times [0, T])$, such that $\lim_{n \rightarrow \infty} \|u_n-u\|_{W^{2,1}_{2d+2}(B_R \times (0,T))}=0$ and $\lim_{n \rightarrow \infty} \|\nabla u_n-\nabla u\|_{L^{4d+4}(B_R \times (0,T))}=0$ . By It\^{o}'s formula,  for $x \in \R^d$, it holds for any $n\geq 1$
\begin{eqnarray} 
&&u_n(\widetilde{X}_{T\wedge \widetilde{D}_R}, T \wedge \widetilde{D}_R) - u_n(x,0) \nonumber \\[5pt]
&=& \int_0^{ T \wedge \widetilde{D}_R} \nabla u_n(\widetilde{X}_s,s)\,\widehat{\sigma}(\widetilde{X}_s) d\widetilde{W}_s+\int_0^{ T\wedge \widetilde{D}_R}(\partial_t u_n + Lu_n) (\widetilde{X}_s,s) ds,  \;\; \;\; \widetilde{\P}_{x}\text{-a.s. } \qquad \;\; \label{itoaprox}
\end{eqnarray}
By Sobolev embedding, there exists a constant $C>0$, independent of $u_n$ and $u$, such that
$$
\displaystyle \sup_{\overline{B}_R \times [0,T]} |u_n-u| \leq  C\|u_n-u\|_{W^{2,1}_{2d+2}(B_R \times (0,T))}.
$$
Thus $\lim_{n \rightarrow \infty} u_n(x,0) = u(x,0)$ and 
$$
u_n(X_{ T \wedge \widetilde{D}_R}, T \wedge \widetilde{D}_R)  \, \text{ converges } \text{ $\P_x$-a.s. to } \, u(X_{ T \wedge \widetilde{D}_R}, T \wedge \widetilde{D}_R) \, \text{ as }\, n \rightarrow \infty.
$$ 
By Theorem \ref{krylovtype},
\begin{eqnarray*}
&&\widetilde{\mathbb{E}}_x\left[ \left| \int_0^{T\wedge \widetilde{D}_R} (\partial_t u_n + Lu_n) (\widetilde{X}_s,s)ds  -\int_0^{T\wedge \widetilde{D}_R} (\partial_t u + Lu) (\widetilde{X}_s,s)ds  \right|\right] \\
&&\leq \widetilde{\E}_x \left[\int_0^{T \wedge \widetilde{D}_R} |\partial_t u-\partial_t u_n|(\widetilde{X}_s,s) ds \right] + \widetilde{\E}_x \left[ \int_0^{T \wedge \widetilde{D}_R} |Lu-Lu_n|(\widetilde{X}_s,s)ds   \right] \\
&&\leq C\|\partial_t u_n- \partial_t u\|_{L^{2d+2, d+1}(B_R \times (0,T))} + C\|Lu-Lu_n\|_{L^{2d+2, d+1}(B_R \times (0,T))} \\
&& \; \longrightarrow 0 \quad \; \text{as } n \rightarrow \infty, \;
\end{eqnarray*}
where $C>0$ is a constant which is independent of $u$ and $u_n$. Using Jensen's inequality, It\^{o} isometry, and Theorem \ref{krylovtype}, we obtain
\begin{eqnarray*}
&&\widetilde{\E}_x \left[\int_0^{ T \wedge \widetilde{D}_R} \left( \nabla u_n(\widetilde{X}_s,s) - \nabla u(\widetilde{X}_s,s) \right)\,\widehat{\sigma}(\widetilde{X}_s) dW_s\right] \\
&&\leq  \widetilde{\E}_x \left[ \left| \int_0^{ T \wedge \widetilde{D}_R} \left( \nabla u_n(\widetilde{X}_s,s) - \nabla u(\widetilde{X}_s,s) \right)\,\widehat{\sigma}(\widetilde{X}_s) dW_s \right|^2 \right]^{1/2} \\
&&= \widetilde{\E}_x \left[ \int_0^{ T \wedge \widetilde{D}_R}\big \|  \left( \nabla u_n(\widetilde{X}_s,s) - \nabla u(\widetilde{X}_s,s) \right)\,\widehat{\sigma} (\widetilde{X}_s) \big \| ^2 ds\right]^{1/2}  \\
&&\leq C\| (\nabla u_n - \nabla u) \widehat{\sigma} \|_{L^{4d+4, 2d+2}(B_R \times (0,T))} \\
&&\leq CC'\| \widehat{\sigma} \|_{L^{\infty}(B_R)} \| \nabla u_n- \nabla u\|_{L^{4d+4, 2d+2}(B_R \times (0,T))} \longrightarrow 0 \; \text{ as } n \rightarrow \infty.
\end{eqnarray*}
Letting $n \rightarrow \infty$ in \eqref{itoaprox}, the assertion holds.
\end{proof}
\begin{theo} \label{feokokoe}
Assume {\bf (A4)$^\prime$} and let $q_0>2d+2$ be such that $\frac{1}{q_0}+\frac{1}{q} = \frac{1}{2d+2}$. If $u \in D(L_{q_0})$, then $u \in H_{loc}^{2,2d+2}(\R^d)$. Moreover, for any open ball $B$ in $\R^d$, there exists a constant $C>0$, independent of $u$, such that 
$$
\|u\|_{H^{2,2d+2}(B)} \leq C \|u\|_{D(L_{q_0})}.
$$
\end{theo}
\begin{proof}
By the assumption {\bf (A4)$^\prime$} and Theorem \ref{helholmop}, $\rho \in H_{loc}^{1,2d+2}(\R^d) \cap C_{loc}^{0, 1-\frac{d}{2d+2}}(\R^d)$ and $\rho \psi \mathbf{B} \in L^{2d+2}_{loc}(\R^d)$.
%First let $f \in C_0^{\infty}(\R^d)$ and $\alpha>0$. Then $G_{\alpha} f  \in D(\overline{L})_b$ and by Theorem \ref{mainijcie} (c)
%\begin{eqnarray*}
%&&\mathcal{E}^0(G_{\alpha} f, \varphi) - \int_{\R^d} \langle \mathbf{B}, \nabla G_{\alpha} f \rangle d \mu\\
%&&\quad =  -\int_{\R^d} \left (\overline{L}\, \overline{G}_{\alpha} f \right) \varphi \,d \mu \\
%&& \quad = \int_{\R^d} f - \alpha G_{\alpha} f \,d\mu, \quad \text{ for all } \varphi \in C_0^{\infty}(\R^d).
%\end{eqnarray*}
Let $f \in C_0^{\infty}(\R^d)$ and $\alpha>0$.
Then by \eqref{feppwikmee}
\begin{eqnarray}
&&\int_{\R^d} \big \langle \frac12 \rho A \nabla G_{\alpha}f, \nabla \varphi \big \rangle dx - \int_{\R^d} \langle \rho \psi \mathbf{B}, \nabla G_{\alpha} f \rangle \varphi\, dx  +\int_{\R^d} (\alpha \rho \psi G_{\alpha} f) \,\varphi dx  \nonumber \\
&& = \int_{\R^d} (\rho \psi f) \,\varphi dx, \quad \; \text{ for all } \varphi \in C_0^{\infty}(\R^d). \label{fewbllle}
\end{eqnarray}
Let $\widetilde{q}:= \left( \frac{1}{2d+2}+\frac{1}{d}  \right)^{-1}$. Then $\alpha \rho \psi \in L^{2d+2}_{loc}(\R^d) \subset L^{\widetilde{q}}_{loc}(\R^d)$, $\rho \psi f \in L^{2d+2}_{loc}(\R^d) \subset  L^{\widetilde{q}}_{loc}(\R^d)$, hence by \cite[Theorem 1.8.3]{BKRS}, 
$G_{\alpha} f \in H^{1,2d+2}_{loc}(\R^d)$. Moreover, using \cite[Theorem 1.7.4]{BKRS} and the resolvent contraction property, for any open balls $V$, $V'$ in $\R^d$ with $\overline{V} \subset V'$, there exists a constant $\widetilde{C}>0$, independent of $f$, such that %has continuous $dx$-version $R_{\alpha} f$ on $\R^d$. \\
\begin{eqnarray}
\quad \|G_{\alpha}f\|_{H^{1,2d+2}(V)}  &\leq& \widetilde{C}(\|G_{\alpha}f\|_{L^1(V')} + \| \rho \psi f \|_{L^{\widetilde{q}}(V')})  \nonumber  \\
&\leq& \widetilde{C}(\|G_{\alpha}f\|_{L^1(V')} + \| \rho \psi  \|_{L^{2d+2}(V')} \|f\|_{L^d(V')})   \nonumber  \\
%&\leq& C( |B'|^{1-1/d} \|G_{\alpha} f\|_{L^d(B')}  + \|\rho \psi \|_{L^{2d+2}(B')}\|f\|_{L^d(B')} ) \\
&\leq& \widetilde{C}\,  \widetilde{C}_1 \|f\|_{L^{q_0}(\R^d, \mu)}, \qquad \;\; \label{wefweft}
\end{eqnarray}
where  $\widetilde{C}_1:=  \big(\frac{1}{ \inf \rho \psi}\big)^{\frac{1}{q_0}}(\alpha^{-1}|V'|^{1-\frac{1}{q_0}}  + \|\rho \psi \|_{L^{2d+2}(V')} |V'|^{\frac{1}{d}-\frac{1}{q_0}})$.\, Using Morrey's inequality and \eqref{wefweft}, there exists a constant $\widetilde{C}_2>0$ which is independent of $f$ such that
\begin{equation} \label{vlfdygoe4}
\|G_{\alpha}f\|_{L^{\infty}(V)} \leq  \widetilde{C}_2 \widetilde{C} \widetilde{C}_1 \|f\|_{L^{q_0}(\R^d, \mu)}.
\end{equation}
 Now set
$$
h_1:= \left \langle  \rho \psi \mathbf{B}, \nabla G_{\alpha} f \right \rangle - \alpha \rho \psi G_{\alpha} f  + \rho \psi f \in L^{d+1}_{loc}(\R^d).
$$
Then \eqref{fewbllle} implies
\begin{eqnarray} \label{weoijoir2}
\int_{\R^d} \big \langle \frac12 \rho A \nabla G_{\alpha}f, \nabla \varphi \big \rangle dx = \int_{\R^d} h_1 \varphi dx, \quad \text{ for all } \varphi \in C_0^{\infty}(\R^d).
\end{eqnarray}
Let $U_{1}$, $U_{2}$ be open balls in $\R^d$ satisfying $\overline{B} \subset U_{1} \subset \overline{U}_{1} \subset U_{2}$. Take  $\zeta_1 \in C_0^{\infty}(U_{2})$ such that $\zeta_1 \equiv 1$ on $\overline{U}_{1}$.  Then using integration by parts, and \eqref{weoijoir2}
\begin{eqnarray}    
&&\hspace{-1em} \int_{U_{2}}  \langle \frac12 \rho A \nabla (\zeta_1 G_{\alpha} f), \nabla \varphi  \rangle dx \nonumber  = \int_{U_{2}} \langle  \frac12 \rho A \nabla G_{\alpha} f,   \zeta_1 \nabla \varphi  \rangle dx + \int_{U_{2}} \frac12 \langle  A \nabla \zeta_1, \nabla \varphi  \rangle  \rho G_{\alpha} f   dx \nonumber \\
&& \hspace{-1em} = \int_{U_{2}} \langle \frac12 \rho A \nabla G_{\alpha} f, \nabla (\zeta_1 \varphi)  \rangle  dx - \int_{U_2} \underbrace{  \langle \frac12 \rho A \nabla G_{\alpha} f,  \nabla \zeta_1 \rangle}_{=:h_2} \varphi dx\nonumber  \\
&&   \hspace{-0.6em}+ \int_{U_2} \underbrace{-\frac12  \big( \langle G_{\alpha} f \nabla \rho + \rho \nabla G_{\alpha}f , A \nabla \zeta_1 \rangle + \rho G_{\alpha}f  \langle \nabla A, \nabla \zeta_1 \rangle + \rho G_{\alpha} f \text{trace}(A \nabla ^2 \zeta_1 ) \big)}_{=:h_3}\varphi dx \nonumber \\
&&\hspace{-1em}=  \int_{U_2}  (h_1\zeta_1 - h_2 + h_3) \varphi dx, \quad \text{ for all } \varphi \in C_0^{\infty}(U_2). \label{poijoijw34}
\end{eqnarray}
Note that $h_2, h_3 \in L^{2d+2}_{loc}(\R^d)$. Let $h_4:=\langle \frac12 \nabla (\rho A), \nabla (\zeta_1 G_{\alpha}f) \rangle \in L^{d+1}_{loc}(\R^d)$. Using \eqref{poijoijw34},
\begin{eqnarray}
&&\int_{U_2} \langle \frac12 \rho A \nabla (\zeta_1 G_{\alpha} f), \nabla \varphi \rangle dx +\int_{U_2} \langle \frac12 \nabla (\rho A), \nabla (\zeta_1 G_{\alpha}f) \rangle\varphi dx \nonumber \\
&=& \int_{U_2}  (h_1\zeta_1 - h_2 + h_3 +h_4) \varphi dx, \quad \text{ for all } \varphi \in C_0^{\infty}(U_2). \label{pkpekek}
\end{eqnarray}
We have  $h:=h_1\zeta_1-h_2+h_3+h_4 \in L_{loc}^{d+1}(\R^d)$ and
\begin{equation} \label{oijwoijmn34}
\|h\|_{L^{d+1}(U_2)} \leq C_2(\|G_{\alpha}f\|_{H^{1,2d+2}(U_2)} + \|\rho \psi f\|_{L^{d+1}(U_2)} ),
\end{equation}
where $C_2>0$ is a constant which is independent of $f$. By \cite[Theorem 9.15]{Gilbarg}, there exists $w \in H^{2,d+1}(U_2) \cap H^{1,d+1}_0(U_2)$ such that
\begin{equation} \label{jpojiijiw}
-\frac12 \text{trace}(\rho A \nabla^2 w)  = h \;\; \text{ a.e. on } U_2.
\end{equation}
Furthermore, using \cite[Lemma 9.17]{Gilbarg}, \eqref{oijwoijmn34}, \eqref{wefweft}, there exists a constant $C_1>0$ which is independent of $f$ such that
\begin{eqnarray*}
\|w\|_{H^{2,d+1}(U_2)} &\leq& C_1 \| h\|_{L^{d+1}(U_2)} \\
&\leq& C_1 C_2 \left(\|G_{\alpha} f \|_{H^{1,2d+2}(U_2)}+ \| \rho \psi f\|_{L^{d+1}(U_2)}\right) \\
&\leq& C_1 C_2  C_3 \|f\|_{L^{q_0}(\R^d, \mu)},
\end{eqnarray*}
where $C_3:=\widetilde{C}_1 +  \|\rho \psi \|_{L^{2d+2}(U_2)} |U_2|^{\frac{1}{2d+2}-\frac{1}{q_0}} \big (\frac{1}{ \inf_{U_2} \rho \psi} \big)^{\frac{1}{q_0}}$. Note that \eqref{jpojiijiw} implies
\begin{eqnarray}
&&\int_{U_2} \langle \frac12 \rho A \nabla w, \nabla \varphi \rangle dx +\int_{U_2} \langle \frac12 \nabla (\rho A), \nabla w \rangle\varphi dx \nonumber \\
&=& \int_{U_2}  h \varphi dx, \quad \text{ for all } \varphi \in C_0^{\infty}(U_2). \label{oiowheo3}
\end{eqnarray}
Using the maximum principle  of \cite[Theorem 1]{T77} and comparing \eqref{oiowheo3} with \eqref{pkpekek}, we obtain $\zeta G_{\alpha} f = w$ on $U_2$, hence $G_{\alpha} f = w$ on $U_1$, so that $G_{\alpha} f \in H^{2,d+1}(U_1)$. Therefore, by Morrey's inequality, we obtain $\partial_i G_{\alpha}f \in L^{\infty}(U_1)$, \,$1 \leq i \leq d$, and
\begin{eqnarray}
\| \partial_i G_{\alpha} f \|_{L^{\infty}(U_1)} &\leq& C_4\|G_{\alpha}f \|_{H^{2,d+1}(U_1)} \nonumber \\
&\leq& C_4\|w\|_{H^{2,d+1}(U_2)} \nonumber  \\
&\leq& C_1 C_2 C_3 C_4\|f\|_{L^{q_0}(\R^d, \mu)}, \label{weikjpoi42}
\end{eqnarray}
where $C_4>0$ is a constant which is independent of $f$. Thus we obtain $h \in L^{2d+2}(U_1)$.
%Observe that
%$$
%-\frac12 \text{trace}(\rho A \nabla^2 G_{\alpha} f)  = h+ \left \langle \frac12 \nabla (\rho A), \nabla ( G_{\alpha}f)  %\right \rangle \quad  \text{ on } U_1. 
%$$
Now take $\zeta_2 \in C_0^{\infty}(U_1)$ such that $\zeta_2 \equiv 1$ on $\overline{B}$. Note that a.e. on $U_1$ it holds
\begin{eqnarray*}
&&-\frac12 \text{trace}\big(\rho A \nabla^2 (\zeta_2 G_{\alpha} f)\big) \\
&&= -\frac12 \zeta_2 \cdot \text{trace}(\rho A \nabla^2  G_{\alpha} f) - \frac12 G_{\alpha} f \cdot \text{trace}(\rho A \nabla^2 \zeta_2)  -\langle \rho A \nabla \zeta_2, \nabla G_{\alpha} f \rangle.  \\
&&= -\frac{1}{2} \zeta_2 h- \frac12 G_{\alpha} f \cdot \text{trace}(\rho A \nabla^2 \zeta_2)  -\langle \rho A \nabla \zeta_2, \nabla G_{\alpha} f \rangle=:\widetilde{h}.
\end{eqnarray*}
Since $\| \nabla G_{\alpha} f \| \in L^{\infty}(U_1)$, $\widetilde{h} \in L^{2d+2}(U_1)$, by \cite[Theorem 9.15]{Gilbarg}, we get $\zeta_2 G_{\alpha} f \in H^{2,2d+2}(U_1)$, hence $G_{\alpha}f \in H^{2,2d+2}(B)$. Using \cite[Lemma 9.17]{Gilbarg}, \eqref{vlfdygoe4}, \eqref{weikjpoi42}, there exist positive constants $C_5$, $C_6$ which are independent of $f$ such that
\begin{eqnarray}
\|G_{\alpha} f \|_{H^{2,2d+2}(B)} &\leq& \| \zeta_2  G_{\alpha}  f \|_{H^{2,2d+2}(U_1)} \nonumber   \\
&\leq& C_5 \|\widetilde{h}\|_{L^{2d+2}(U_1)}\nonumber  \\
&\leq& C_5 C_6 (\|f\|_{L^{q_0}(\R^d, \mu)} + \|\rho \psi f \|_{L^{2d+2}(U_1)})\nonumber   \\
&\leq&  C_5 C_6 (\|f\|_{L^{q_0}(\R^d, \mu)} + \|\rho \psi\|_{L^{q}(U_1)} (\inf_{U_1} \rho \psi)^{-1/q_0}\|f\|_{L^{q_0}(\R^d, \mu)}) \quad \;\;\nonumber  \\
&\leq& C\|f\|_{L^{q_0}(\R^d, \mu)},  \qquad \; \label{oiowieji3no3}
\end{eqnarray}
where $C:= C_5 C_6 (1 \vee  \|\rho \psi\|_{L^{q}(U_1)} (\inf_U \rho \psi)^{-1/q_0})$.  Using the denseness of $C_0^{\infty}(\R^d)$ in $L^{q_0}(\R^d, \mu)$, \eqref{oiowieji3no3} extends to  $f \in L^{q_0}(\R^d, \mu)$. Now let $u \in D(L_{q_0})$, Then $(1-L_{q_0})u \in L^{q_0}(\R^d, \mu)$, hence by \eqref{oiowieji3no3}, it holds $u = G_1(1-L_{q_0}) u \in H^{2,2d+2}_{loc}(\R^d)$ and
\begin{eqnarray*}
\|u\|_{H^{2,2d+2}(B)} &=& \| G_1(1-L_{q_0}) u \|_{H^{2,2d+2}(B)}  \\
&\leq& C\|(1-L_{q_0})u \|_{L^{q_0}(\R^d, \mu)} \\
&\leq& C\|u\|_{D(L_{q_0})}. 
\end{eqnarray*}
\end{proof}
\begin{theo} \label{pjpjwoeij}
Assume {\bf (A1)}, {\bf (A2)}. Let $f \in D(\overline{L})_b \cap D(L_s)  \cap D(L_2)$ and define 
$$
u_f:=P_{\cdot} f \in C(\R^d \times [0, \infty))
$$ 
as in Lemm \ref{contiiokc}. Then for any open set $U$ in $\R^d$ and $T>0$, 
$$
\partial_t u_f, \, \partial_i u_f \in L^{2, \infty}(U \times (0,T))  \text{ for all } 1 \leq i \leq d,
$$
and for each $t \in (0,T)$, it holds
$$
\partial_t u_f(\cdot,t) = T_t L_2 f \in L^{2}(U), \; \text{ and }\; \partial_i u_f(\cdot, t) = \partial_i P_t f  \in L^2(U).
$$
If we additionally assume {\bf (A4)$^\prime$} and  $f \in D(L_{q_0})$ where $q_0$ is as in Theorem \ref{feokokoe}, then $\partial_i \partial_j u_f  \in L^{2d+2, \infty}(U \times (0,T))$  for all $1 \leq i,j \leq d$, and for each $t \in (0, T)$, it holds
$$
\partial_i \partial_j u_f (\cdot, t ) = \partial_i \partial_j P_t f \in L^{2d+2}(U).
$$
\end{theo}
\begin{proof}
Assume {\bf (A1)}, {\bf (A2)}. Let $f \in D(\overline{L})_b \cap D(L_s)  \cap D(L_2)$ and $t>0$, $t_0 \geq 0$.  Then by Theorem \ref{mainijcie}(c), 
%$P_t f$, $P_{t_0} f$ as in  Lemm \ref{contiiokc} satisfy
$$
P_{t_0} f = \overline{T}_{t_0} f \in  D(\overline{L})_b \subset D(\mathcal{E}^0),
$$    
where $\overline{T}_0 := id$. Observe that by Theorem \ref{mainijcie}(c), for any open ball  $B$ in $\R^d$ with $\overline{U} \subset B$,
\begin{eqnarray}
&& \; \quad \| \nabla P_t f - \nabla P_{t_0}f \|^2_{L^2(B)}  \nonumber \\
&&\leq  (\lambda_B \inf_{B} \rho)^{-1}  \int_{B} \langle A\nabla(P_t f- P_{t_0}f), \nabla(P_t f- P_{t_0}f) \rangle \rho dx  \nonumber \\
&&\leq   2(\lambda_B \inf_{B} \rho)^{-1}\mathcal{E}^0(P_t f - P_{t_0}f, P_t f- P_{t_0}f)  \nonumber \\
&&\leq  2(\lambda_B \inf_{B} \rho)^{-1}  \int_{\R^d} -\overline{L}(\overline{T}_t f- \overline{T}_{t_0}f) \cdot (\overline{T}_t f- \overline{T}_{t_0}f) d\mu \nonumber  \\
&&\leq 4 (\lambda_B \inf_{B} \rho)^{-1}  \|f\|_{L^{\infty}(\R^d, \mu)} \|\overline{T}_t\overline{L}f - \overline{T}_{t_0} \overline{L}f \|_{L^1(\R^d, \mu)}. \label{feftlklere}
\end{eqnarray}
Likewise, 
\begin{eqnarray*} \label{fekokwe2}
 \| \nabla P_t f \|^2_{L^2(B)} \leq  2(\lambda_B \inf_{B} \rho)^{-1}  \|f\|_{L^{\infty}(\R^d, \mu)} \|\overline{T}_t  \overline{L} f\|_{L^1(\R^d,\mu)}.
\end{eqnarray*}
For each $i=1, \dots, d$, define a map
$$
\partial_i P_{\cdot} f: [0,T] \rightarrow L^2(U), \; \;  t \mapsto \partial_i P_{t} f.
$$
Then by \eqref{feftlklere} and the $L^1(\R^d, \mu)$-strong continuity of $(\overline{T}_t)_{t>0}$, the map $\partial_i P_{\cdot} f$ is continuous with respect to the $\| \cdot \|_{L^2(B)}$-norm, hence by \cite[Theorem, p91]{Nev65}, there exists a Borel measurable function $u^i_f$ on $U \times (0,T)$ such that for each $t \in (0,T)$ it holds
$$
u^i_f(\cdot, t) = \partial_i P_t f \in L^2(U).
$$
Thus using \eqref{feftlklere} and the $L^1(\R^d, \mu)$-contraction property of $(\overline{T}_t)_{t >0}$, it holds $u^i_f \in L^{2, \infty}(U \times (0,T))$ and
\begin{eqnarray*}
\|u^i_f \|_{L^{2, \infty}(U \times (0,T))} &=& \sup_{t \in (0,T)} \| \partial_i P_t  f\|_{L^2(U)} \\
&\leq&  2(\lambda_B \inf_{B} \rho)^{-1/2}  \|f\|^{1/2}_{L^{\infty}(\R^d, \mu)} \|\overline{L} f\|^{1/2}_{L^1(\R^d,\mu)}.
\end{eqnarray*}
Now let $\varphi_1 \in C_0^{\infty}(U)$ and $\varphi_2 \in C_0^{\infty}((0,T))$. Then
\begin{eqnarray}
\iint_{U \times (0,T)} u_f \cdot \partial_i(\varphi_1 \varphi_2) dx dt &=& \int_0^T \big(\int_U P_t f \cdot \partial_i \varphi_1 dx \big) \varphi_2 dt  \nonumber \\
&=& \int_0^T -\big( \int_U \partial_i P_t f  \cdot \varphi_1 dx \big) \varphi_2 dt \nonumber \\
&=& -\iint_U u_f^i  \cdot \varphi_1 \varphi_2 dxdt.   \label{pokpoin3}
\end{eqnarray}
Using the approximation as in Lemma \ref{stoneweier},  $\partial_i u_f = u^i_f \in L^{2,\infty}(U \times (0,T))$.\\
Now define a map
$$
T_{\cdot}L_2 f: [0,T] \rightarrow L^2(U), \; \; \, t \mapsto T_t L_2 f,
$$
where $T_{0}:= id$. \,Since
\begin{eqnarray*}
\|T_t L_2 f- T_{t_0} L_2 f\|_{L^2(U)} \leq (\inf_{U} \rho \psi)^{-1/2} \|T_t L_2 f- T_{t_0} L_2 f\|_{L^2(\R^d, \mu)},
\end{eqnarray*}
using the $L^2(\R^d, \mu)$-strong continuity of $(T_t)_{t>0}$ and \cite[Theorem, p91]{Nev65}, there exists a Borel measurable function $u^0_f$ on $U \times (0,T)$ such that for each $t \in (0,T)$ it holds
$$
u^0_f(\cdot, t) = T_t L_2 f \in L^2(U).
$$
Using the $L^2(\R^d, \mu)$-contraction property of $(T_t)_{t >0}$, it holds $u^0_f \in L^{2, \infty}(U \times (0,T))$ and
\begin{eqnarray*}
\|u^0_f \|_{L^{2, \infty}(U \times (0,T))} &=& \sup_{t \in (0,T)} \| T_t L_2 f\|_{L^2(U)} \\
&\leq&  ( \inf_{U} \rho \psi)^{-1/2}  \|L_2 f\|_{L^2(\R^d,\mu)}.
\end{eqnarray*}
Observe that
\begin{eqnarray*}
\iint_{U \times (0,T)} u_f \cdot \partial_t(\varphi_1 \varphi_2) dx dt &=& \int_0^T \big(\int_U T_t f\cdot   \varphi_1 dx \big) \partial_t\varphi_2 dt \\
&=& \int_0^T -\big( \int_U T_t L_2 f  \cdot \varphi_1 dx \big) \varphi_2 dt  \\
&=& -\iint_U u_f^0 \cdot\varphi_1 \varphi_2 dxdt.
\end{eqnarray*}
Using the approximation of Lemma \ref{stoneweier}, we obtain $\partial_t u_f = u^0_f \in L^{2, \infty}(U \times (0,T))$. \\
Now assume {\bf (A4$^\prime$)}. Then  by Theorem \ref{feokokoe}, $P_{t_0} f \in D(L_{q_0}) \subset H_{loc}^{2,2d+2}(\R^d)$ and for each $ 1 \leq i, j \leq d$, it holds
\begin{eqnarray}
&&\,\quad \| \partial_i \partial_j P_t f - \partial_i \partial_j P_{t_0}f \|_{L^{2d+2}(U)} \nonumber \\
&&\leq \|P_{t} f - P_{t_0} f \|_{H^{2,2d+2}(U)}  \nonumber \\
&&\leq \|T_t f- T_{t_0} f \|_{L^{q_0}(\R^d, \mu)}  + \|T_t L_{q_0} f - T_{t_0} L_{q_0} f \|_{L^{q_0}(\R^d, \mu)} \label{eorijgoijoi}
\end{eqnarray}
Define a map
$$
\partial_i \partial_j P_{\cdot} f: [0,T] \rightarrow L^2(U), \; \;  t \mapsto \partial_i  \partial_j P_{t} f.
$$
By the $L^{q_0}(\R^d, \mu)$-strong continuity of $(T_t)_{t>0}$ and \eqref{eorijgoijoi}, the map $\partial_i \partial_j P_{\cdot} f$ is continuous with respect to the $\| \cdot \|_{L^{2d+2}(U)}$-norm. Hence by \cite[Theorem, p91]{Nev65}, there exists a Borel measurable function $u^{ij}_f$ on $U \times (0,T)$ such that for each $t \in (0,T)$,  it holds
$$
u^{ij}_f(\cdot, t) = \partial_i  \partial_j P_t f.
$$
Using Theorem \ref{feokokoe} and the  $L^{q_0}(\R^d, \mu)$-contraction property of $(T_t)_{t>0}$, $u_f^{ij} \in L^{2d+2, \infty}(U \times (0,T))$ and
\begin{eqnarray*}
\|u_f^{ij}\|_{L^{2d+2, \infty}(U \times (0,T))}  &&\leq \sup_{t \in (0,T)} \|\partial_i \partial_j P_t f \|_{L^{2d+2}(U )} \nonumber \\
&&\leq  \sup_{t \in (0,T)}  \|P_t f\|_{H^{2,2d+2}(U)} \nonumber  \nonumber \\
&&\leq\sup_{t \in (0,T)}   C \left( \|T_t f\|_{L^{q_0}(\R^d, \mu)}+ \|T_t L_{q_0} f \|_{L^{q_0}(\R^d, \mu)} \right) \nonumber \\
&&\leq  C \|f\|_{D(L_{q_0})}, \label{foooekow}
\end{eqnarray*}
where $C>0$ is a constant which is independent of $f$.
Using the same line of arguments as in \eqref{pokpoin3} and the approximation as in Lemma \ref{stoneweier}, 
$$
\partial_i \partial_j u_f = u^{ij}_f \in L^{2d+2, \infty}(U \times (0,T))  .
$$
\end{proof}

\begin{theo} \label{fvodkoko}
Assume  {\bf (A4)$^\prime$} and  let $f \in C_0^{\infty}(\R^d)$. Then there exists
$$
u_f \in C_b\left(\R^d \times [0, \infty) \right) \cap \big( \displaystyle \bigcap_{r>0} W^{2,1}_{2d+2, \infty} (B_r \times (0,\infty)) \big) 
$$ 
satisfying $u_f(x, 0) = f(x)$ for all $x \in \R^d$ such that 
$$
\partial_t u_f \in  L^{\infty}(\R^d \times (0,\infty)), \; \partial_i u_f \in \displaystyle \bigcap_{r>0} L^{\infty}(B_{r} \times (0,\infty)) \; \text{ for all } 1\leq i \leq d,
$$
and
$$
\partial_t u_f = \frac12 \text{\rm trace}(\widehat{A} \nabla^2 u_f)+ \langle \mathbf{G}, \nabla u_f \rangle \, \; \text{ a.e. on } \R^d \times (0, \infty).
$$
\end{theo}
\begin{proof}
Let $f \in C_0^{\infty}(\R^d)$. Then $f \in D(L_s)$.  Define
$u_f:=P_{\cdot} f(\cdot)$. Then by  Lemma \ref{contiiokc}, $u_f \in C_b(\R^d \times [0, \infty))$ and $u_f(x, 0) = f(x)$ for all $x \in \R^d$. In particular, since $\mathbf{G} \in L^{\infty}_{loc}(\R^d, \R^d)$, it holds $f \in D(L_{q_0})$, so that $P_t f \in D(L_{q_0})$ for any $t \geq 0$. By Theorem \ref{pjpjwoeij}, for each $t>0$, it holds $\partial_t u_f(\cdot ,t) = T_t L_s f = T_t L f$\; $\mu$-a.e. on $\R^d$. Note that for each $t>0$, using  the sub-Markovian property,
\begin{eqnarray*}
\| \partial_t u_f(\cdot, t) \|_{L^{\infty}(\R^d)} &=& \|T_t L f \|_{L^{\infty}(\R^d)}\\
&\leq&  \|Lf\|_{L^{\infty}(\R^d, \mu)},
\end{eqnarray*}
hence $\partial_t u_f \in L^{\infty}(\R^d \times (0, \infty))$. By Theorem \ref{pjpjwoeij}, for $1 \leq i,j \leq d$, $t>0$, \;%$\partial_i u_f, \partial_i \partial_j u_f \in L^{2d+2, \infty}(B_R \times (0,T))$ and  
$\partial_i u_f(\cdot, t)  = \partial_i P_t f$,\, $\partial_i \partial_j u_f (\cdot, t)= \partial_i \partial_j P_t f$ \,$\mu$-a.e. on $\R^d$. Using Theorem \ref{feokokoe} and the $L^{q_0}(\R^d, \mu)$-contraction property of $(T_t)_{t>0}$, for any $R>0$ and for each $1 \leq i,j \leq d$, $t>0$, it holds
\begin{eqnarray*}
 \|\partial_i \partial_j u_f(\cdot,t ) \|_{L^{2d+2}(B_R )}&\leq& \|P_t f\|_{H^{2,2d+2}(B_R)} \nonumber \\
&\leq& C \left( \|T_t f\|_{L^{q_0}(\R^d, \mu)}+ \|T_t L_{q_0} f \|_{L^{q_0}(\R^d, \mu)} \right) \nonumber \\
&\leq& C \|f\|_{D(L_{q_0})}, \label{foooekowwev}
\end{eqnarray*}
where $C>0$ is as in Theorem \ref{feokokoe} and independent of $f$ and $t>0$. By Morrey's inequality, there exists a constant $C_{R,d}>0$, independent of $f$ and $t>0$, such that for each $t>0$, $1 \leq i \leq d$, 
\begin{eqnarray*}
\|\partial_i u_f(\cdot, t)\|_{L^{\infty}(B_R)}&\leq& \|\partial_i P_t f\|_{L^{\infty}(B_R)} \\
&\leq& C_{R,d} \|P_t f\|_{H^{2,2d+2}(B_R)}  \\
&\leq& C_{R,d} C \|f\|_{D(L_{q_0})}.
\end{eqnarray*}
Thus, $u_f \in  W^{2,1}_{2d+2, \infty} (B_R \times (0,\infty))$ and $\partial_t u_f, \partial_i u_f \in  L^{\infty}(B_{R} \times (0,\infty))$ for all $1 \leq i \leq d$.  By \eqref{maindivkjbie-3}, it holds 
\begin{eqnarray*}
&&\iint_{\R^d \times (0,\infty)} \langle \frac{1}{2} \rho A  \nabla u_f,  \nabla \varphi  \rangle-   \langle  \rho \psi \bold{B}, \nabla u_f  \rangle  \varphi  \;dx dt \nonumber  \\
&&\qquad  =\iint_{\R^d \times (0,\infty)} -\partial_t u_f \cdot \varphi  \rho \psi dx dt \quad \text{ for all } \varphi \in C_0^{\infty}(\R^d \times (0,\infty)). 
\end{eqnarray*}
Using integration by parts, we obtain
\begin{eqnarray*}
&&-\iint_{\R^d \times (0,\infty)} \big(\, \frac12 \text{trace}( \widehat{A} \nabla^2 u_f) +
\big \langle \beta^{\rho, A, \psi}+ \bold{B}, \nabla u_f \big  \rangle \big) \varphi  \, d\mu dt \nonumber  \\
&&\qquad  =\iint_{\R^d \times (0,\infty)} -\partial_t u_f \cdot \varphi \, d\mu dt \quad \text{ for all } \varphi \in C_0^{\infty}(\R^d \times (0,\infty)). 
\end{eqnarray*}
Therefore,
$$
\partial_t u_f = \frac12 \text{trace}(\widehat{A} \nabla^2 u_f)+ \langle \mathbf{G}, \nabla u_f \rangle \,\, \; \text{ a.e. on } \R^d \times (0, \infty).  \vspace{-0.5em}
$$
\end{proof}
\begin{defn}\label{wellposedness}
We say that {\bf uniqueness in law} holds for \eqref{fepojpow}, if 
for any two weak solutions 
$$\
\M=  
(\Omega,\F, (\F_t)_{t \ge 0}, (X_t)_{t \ge 0}, (W_t)_{t \ge 0}, (\P_x)_{x \in \R^d})
$$ 
and 
$$
\widetilde{\M} =  
(\widetilde{\Omega}, \widetilde{\F}, (\widetilde{\F}_t)_{t \ge 0}, (\widetilde{X}_t)_{t \ge 0}, (\widetilde{W}_t)_{t \ge 0}, (\widetilde{\P}_x)_{x \in \R^d})
$$ 
of  \eqref{fepojpow} it holds $\P_x\circ X^{-1}=\widetilde{\P}_x\circ \widetilde{X}^{-1}$ for all $x\in \R^d$. We say that the stochastic differential equation \eqref{fepojpow} is {\bf well-posed} if there exists a weak solution to it and uniqueness in law holds.
\end{defn}
\begin{theo} \label{weifjowej}
Assume {\bf (A4)$^\prime$}. Let $\M =  (\Omega, \F, (\F_t)_{t \ge 0}, (X_t)_{t \ge 0}, (W_t)_{t \ge 0}, (\P_x)_{x \in \R^d})$ and \\
$\widetilde{\M}= (\widetilde{\Omega}, \widetilde{\F}, (\widetilde{\F}_t)_{t \ge 0}, (\widetilde{X}_t)_{t \ge 0}, (\widetilde{W}_t)_{t \ge 0}, (\widetilde{\P}_x)_{x \in \R^d})$ be any two weak solutions to \eqref{fepojpow}. Suppose 
\begin{eqnarray}\label{wellposed setssame}
\P_x(\Lambda(Z^{\M}) ) = \widetilde{\P}_{x}(\Lambda(Z^{\widetilde{\M}})) =1,\quad  \text{ for all } x \in \R^d.
\end{eqnarray} 
Then $\P_x\circ X^{-1}=\widetilde{\P}_x\circ \widetilde{X}^{-1}$  for all $x \in \R^d$. In particular, under  {\bf (A4)$^\prime$} any weak solution to \eqref{fepojpow} is a strong Markov process.
%Moreover, $\M$ and $\widetilde{\M}$ satisfy the strong Markov property.
\end{theo}
\begin{proof} Let $x\in \R^d$ be arbitrary. Let $\mathbb{Q}_x=\P_x\circ X^{-1}$ and $\mathbb{\widetilde{Q}}_x={\widetilde{\P}}_x\circ \widetilde{X}^{-1}$ respectively. Then $\mathbb{Q}_x$, $\widetilde{\mathbb{Q}}_x$ are two solutions of the time-homogeneous martingale problem with initial condition $x$ and coefficients $(\widehat{\sigma},\mathbf{G})$ as defined in \cite[Chapter 5, 4.15 Definition]{KaSh}. Let $f \in C_0^{\infty}(\R^d)$. For  $T>0$, define $g(x,t):=u_f(x, T-t)$,  $(x,t) \in \R^d \times [0,T]$, where $u_f$ is defined as in Theorem \ref{fvodkoko}.
Then by Theorem \ref{fvodkoko},
\begin{eqnarray*}
&&g \in C_b\left(\R^d \times [0, T] \right) \cap \big( \displaystyle \bigcap_{r>0} W^{2,1}_{2d+2, \infty} (B_r \times (0,T)) \big), \\
&&\partial_t g \in  L^{\infty}(\R^d \times (0,T)),\; \,  \; \partial_i g \in \displaystyle \bigcap_{r>0} L^{\infty}(B_{r} \times (0,T)),  \, 1\leq i \leq d,
\end{eqnarray*}
and it holds
\begin{eqnarray*}
\frac{\partial g}{\partial t} + Lg = 0 \;\;\,  \text{ a.e. in } \R^d \times (0,T), \;\; \; \;g(x, T) = f(x) \;\; \text{for all }  x \in \R^d.
\end{eqnarray*}
%Let $\E_x$,  $\widetilde{\E}_x$ be the expectations w.r.t. $\P_x$, $\widetilde{\P}_x$, respectively. Let $D_R:=\inf \{ t \geq 0 \mid X_t \in \R^d \setminus B_R  \}$. 
Applying Theorem \ref{krylovtype} to $\M$, for  $x \in \R^d$, $R>0$, it holds 
$$
\E_x \left[ \int_0^{T \wedge D_R}  \Big \vert \frac{\partial g}{\partial t} + Lg \Big \vert (X_s,s) ds   \right] =0, \; \; \quad
$$
hence
$$
\int_0^{T \wedge D_R}  \big( \frac{\partial g}{\partial t} + Lg  \big)(X_s,s) ds  = 0, \;\;  \P_x\text{-a.s., }
$$
hence by Theorem \ref{fefere},
$$
g(X_{T \wedge D_R}, T\wedge D_R) - g(x,0) = \int_0^{T \wedge D_R} \nabla g(X_s,s) \widehat{\sigma}(X_s) dW_s, \; \;\; \; \P_{x}\text{-a.s. }
$$
Therefore
$$
\E_x \left[ g(X_{T\wedge D_R},T \wedge D_R)  \right]  = g(x,0).
$$
Letting $R \rightarrow \infty$ and using Lebesgue's Theorem, we obtain 
$$
\E_x[f(X_T)] =\E_x[g(X_T, T)]  = g(x,0).
$$
Analogously for $\widetilde{\M}$, we obtain $ \widetilde{\E}_x[f(\widetilde{X}_T)] = g(x,0)$. Thus
$$
\E_x[f(X_T)] =\widetilde{\E}_x[f(\widetilde{X}_T)].
$$ 
Therefore $\mathbb{Q}_x$, and $\widetilde{\mathbb{Q}}_x$ have the same one-dimensional marginal distributions and we can conclude as in \cite[Chapter 5, proof of 4.27 Proposition]{KaSh} that $\mathbb{Q}_x = \widetilde{\mathbb{Q}}_x$.\\
For the last statement, see \cite[Chapter 5, 4.20 Theorem]{KaSh}.
\end{proof}
\\
Combining Theorem \ref{weifjowej}, Remark \ref{mostresults4} and Theorem \ref{weakexistence4}, we obtain the following result.
\begin{theo} \label{poijoijweuniqe}
Assume  {\bf (A4)$^\prime$} and suppose that $\M$ of Theorem \ref{weakexistence4} is non-explosive.
(This is for instance the case if \eqref{consuniquelaw} holds, see Theorem \ref{weakexistence4}.)
%there exists a constant  $M> 0$ such that 
%\begin{eqnarray}\label{consuniquelaw}
%-\frac{\langle \widehat{A}(x)x, x \rangle}{ \left \| x \right \|^2 +1}+ \frac12\mathrm{trace}\widehat{A}(x)+ \big \langle \mathbf{G}(x), x \big \rangle  \leq M\left ( \left \| x \right \|^2+1\right )\left (  {\rm ln}(\left \| x \right \|^2+1)+1\right )  
%\end{eqnarray}
%for a.e. $x\in \R^d$. 
Then the Hunt process $\M$ is the  unique solution (in law) to \eqref{fepojpow} that satisfies $\P_x(\Lambda(Z^{\M}) ) =1$, for all  $x \in \R^d$. Moreover, under the same conditions as in \cite[Theorem 4.4]{LT18} but replacing $A,\sigma$ there with $\frac{1}{\psi}A$, $\sqrt{\frac{1}{\psi}}\sigma$ respectively, the moment inequalities of the mentioned theorem also hold for our $\M$ here.
\end{theo}
\begin{rem}\label{remuniquenessinlawconsequences}
Once uniqueness in law holds for \eqref{fepojpow}, any weak solution to \eqref{fepojpow} satisfies the improved (time homogeneous) Krylov type estimate \eqref{sfpolkd}. We will illustrate this in two with respect to each other extreme cases. 
For the first case suppose that {\bf (A4)$^\prime$} holds with $q=2d+2+\varepsilon$ for some small $\varepsilon>0$. Then we may choose $s=\frac{2}{3}d$ and $s_0:=\frac{sq}{q-1}=\frac{2}{3}d\cdot \frac{2d+2+\varepsilon}{2d+1+\varepsilon}$ satisfies $s_0<\frac{4}{5}d$, actually $s_0=\frac{4}{5}d-\delta$ for small $\delta>0$ and for any bounded open set $V$, any ball $B\subset \R^d$, and $g \in L^{s_0}(\R^d)_0$ with $\text{supp}(g) \subset V$, we have by \eqref{sfpolkd} for any $x\in \overline{B}$
\begin{eqnarray*}
\mathbb{E}_x\left[ \int_0^T g(X_s) ds\right] &\leq& C_{B,s,t}\,\|g\|_{L^s(\R^d, \mu)}  \\
&\leq& C_{B,s,t}\,\|\rho \psi \|^{1/s}_{L^q(V)}\left( \int_{V} |g|^{\frac{sq}{q-1}} dx   \right)^{\frac{q-1}{sq}}  = C_{B,s,t}\,  \|\rho \psi \|^{1/s}_{L^q(V)}\|g\|_{L^{s_0}{(V)}}.
\end{eqnarray*}
On the other hand, if {\bf (A4)$^\prime$} holds with $q=\infty$ and $\frac{1}{\psi}$ is supposed to be locally pointwise bounded below and above by strictly  positive constants, we may choose $s=\frac{d}{2}+\varepsilon$ for arbitrarily small $\varepsilon>0$ and we obtain for $g\in  L^{s}(\R^d)_0$ with $\text{supp}(g) \subset V$, $V,B$ and $x$ as above,
$$
\mathbb{E}_x\left[ \int_0^T g(X_s) ds\right] \leq C_{B,s,t}\, \|\rho\psi\|_{L^\infty(V)}^{1/s}\|g\|_{L^s(V)}.
$$
\end{rem}
\begin{exam}\label{wellposednessexplicitexample}
Consider the situation in Example \ref{imporexam} except the conditions (a), (b), (c). Let $p:=2d+2$ and assume $\mathbf{G} \in L^{\infty}_{loc}(\R^d, \R^d)$. Let $\alpha \geq 0$ be such that $\alpha (2d+2)< d$. Take $q \in (2d+2, \frac{d}{\alpha})$. Then $A$, $\mathbf{G}$, $\psi$ satisfy {\bf (A4)$^\prime$}.  Therefore, the Hunt process $\M$ of Theorem \ref{weakexistence4}
solves weakly \,$\P_x$-a.s. for any $x \in \R^d$,
\begin{equation}\label{weaksolutiondcuq}
X_t= x+\int_0^t \|X_s\|^{\alpha/2}  \cdot \frac{\sigma}{\sqrt{\phi}} (X_s) \, dW_s +   \int^{t}_{0}   \mathbf{G}(X_s) \, ds, \quad 0\le  t <\zeta. \quad 
\end{equation}
Assume \eqref{consuniquelaw}. Then $\zeta=\infty$ and by Theorem \ref{poijoijweuniqe}, $\M$ is the unique (in law) solution to \eqref{weaksolutiondcuq} that satisfies $\P_x(\Lambda(Z^{\M}) ) =1$, for all  $x \in \R^d$. If we choose the following Borel measurable version of $\|x\|^{\alpha/2}$, namely
$$
f_{\gamma}(x):=\|x\|^{\alpha/2}1_{\{x\not=0\}}(x)+\gamma 1_{\{x=0\}}(x), \quad x\in \R^d
$$
where $\gamma$ is an arbitrary but fixed strictly positive real number, then $\P_x(\Lambda(Z^{\widetilde{\M}}) ) =1$ (here of course $Z^{\widetilde{\M}}$ is defined w.r.t. $\sqrt{\frac{1}{\psi}}=\frac{f_{\gamma} }{\sqrt{\phi}}$) is automatically satisfied by Lemma \ref{condforpsitohold}(i) for any weak solution $\widetilde{\M}$ to 

\begin{equation}\label{exampleEngSchm}
\widetilde{X}_t= x+\int_0^t \frac{f_{\gamma}\cdot \sigma }{\sqrt{\phi}}(\widetilde{X}_s) \, d\widetilde{W}_s +   \int^{t}_{0}   \mathbf{G}(\widetilde{X}_s) \, ds, \quad t\ge 0, \ x\in\R^d, \quad 
\end{equation}
Thus under \eqref{consuniquelaw}, the SDE  \eqref{exampleEngSchm}
is well-posed for any $\gamma>0$ and moreover $\M$ of Theorem \ref{weakexistence4} also solves \eqref{exampleEngSchm}.
\end{exam}
{\bf Acknowledgement.} We would like to thank our colleagues Khaled Bahlali for very useful suggestions about uniqueness in law and Francesco Russo for a valuable discussion about the subject.

\centerline{}
Haesung Lee, Gerald Trutnau\\
Department of Mathematical Sciences and \\
Research Institute of Mathematics of Seoul National University,\\
1 Gwanak-Ro, Gwanak-Gu,
Seoul 08826, South Korea,  \\
E-mail: fthslt@snu.ac.kr, trutnau@snu.ac.kr
\end{document}